\documentclass[]{amsart}

\usepackage{amsfonts}
\usepackage{amscd}
\usepackage{amsmath, mathrsfs, amssymb, mathtools}
\usepackage{amsthm}
\usepackage{setspace}
\usepackage{hyperref}
\usepackage{color}
\usepackage{epsfig}
\usepackage{here}
\usepackage{graphicx}
\usepackage[all]{xy}
\usepackage{psfrag}
\usepackage{graphicx,transparent}
\usepackage{enumerate}
\usepackage{caption}

\theoremstyle{plain}
\newtheorem{theorem}{Theorem}[section]
\newtheorem{lemma}[theorem]{Lemma}
\newtheorem{proposition}[theorem]{Proposition}
\newtheorem{corollary}[theorem]{Corollary}

\theoremstyle{definition}
\newtheorem{remark}[theorem]{Remark}
\newtheorem{definition}[theorem]{Definition}
\newtheorem{example}[theorem]{Example}

\newcommand{\calM}{\mathcal M}

\newcommand{\BM}{\overline{\mathcal M}}

\newcommand{\calC}{\mathcal C}

\newcommand{\calO}{\mathcal O}

\newcommand{\calH}{\mathcal H}

\newcommand{\bbA}{\mathbb A}

\newcommand{\calL}{\mathcal L}

\newcommand{\calF}{\mathcal F}

\newcommand{\hyp}{\operatorname{hyp}}

\newcommand{\even}{\operatorname{even}}
\newcommand{\odd}{\operatorname{odd}}

\newcommand{\nonhyp}{\operatorname{nonhyp}}

\newcommand{\ord}{\operatorname{ord}}

\newcommand{\area}{\operatorname{area}}

\newcommand{\bbC}{\mathbb C}

\newcommand{\bbN}{\mathbb N}

\newcommand{\bbP}{\mathbb P}

\newcommand{\bbQ}{\mathbb Q}

\newcommand{\bbZ}{\mathbb Z}

\newcommand{\bbG}{\mathbb G}

\newcommand{\lcm}{\operatorname{lcm}}

\newcommand{\Proj}{\operatorname{Proj}}
\newcommand{\Spec}{\operatorname{Spec}}

\newcommand{\fkm}{\mathfrak{m}}
\newcommand{\fkc}{\mathfrak{c}}
\newcommand{\wfkm}{\widetilde{\mathfrak{m}}}
\newcommand{\wR}{\widetilde{R}}

\begin{document}

\title[]{Gorenstein singularities with $\bbG_m$-action and moduli spaces of holomorphic differentials}

\date{\today}

\author{Dawei Chen}
\address{Department of Mathematics, Boston College, Chestnut Hill, MA 02467, USA}
\email{dawei.chen@bc.edu}

\author{Fei Yu}
\address{School of Mathematical Sciences, Zhejiang University, Hangzhou, China}
\email{yufei@zju.edu.cn}


\thanks{Research of D.C. was supported by the National Science Foundation under Grant DMS-2301030, Simons Travel Support for Mathematicians, and a Simons Fellowship under Record ID SFI-MPS-SFM-00005694.}
\thanks{Research of F.Y. was supported by Fundamental Research Funds for the Central Universities 2024FZZX02-01-01.} 

\begin{abstract}
Given a holomorphic differential on a smooth complex algebraic curve, we associate to it a Gorenstein curve singularity with $\bbG_m$-action via a test configuration. This construction decomposes the strata of holomorphic differentials with prescribed orders of zeros into negatively graded miniversal deformation spaces of such singularities. Additionally, it provides a natural description for the singular curves that appear in the boundary of the miniversal deformation spaces. 

Our approach leads to a number of applications. We classify the unique Gorenstein singularity with $\bbG_m$-action for each nonvarying stratum of holomorphic differentials and study when these nonvarying strata can be compactified by weighted projective spaces. Moreover, extending the classical results about $ADE$ singularities, we establish the $K(\pi,1)$-property for non-hypersurface complete intersection 
singularities of type $U_7$, $U_8$, $U_9$, and $S_{k}$. We also study  singularities with bounded $\alpha$-invariants in the log minimal model program for $\overline{\mathcal M}_g$ and utilize them to bound the slopes of effective divisors in $\overline{\mathcal M}_g$. Finally, we show that the loci of subcanonical points with fixed semigroups have trivial tautological rings and provide a criterion to determine whether they are affine varieties. 
\end{abstract}

\maketitle
     
\setcounter{tocdepth}{1}
\tableofcontents

\section{Introduction}
\label{sec:intro}

For $g, n\geq 1$, let $\mu = (m_1, \ldots, m_n)$ be a partition of $2g-2$, where $m_i \in \bbN$ and $\sum_{i=1}^n m_i = 2g-2$. Denote by $\bbP \calH (\mu)$ the moduli space parameterizing smooth and connected complex algebraic curves $C$ with distinct marked points $p_1, \ldots, p_n$ such that $\sum_{i=1}^n m_i p_i$ is a canonical divisor. Equivalently, $\bbP \calH (\mu)$ parameterizes holomorphic differentials $\omega \in H^0(C, K)$ (up to scaling) of zero type specified by $\mu$, where $K$ is the canonical bundle. 

Since the union of $\bbP \calH (\mu)$ for all partitions $\mu$ of $2g-2$ stratifies the (projectivized) Hodge bundle of holomorphic differentials, $\bbP \calH (\mu)$ is called the (projectivized) {\em stratum of holomorphic differentials of type $\mu$}. For special $\mu$, the stratum $\bbP \calH (\mu)$ can have up to three connected components, due to the {\em spin parity} and {\em hyperelliptic structure}; see \cite{KZ03}. We use $\bbP \calH (\mu)^{\ast}$ for $\ast\in \{\odd, \even, \hyp, \nonhyp\}$ to denote these connected components of the strata. As a convention, a ``stratum'' in this paper will refer to a {\em connected component} if the corresponding $\bbP \calH (\mu)$ is disconnected. When $\mu$ contains many entries of the same value, we also use the exponential notation; e.g., $\mu = (1^{2g-2})$ is the signature of differentials with $2g-2$ simple zeros. 

The study of differentials with prescribed orders of zeros has played an important role in understanding the geometry of moduli spaces and dynamics on flat surfaces. The various aspects exhibited by differentials have inspired fruitful results and revealed profound connections in the fields of billiards in polygons, cycle class computations, enumerative geometry, modular compactifications, and Teichm\"uller dynamics. We refer to \cite{Z06, W15, C17, F24} for some introductions to this fascinating subject. 

In this paper, we focus on an interesting new interplay between the strata of differentials and deformations of singularities. We summarize our construction and main results as follows. 

\subsection{Gorenstein singularities with $\bbG_m$-action and test configurations}

Our setup goes as follows. Given a canonical divisor of type $\mu$ supported on the smooth pointed curve $(C, p_1, \ldots, p_n)$, we take a trivial family $C\times \bbA^1$, perform a weighted blowup at the zeros $p_i$ of the central fiber, where the weights of the blowup depend on $\mu$, and finally contract the proper transform of the central fiber to form a singular curve $Y$ with an isolated singularity $q$. The resulting family provides a {\em test configuration} for $(Y, q)$, whose isomorphism class is conversely determined by the associated Rees filtration and graded algebra. The details of the construction are provided in Section~\ref{sec:test}. 

This simple construction possesses a number of remarkable properties, which we summarize as follows; see Theorem~\ref{thm:filtration} and Theorem~\ref{thm:boundary} for detailed elaborations.  

\begin{theorem}
\label{thm:main} 
In the above construction, $Y$ admits an isolated Gorenstein singularity at $q$ with $\bbG_m$-action. Canonical divisors in $\bbP\calH (\mu)$ that produce the same isomorphism class of $(Y, q)$ form an open subset $\bbP{\rm Def}^{-}_s(Y, q)$ of smooth deformations in the negatively graded versal deformation space $\bbP{\rm Def}^{-}(Y, q)$, where the union of these $\bbP{\rm Def}^{-}_s(Y, q)$ provides a decomposition for $\bbP\calH(\mu)$. 

Moreover, the boundaries of these $\bbP{\rm Def}^{-}(Y, q)\setminus \bbP{\rm Def}^{-}_s(Y, q)$ parameterize stable Gorenstein singular curves that admit a canonical divisor of type $\mu$. 

Finally, the weights and characters of the $\bbG_m$-action on the space of pluricanonical differential forms of $Y$ can be determined explicitly from the associated filtration of the test configuration. 
\end{theorem}

We remark that the special case of $n = 1$ was studied in detail by Pinkham in \cite{P74} (for unibranched singularities with $\bbG_m$-action which are not necessarily Gorenstein). In this case, $(Y, q)$ is a monomial singularity and $\bbP{\rm Def}^{-}_s(Y, q)$ can be identified with the locus  $\calM_{g,1}^H\subset  \calM_{g,1}$ parameterizing smooth pointed curves $(C, p)$, where the Weierstrass semigroup $H$ of $p$ is generated by the exponents of the monomials defining $(Y, q)$. Therefore, our construction can be regarded as a refined generalization of Pinkham's result from the case of unibranched singularities to Gorenstein singularities with multiple branches; see also~\cite[Appendix]{L84} for a formal description of such constructions with modular interpretations. 

We also remark that the relation between holomorphic differentials and Gorenstein singularities (not necessarily with $\bbG_m$-action) was speculated by Battistella; see \cite[Section 1.4]{B24}. In particular, \cite[Section 2]{B24} classified the Gorenstein curve singularities in genus three; see also \cite[Appendix A]{S11} and \cite[Section 2]{B22} for the classifications in genus one and two, respectively. Therefore, our work provides a uniform approach towards further classifications of Gorenstein singularities in higher genera with the help of $\bbG_m$-action; see below.  

\subsection{Uniqueness of singularities for the nonvarying strata}

A natural question regarding the classifications of singularities is when a stratum of differentials corresponds to a {\em unique} Gorenstein singularity $(Y,q)$ with $\bbG_m$-action. This turns out to be closely related to the so-called {\em nonvarying strata}, whose original definition is given by all Teichm\"uller curves within the stratum possessing the same numerical behavior (in terms of their slopes, sums of nonnegative Lyapunov exponents, and area Siegel--Veech constants). On the one hand, the hyperelliptic strata and a number of low genus strata are known to be nonvarying; see \cite{CM12, YZ13}. On the other hand, all other strata are expected to be varying based on numerical evidences of computer check. The nonvarying property is also related to strata being affine varieties with trivial tautological rings; see \cite[Theorem 1.1]{C24}. 

Our next result reveals another hidden aspect of this nonvarying phenomenon.  

\begin{theorem}
\label{thm:nonvarying}
There is a unique isomorphism class of Gorenstein singularities with $\bbG_m$-action for every nonvarying stratum $\bbP\calH(\mu)$ of holomorphic differentials of type $\mu$ in the following list: $\mu = (4)^{\odd}$, $(3,1)$, $(2,2)^{\odd}$, $(2,1,1)$, $(6)^{\odd}$, $(6)^{\even}$, $(5,1)$, $(4,2)^{\even}$, $(4,2)^{\odd}$, $(3,3)^{\nonhyp}$, $(3,2,1)$, 
$(2,2,2)^{\odd}$, $(6,2)^{\odd}$, $(5,3)$, as well as the hyperelliptic strata $(2g-2)^{\hyp}$ and $(g-1,g-1)^{\hyp}$, where each of these nonvarying strata can be identified with the locus of smooth deformations in the miniversal deformation space of the corresponding singularity. 

Moreover, except possibly $(2,2,2)^{\odd}$ and $(6,2)^{\odd}$, each of the remaining nonvarying strata is a hypersurface complement in a weighted projective space that can be described explicitly, and hence it is a rational variety with trivial rational Chow group in any positive degree.

Finally, there is a unique isomorphism class of Gorenstein singularities with $\bbG_m$-action associated to every locus of hyperelliptic differentials with signature 
$$\mu = (2a_1, \ldots, 2a_m, b_1, b_1, \ldots, b_n, b_n),$$ where each zero of order $2a_i$ is a Weierstrass point and each pair of zeros of order $b_j$ is hyperelliptic conjugate. 
\end{theorem} 

Indeed, for the above nonvarying strata in low genus, 
we can describe the defining equations for the corresponding  singularities as well as their weighted projective embeddings; see Section~\ref{sec:nonvarying}. The two possible exceptional cases in the second part of Theorem~\ref{thm:nonvarying} are due to the fact that the corresponding singularities have nonzero infinitesimal obstructions; see Remark~\ref{rem:T2}. For the remaining cases, with the defining equations of the singularities at hand, we can also describe explicitly the weighted projective spaces that compactify the corresponding strata; see Example~\ref{ex:wps} for a sample description. 

\subsection{Adding ordinary marked points and the $K(\pi,1)$-property} 
\label{subsec:Kpi1}

Recall that a space is $K(\pi, 1)$ if its (orbifold) universal cover is contractible;  equivalently, if its homotopy groups of degree higher than one all vanish. The geometry of $\bbP \calH (\mu)$ and its decomposition into deformation spaces $\bbP{\rm Def}^{-}_s(Y, q)$ can motivate a number of intriguing $K(\pi, 1)$ and affine patching related questions; see \cite{K97, LM14, M17, G24, G25, Q25} for some related conjectures and results. 

Note that our result has identified each of the nonvarying strata with the locus of smooth deformations of the corresponding singularity, which can be regarded as the discriminant hypersurface complement in the miniversal deformation space. A conjecture (sometimes attributed to Saito; see \cite[p. 26]{Ar93} and \cite[p. 185]{L84Book}) says that the discriminant complement in the miniversal deformation space of a positive-dimensional isolated complete intersection singularity is $K(\pi, 1)$. The cases for $A$ and $DE$ singularities were established by Brieskorn and Deligne, respectively; see~\cite{B73, D72} as well as \cite[Chapter 9]{L84Book} for a more detailed discussion. 

Although the original definition of nonvarying strata was stated for $\mu$ with only positive entries, we can also add entries of $0$ into $\mu$ which amounts to adding ordinary marked points away from the original zeros. From the perspective of associated singularities, the former and latter determine each other; see Section~\ref{subsec:ordinary} for a detailed description. This observation has a striking application for the $K(\pi,1)$-property of certain {\em non-hypersurface} complete intersection singularities as follows.

\begin{theorem}
\label{thm:US}
Let $(Y', q')$ be a singularity of type $(\mu, 0)$ obtained by adding an ordinary marked point away from the zeros in the construction of a singularity $(Y,q)$ of type $\mu$. If $\bbP{\rm Def}^{-}_s(Y,q)$ is $K(\pi,1)$, then $\bbP{\rm Def}^{-}_s(Y',q')$ is also $K(\pi,1)$. 

In particular, the $K(\pi,1)$-property holds for the discriminant complement in the miniversal deformation space of the singularities 
$U_7$, $U_8$, $U_9$, $S_{2g+2}$, and $S_{2g+3}$, which arise from the nonvarying strata for $\mu' = 
(4,0)^{\odd}$, $(3,1,0)$, $(6,0)^{\even}$, $(2g-2,0,0)^{\hyp}$, and $(g-1,g-1,0,0)^{\hyp}$, respectively. 
\end{theorem}

 We refer to Theorem~\ref{thm:ordinary} for a detailed description of the above singularities. See also Section~\ref{subsec:hyp-nonvarying} for a related discussion about the $K(\pi,1)$-property of singularities associated to hyperelliptic differentials. Additionally, some of the nonvarying strata in our classification correspond to simple space curve singularities of type $T_7$ (for $\mu = (2,2)^{\odd}$), 
 $T_8$ (for $\mu = (2,1,1)$), 
 $T_9$ (for $\mu = (4,2)^{\even}$), 
 $W_8$ (for $\mu = (6)^{\odd}$), 
 $W_9$ (for $\mu = (5,1)$), 
 $Z_9$ (for $\mu = (3,3)^{\nonhyp}$), 
 and $Z_{10}$ (for $\mu = (8)^{\even}$); 
 see also \cite[(7.22)]{L84Book} and \cite[Table 2]{S15}. Therefore, the $K(\pi,1)$-conjecture for them is thus equivalent to that of the corresponding strata of differentials. 

\subsection{Weights and characters of singularities}

As mentioned in Theorem~\ref{thm:main}, the {\em weights} and {\em characters} of the $\bbG_m$-action on the space of differentials of $Y$ can be determined explicitly through our construction. These important invariants can help us better understand the Hassett--Keel log minimal model program for $\BM_{g}$, as studied extensively by Alper--Fedorchuk--Smyth in \cite{AFS16}.  By utilizing our construction, we can systematically compute these invariants, paving an avenue towards a further classification of related singularities. To this end, we provide various examples and demonstrate the computations in Section~\ref{sec:ex}. Moreover, in Section~\ref{sec:numerical}, we establish the following result. 

\begin{theorem}
\label{thm:char} 
Let $(Y, q)$ be a smoothable Gorenstein curve singularity with $\bbG_m$-action. Then the weights of the $\bbG_m$-action on the space of differential one-forms of $Y$ can determine those of higher-order differentials of $Y$. 
\end{theorem}

In particular, we provide an explicit relation between the associated characters; see Proposition~\ref{prop:weight} for more details. 

\subsection{Singularities with bounded $\alpha$-invariants}

If a new singularity first appears in the log minimal model program for $\BM_{g}$ with respect to the divisor class $K + \alpha \delta$, where $K$ is the canonical class and 
$\delta$ is the total boundary divisor class of $\BM_g$, then $\alpha$ is called the {\em $\alpha$-invariant} of the singularity. There is a critical value $\alpha = 3/8$ that plays a significant role in this process for large genus, as explained in \cite[Section 3.4]{AFS16}. Roughly speaking, the polarization on the GIT quotient of the Chow variety of canonical curves is asymptotically proportional to $K + \frac{3}{8}\delta$. Therefore, it is an interesting question to classify Gorenstein singularities with $\alpha$-invariants above $3/8$, as asked in \cite[Problem 3.15]{AFS16}. As another application of our method, we provide the following classification. 

\begin{theorem}
\label{thm:3/8}
A smoothable isolated Gorenstein curve singularity with $\bbG_m$-action has $\alpha$-invariant $\geq 3/8$ (without dangling branches) if and only if the corresponding locus of holomorphic differentials is one of the following cases: 
\begin{enumerate}[{\rm (i)}]
\item $\mu = (0^m)$ (i.e., elliptic $m$-fold points) for $m \leq 4$. 
\item $\mu = (a, a, b, b)^{\hyp}$, $(2a, b, b)^{\hyp}$, $(2a, 2b)^{\hyp}$, $(g-1, g-1)^{\hyp}$, $(2g-2)^{\hyp}$, $(2g-2, 0)^{\hyp}$, and $(2g-2, 0, 0)^{\hyp}$, where the underlying curves are hyperelliptic, the zeros of even order are Weierstrass points, the zeros of equal order form hyperelliptic conjugacy pairs, and the ordinary marked points (zeros of order $0$) can be any points. 
\item $\mu = (4)^{\odd}$, $(4,0)^{\odd}$, $(3,1)$, $(3,1,0)$, $(2,2)^{\odd}$, $(2,1,1)$, $(1,1,1,1)$, $(6)^{\odd}$, $(6)^{\even}$, $(6,0)^{\even}$, $(5,1)$, $(4,2)^{\even}$, $(3,3)^{\nonhyp}$, $(2,2,2)^{\even}$. 
\item The locus of $h^0(C, 3p_1) = 2$ in the stratum with $\mu = (7,1)$. 
\item The locus of $h^0(C, 2p_1+p_2) = 2$ in the stratum with $\mu = (7,3)$.  
\item The locus of Weierstrass semigroup generated by $3$ and $7$ in the stratum with $\mu = (10)^{\even}$.  
\end{enumerate}
\end{theorem}

We are also able to describe the defining equations for all the singularities that appear in Theorem~\ref{thm:3/8}. Some of these singularities are well-studied in the literature, while it seems the others can be difficult to figure out without invoking our perspective. Just for the reader to get a feel, for example, the defining equations for the unique singularity corresponding to the nonvarying stratum $\bbP\calH(6,2)^{\odd}$ are given by 
$$ (xz, xw, xu, zu - w^2, yw - u^2, yz - wu, u^2 - z^3, yu - z^2w, y^2 - x^3 - zw^2). $$

Additionally, the weights, characters, and $\alpha$-invariants for the case with dangling branches can be fully determined by those without dangling branches; see~\cite[Corollary 3.3 and Corollary 3.6]{AFS16} and also Section~\ref{subsec:alpha} for a detailed definition and relation between the two cases. Indeed, if we consider a higher critical value $\alpha = 5/9$, then the proof of Theorem~\ref{thm:3/8} can be significantly simplified, for both cases of with or without dangling branches, thus verifying the predictions of such singularities with $\alpha \geq 5/9$ in \cite[Table 3]{AFS16}. That said, in this paper we focus on the $\alpha$-invariants for the case without dangling branches so as to treat all zeros of holomorphic differentials uniformly. 

\subsection{Slopes of singularities and effective divisors in $\BM_g$}

For a one-parameter family $B$ of stable curves of genus $g$, we define its {\em slope} to be 
$$s(B) \coloneqq \frac{ \deg_B \delta}{\deg_B \lambda_1},$$ 
where 
$\delta$ is the total boundary divisor class of $\BM_g$ and $\lambda_1$ is the first Chern class of the Hodge bundle. In particular, if $B$ represents the class of a {\em moving curve} in $\BM_g$, then the slope of $B$ can provide a lower bound for the slopes of effective divisors in $\BM_g$ through their intersection numbers. This is due to the fact that an effective divisor $D$ cannot contain the union of such moving curves, and hence $D$ and $B$ must have a nonnegative intersection number. 

Determining the lower bound for the slopes of effective divisors in $\BM_g$ is an important, long-standing problem; see~\cite{CFM13} for a survey discussion. In our context, motivated by~\cite[Theorem 4.2]{AFS16}, we can similarly define the slope for singularity classes $(Y, q)$ that are associated to $\bbP\calH(\mu)$. Moreover, if $\bbP\calH(\mu)$ dominates $\calM_g$ and the discriminant locus of singular deformations in ${\rm Def}^{-}(Y, q)$ consists of generically nodal curves, then we can use the slope of $(Y,q)$ to bound the slopes of effective divisors in $\BM_g$. 

\begin{theorem}
\label{thm:slope-main}
Let $\mu$ be a partition of $2g-2$ such that $\mu$ has at least $g-1$ positive entries and $\mu \neq (2^{g-1})^{\even}$. Let $(Y,q)$ be a singularity class associated to a generic canonical divisor of type $\mu$. If the discriminant locus of singular deformations in ${\rm Def}^{-}(Y, q)$ is a $\bbQ$-Cartier divisor parameterizing generically nodal curves, then the slope of $(Y, q)$ is a lower bound for the slopes of effective divisors in $\BM_g$.  
\end{theorem}

In Section~\ref{sec:slope}, we provide a more detailed statement of the above result in Theorem~\ref{thm:slope}. We will also compute the slopes of singularities for a number of strata; see Example~\ref{ex:nonvarying}, Example~\ref{ex:BN}, Example~\ref{ex:spin-odd}, and Example~\ref{ex:weierstrass}. 

\subsection{Affine and tautological geometry of subcanonical points}

Finally, we turn to the affine and tautological geometry of moduli spaces of differentials. It was shown in \cite{C19} that the strata of (strictly) meromorphic canonical divisors with prescribed zero and pole orders are affine varieties by using relations of tautological divisor classes in the strata. However, the same question remains open for the holomorphic case. Therefore, it is meaningful to study the affine and tautological geometry for each $\bbP{\rm Def}^{-}_s(Y, q)$ in the decomposition of the strata $\bbP \calH (\mu)$ of holomorphic differentials, and the results can help us better understand the global geometry of $\bbP \calH (\mu)$. 

Here we concentrate on the case of $(Y, q)$ being a monomial singularity, where $\bbP{\rm Def}^{-}_s(Y, q)$ can be identified with the locus $\calM_{g,1}^H$ of Weierstrass points $(C, p)$ with a fixed semigroup $H$. In particular, $2g-1\not\in H$ is equivalent to that $(2g-2)p$ is a canonical divisor, i.e., $(C, p)\in \bbP\calH(2g-2)$. Such a semigroup $H$ is called {\em symmetric} and such a point $p$ is called a {\em subcanonical point}. 

\begin{theorem}
\label{thm:Weierstrass}
Let $H$ be a semigroup in genus $g$ such that $2g-1\not\in H$. Then the tautological ring of $\calM_{g,1}^H$ is trivial in all positive degrees. Moreover, an irreducible component $\calM_{g,1}^{H^{\circ}}$ of $\calM_{g,1}^H$ is an affine variety if and only if the boundary of $\calM_{g,1}^{H^{\circ}}$ is pure codimension $1$. 
\end{theorem}

In Section~\ref{sec:Weierstrass}, we provide more detailed descriptions for the above results in Theorem~\ref{thm:(2g-2)-tauto} and Theorem~\ref{thm:(2g-2)-affine}. Additionally, we apply the results to study the varying minimal spin strata in genus six; see Example~\ref{ex:(10)-odd} and Example~\ref{ex:(10)-even}. 

\subsection{Further directions} 
The explicit defining equations of the nonvarying singularities we have obtained pave an avenue towards determining the $K(\pi, 1)$-property for the corresponding strata of differentials as well as understanding special loci in Teichm\"uller dynamics. Note that the Coxeter--Dynkin diagrams of $ADE$ singularities and the corresponding curve systems also appear in the construction of Teichm\"uller curves; see~\cite[Section 6]{CTM23} for more details.

Additionally, certain nonreduced Gorenstein curves (such as ribbons) can naturally appear in the log minimal model program for $\BM_g$; see \cite[Section 3.3]{AFS16}. Therefore, it would be interesting to find a correspondence between holomorphic differentials and nonreduced Gorenstein curves. Similarly, we can ask how other singularities, e.g., those that are not Gorenstein, or without $\bbG_m$-action, or of higher dimension, can appear in the story, as well as how their deformation spaces fit together. From the perspective of differentials, we can also study other variants of the construction by using meromorphic differentials or higher-order differentials instead of holomorphic one-forms. 

The weight spectra we have computed for the spaces of holomorphic differentials on Gorenstein curves with $\bbG_m$-action are also closely related to the Harder--Narasimhan polygons of vector bundles. In particular, similar ideas have been successfully employed to study the Lyapunov spectra of Teichm\"uller geodesic flows; see~\cite{YZ13, Y18, EKMZ18, CY25}. Indeed, our construction is similar to the 
Harder--Narasimhan stratification for moduli spaces of vector bundles and the Bialynicki--Birula stratification for moduli spaces of Higgs bundles. Moreover, analogous calculations for the weight spectra of deformation spaces of singularities are prevalent in the study of Coulomb branches in Seiberg--Witten theory. It would be interesting to find out the precise relations as well as to draw on novel ideas and tools between these subjects.   

We expect that the results and techniques developed in this paper can lead to further advances on many related topics, particularly for the construction of new compactifications of moduli spaces of curves and differentials, their topology, birational and affine geometry, as well as the computations of associated invariants. We plan to explore these questions in future work.  

\subsection{Outline of the paper} In Section~\ref{sec:test}, we carry out the construction from a holomorphic differential of given zero type to a Gorenstein curve singularity with $\bbG_m$-action. Moreover, we verify the desired properties of this construction and prove Theorem~\ref{thm:main}. 

In Section~\ref{sec:ex}, through a series of examples, we demonstrate explicitly how to compute the weights and characters of the $\bbG_m$-action on the space of differentials of the resulting Gorenstein curves. Although most of the examples coincide with those in \cite{AFS16}, here the focal point is our unifying method through the filtrations of the Hodge bundle that arise in the construction of test configurations. 

In Section~\ref{sec:numerical}, we explore the numerical properties of the weights and characters of the $\bbG_m$-action. In particular, we show that the weight spectrum of the $\bbG_m$-action on the space of one-forms determines that of higher-order differentials, thus proving Theorem~\ref{thm:char}. 

In Section~\ref{sec:nonvarying}, we study the singularities associated to the nonvarying strata of holomorphic differentials and prove Theorem~\ref{thm:nonvarying}. We also study the behavior of the $K(\pi,1)$-property for singularities obtained by adding an ordinary marked point, which proves Theorem~\ref{thm:US}. In the course of the proofs, we describe the invariants and defining equations for these singularities. Additionally, we study several varying strata and describe the interesting geometry for the singularities associated to them. 

In Section~\ref{sec:3/8}, we classify smoothable isolated Gorenstein curve singularities with $\bbG_m$-action whose $\alpha$-invariants are bigger than or equal to $3/8$ (in the setting of no dangling branches), which proves Theorem~\ref{thm:3/8}. Additionally, we describe the invariants and defining equations for these singularities as well. 

In Section~\ref{sec:slope}, we define the slopes of  singularities and utilize them from the viewpoint of bounding the slopes of effective divisors in $\BM_g$, as in Theorem~\ref{thm:slope-main}. 

Finally, in Section~\ref{sec:Weierstrass}, we study loci of Weierstrass points with fixed semigroups and prove Theorem~\ref{thm:Weierstrass} for symmetric semigroups and subcanonical points. 

\subsection*{Acknowledgements}  The authors would like to thank Dan Abramovich, Jarod Alper, Luca Battistella, Sebastian Bozlee, Iacopo Brivio, Qile Chen, Yifei Chen, Giulio Codogni, David Eisenbud, Yu-Wei Fan, Gavril Farkas, Maksym Fedorchuk, Riccardo Giannini, Samuel Grushevsky, Paul Hacking, Brendan Hassett, David Holmes, Dihua Jiang, Bochao Kong, Hannah Larson, Yongnam Lee, Yuchen Liu, Eduard Looijenga, Michael L\"onne, Martin M\"oller, Adrian Neff, Yuji Odaka, Yu Qiu, Dhruv Ranganathan, Yongbin Ruan, Song Sun, Luca Tasin, Filippo Viviani, Jonathan Wise, Alex Wright, Dan Xie, Junyan Zhao, Anton Zorich, Huaiqing Zuo, and Kang Zuo for inspiring discussions on many related topics. 

Part of the work was carried out when the first-named author visited Zhejiang University in Summer 2024. The authors thank Zhejiang University for providing a wonderful working environment.  

\section{Singularities and test configurations}
\label{sec:test}

In this section, we first review some classical results about deformations of singularities with $\bbG_m$-action; see~\cite{P74} and~\cite[Appendix]{L84} for more details and other related references.  Then we study a correspondence between holomorphic differentials with prescribed orders of zeros and Gorenstein singularities with $\bbG_m$-action by using test configurations. 

\subsection{Deformations of singularities with $\bbG_m$-action}
\label{subsec:def}

Throughout the paper, we use $(Y, q)$ to denote (the analytic germ of) a reduced affine curve $Y$ with a unique isolated singularity at $q$. Suppose $(Y, q)$ admits a (good) $\bbG_m$-action that acts freely on $Y\setminus q$, with the convention that it acts with all positive or all negative weights on the maximal ideal of $\widehat{\calO}_{Y,q}$. In other words, $\widehat{\calO}_{Y,q}\cong \bbC\{y_1, \ldots, y_n \} / I$, where the weights $a_i$ of $y_i$ are either all positive or all negative and $I$ is a graded ideal generated by weighted homogeneous polynomials. The (good) $\bbG_m$-action on $(Y, q)$ is thus given by 
$\lambda\cdot (y_1,\ldots, y_n) = (\lambda^{a_1}y_1, \ldots, \lambda^{a_n}y_n)$. Such singularities are also called {\em quasi-homogeneous}. 

Pinkham \cite[Section 2]{P74} showed that $(Y, q)$ has a {\em $\bbG_m$-equivariant versal deformation space}, which we denote by ${\rm Def}(Y, q)$.
We remark that Pinkham's definition of ``versal'' includes the extra condition that the tangent space to the base space of the deformation has minimal dimension, which is commonly called {\em miniversal} or {\em semi-universal} in the literature; see~\cite[Definition (1.4)]{P74} and \cite[Section 2]{P78}. 

Let $T^1$ denote the space of {\em first-order infinitesimal deformations} of $(Y, q)$.  In general, we have the following dimension bounds: 
$$\dim T^1 \geq \dim {\rm Def}(Y, q) \geq \dim T^1 - \dim T^2,$$ 
where $T^2$ is the space of {\em obstructions} against lifting; see~\cite[Corollary 1.32]{GLS}. We remark that 
$\dim T^1$ is called the {\em Tjurina number} of the singularity. 

Pinkham also showed that 
$$T^1(Y,q) = \sum_{\nu = -\infty}^{\infty} T^1(\nu)$$ 
has a natural {\em grading}. We denote by ${\rm Def}^{-}(Y, q)$ the subspace of ${\rm Def}(Y, q)$ corresponding to 
{\em negatively graded deformations}, i.e., the tangent space to ${\rm Def}^{-}(Y, q)$ is $T^{1, -} = \sum_{\nu < 0} T^1(\nu)$. Both the family of negative deformations and the base space ${\rm Def}^{-}(Y, q)$ are defined by polynomials. By Pinkham's convention, the maximal ideal of $0$ in the coordinate ring of ${\rm Def}^{-}(Y,q)$
is {\em positively} graded due to the duality between the tangent vectors of $T^1$ and the deformation parameters of the base; see~\cite[(2.9)]{P74}. We denote by 
$${\rm Def}^{-}_s(Y, q)\subset {\rm Def}^{-}(Y, q)$$ 
the open subset of {\em smooth deformations with negative grading}. 

In particular, if $\dim T^{1, -} = \dim {\rm Def}^{-}_s(Y, q)$ and $T^2= \{0\}$, then the {\em weighted projective space}
$[({\rm Def}^{-}(Y, q)\setminus 0) / \bbG_m] \cong \bbP (T^{1, -})$ provides a {\em compactification} for the subset 
\begin{align}
\label{eq:def_s}
\bbP {\rm Def}^{-}_s(Y, q)\coloneqq [ ({\rm Def}^{-}_s(Y, q)\setminus 0) / \bbG_m] 
\end{align}
of negatively graded smooth deformations.  

\begin{example}[Deformations of the ordinary cusp]
\label{ex:cusp}
For the reader's convenience, we illustrate the above terms through the well-known example of the ordinary cusp $(Y, q)$ defined by $X_1^3 - X_2^2$; see~\cite[(1.16)]{P74}. The $\bbG_m$-action acts with weight $2$ on $X_1$ and weight $3$ on $X_2$. The miniversal deformation of the cusp is given by 
$$ F = X_1^3 - X_2^2 + t_1 X_1 + t_2$$ 
over $R = \bbC [t_1, t_2]$. We assign to $t_1$ and $t_2$ the weights $4$ and $6$, respectively, and then $F$ becomes a homogeneous equation. In this case, $t_1$ and $t_2$ have positive weights, and hence by Pinkham's convention, the miniversal deformation of the cusp has negative grading only. Next, we replace $t_1$ and $t_2$ by  $t_1 X_0^{4}$ and $t_2 X_0^6$, respectively, where $X_0$ is a new variable with weight $1$. We obtain 
$$ \overline{F} = X_1^3 - X_2^2 + t_1 X_0^{4}X_1 + t_2 X_0^6 \in R [X_0, X_1, X_2]. $$
Consider 
$$ {\rm Proj} (R[X_0, X_1, X_2] / (\overline{F})) \to {\rm Def}^{-}(Y, q)\coloneqq \Spec R, $$
where $R[X_0, X_1, X_2] / (\overline{F})$ is graded in the $X_i$'s only and $\bbG_m$ acts on $R$ as before. It is easy to see that all fibers above points of the same $\bbG_m$-orbit in ${\rm Def}^{-}(Y, q)$ are isomorphic. Indeed, all fibers are smooth elliptic curves, except the original cuspidal curve $(Y, q)$ above $(t_1, t_2)$ and
a family of rational nodal curves above a unique $\bbG_m$-orbit in ${\rm Def}^{-}(Y, q)$. Additionally, there is a natural section obtained by setting $X_0 = 0$. 
Therefore, the quotient $[({\rm Def}^{-}(Y, q) \setminus 0)/\bbG_m]$ thus provides the usual compactification of the affine $j$-line for parameterizing elliptic curves, where the boundary of the compactification parameterizes the isomorphism class of rational nodal curves. 
\end{example}

\begin{example}[Deformations with repulsive $\bbG_m$-action]
This example was already noted in \cite[(13.1)]{P74}. Consider the plane curve singularity $(Y, q)$ defined by $X_1^5 - X_2^4$. It is a special case of monomial curve singularities whose associated semigroup is generated by $4$ and $5$; see~Section~\ref{subsec:monomial} later for a detailed discussion on monomial curve singularities. Consider the deformation $X_1^5 - X_2^4 + tX_1^3 X_2^2$. In this case, $X_1$ and $X_2$ have weights $4$ and $5$ under the $\bbG_m$-action, respectively. Therefore, the deformation parameter $t$ has weight $-2$. In other words, the corresponding tangent vector in $T^1$ has positive grading $2$, hence not in ${\rm Def}^{-}(Y,q)$. As a result, the $\bbG_m$-action along this direction is repulsive, which is not a good $\bbG_m$-action. 
\end{example}

Next, following Looijenga's discussion~\cite[Appendix]{L84}, we let $R$ be the {\em affine coordinate ring} of $Y$, i.e., $Y = \Spec R$, where $R = \bigoplus_{k=0}^{\infty} R_k$ is graded with $R_0 = \bbC$.  The graded structure of $R$ corresponds to the $\bbG_m$-action on $Y$, given by $\lambda \cdot f = \lambda^k f$ for $\lambda \in \bbG_m$ and $f\in R_k$. Let $ t = (1, 0) \in R_0 \oplus R_1$. Then $Y$ has a {\em standard projectivization} given by $\overline{Y} \coloneqq \Proj R[t]$. We define the {\em divisor at infinity} to be 
\begin{align}
\label{eq:infinity}
Y_{\infty}\coloneqq \overline{Y} - Y,
\end{align}
which is the ample divisor defined by the hyperplane $t$. This construction also works in a relative situation. 

Looijenga showed that the part of the negatively graded miniversal deformation ${\rm Def}^{-}(Y, q)$ is a final object in the category of deformations of $(Y, q)$ with good $\bbG_m$-action; see~\cite[(A.2) and (A.6)]{L84}. Furthermore, for $X \in \bbP {\rm Def}^{-}(Y, q)$ with the divisor at infinity $X_{\infty}$, we have the isomorphism of graded algebras $R_{\overline{X}} / t R_{\overline{X}} \cong R$, where 
$R_{\overline{X}}\cong \bigoplus_{k=0}^{\infty} H^0(\calO_{\overline{X}} (k X_{\infty})) $ and $t\in  H^0(\calO_{\overline{X}} (X_{\infty}))$ is the element corresponding to $1$; see~\cite[(A.5)]{L84}. 

\begin{remark}
\label{rem:notation}
In the above context, due to the $\bbG_m$-action, the analytic germ of the singularity at $q\in Y$ and the projective curve $\overline{Y}$ determine the isomorphism classes of each other; see \cite[Proposition 2.4 and Remark 2.5]{AFS16}. Therefore, we do not need to distinguish between the two notions. In particular, to simplify notations, from now on we will also denote by $Y$ the standard projectivization after adding the divisor at infinity as in~\eqref{eq:infinity}.  
\end{remark}

\begin{remark}
\label{rem:smoothing}
In general, ${\rm Def}(Y,q)$ can be reducible. An irreducible component of ${\rm Def}(Y,q)$ is called a {\em smoothing component} if its generic fiber is smooth. If $(Y,q)$ is a smoothable reduced curve singularity, then all smoothing components of $(Y,q)$ have the same dimension, which is called the {\em Deligne number}. Moreover, if $(Y,q)$ is an isolated Gorenstein curve singularity with good $\bbG_m$-action (i.e., quasi-homogeneous), then the Deligne number is equal to the {\em Milnor number} $2g+n-1$, where $g$ is the arithmetic genus of the projectivization of $Y$ and $n$ is the number of branches. We refer to \cite[Section 2.6]{G20} for a survey about these concepts and their relations. 
\end{remark}

\subsection{From differentials to singularities}
\label{subsec:construction}

Let $\mu = (m_1, \ldots, m_n)$ be a partition of $2g-2$, where $m_i \geq 0$ for all $i$. Given any $(C, p_1, \ldots, p_n) \in \bbP\calH (\mu)$, we will associate to it an isolated curve singularity. To this end, we define the following parameters that will be used frequently throughout the paper: 
$$\ell \coloneqq \lcm (m_1 + 1, \ldots, m_n+1) \quad {\rm and} \quad a_i \coloneqq \frac{\ell}{m_i + 1}$$ 
for $i = 1,\ldots, n$. 

Consider the trivial one-parameter family 
$$C_{\bbA^1}\coloneqq C\times \bbA^1$$ 
with base parameter $t$ and fiber parameters $x_i$ at each $p_i$. We perform a {\em weighted blowup} at each $p_i$ in the central fiber over $t=0$, where the local ideal of the blowup is $(x_i, t^{a_i})$. In particular, the resulting exceptional curve $E_i\cong \bbP^1$ is attached to the proper transform $\tilde{C}_0\cong C$ at a node of local type $x_it_i = t^{a_i}$ for $i = 1, \ldots, n$, where $t_i$ is the local coordinate of $E_i$ at the corresponding node.   

Let $\tilde{\pi}\colon \tilde{\calC}\to \bbA^1$ be the family after the blowup and $\tilde{P}_i$ the proper transform of the section $P_i = \{p_i\} \times \bbA^1$. Define the {\em twisted relative log dualizing bundle} 
\begin{align}
\label{eq:twisted}
\tilde{\calL} \coloneqq \omega_{\tilde{\pi}}\left(\sum_{i=1}^n \tilde{P}_i + \ell \tilde{C}_0\right).
\end{align}
It restricted to $\tilde{C}_0$ is isomorphic to 
$$\omega_{\tilde{\pi}}|_{\tilde{C}_0} \left( - \sum_{i=1}^n (\ell / a_i) p_i\right) \cong \omega_C \left(-\sum_{i=1}^n m_i p_i\right) \cong \calO_C.$$
Moreover, $\tilde{\calL}$ restricted to each $E_i$ is isomorphic to $\calO_{\bbP^1}(m_i+1)$ which is ample, and $\tilde{\calL}$ restricted to any other fiber over $t\neq 0$ is isomorphic to 
$$\omega_C^{\log} \coloneqq \omega_C\left(\sum_{i=1}^n p_i\right) \cong \calO_C\left(\sum_{i=1}^n (m_i+1)p_i\right)$$ 
which is also ample. 

In particular, this family $\tilde{\pi}$ parameterizes an {\em isotrivial} degeneration of holomorphic differentials of type $\mu$ on $C$ to a {\em multi-scale differential} $\tilde{\eta}$ on $\tilde{C}_0$ in the sense of \cite{BCGGM1, BCGGM3, CGHMS}, where $\tilde{\eta}$ has a pole of order $m_i+2$ at the nodal branch along each $E_i$. By using the explanation in~\cite[Section 2]{CC26}, we can contract $\tilde{C}_0$ through the following birational contraction 
\begin{align}
\label{eq:contraction}
f\colon \tilde{\calC} \longrightarrow \calC \coloneqq \Proj \left(\bigoplus_{k} \tilde{\pi}_{*}\tilde{\calL}^k\right)
\end{align}
where $f$ is an isomorphism away from $\tilde{C}_0$ and the images of $E_1, \ldots, E_n$ intersect at a singularity formed by the contraction image of $\tilde{C}_0$; see also \cite[Proposition 2.6]{S13} for a more general discussion of contractions. Additionally, $\tilde{\calL}$ descends to a (relatively) ample line bundle $\calL$ on $\calC$, where $f^{*}\calL \cong \tilde{\calL}$. 

Denote by $Y$ the central fiber curve of $\calC$ and $q\in Y$ the resulting singularity, where $Y$ consists of the components $Y_i \coloneqq f(E_i)$ for $i=1,\ldots, n$ which meet at the singularity $q$. We write 
$$s(C, p_1, \ldots, p_n) = (Y, q)$$ 
for the input and outcome through the above construction. We will see that smooth pointed curves in $\bbP \calH(\mu)$ that produce the same $(Y, q)$ possess a refined similarity. Therefore, we gather them together as follows.  

\begin{definition}
\label{def:Def}
Let $S(\mu)$ be the {\em set of isomorphism classes of $(Y, q)$} that arise from the above construction. Define the locus 
$$\bbP\calH(\mu)_{(Y, q)}\subset \bbP \calH(\mu),$$ 
where $(C, p_1, \ldots, p_n)\in \bbP\calH(\mu)_{(Y, q)}$ if $s(C, p_1, \ldots, p_n) \cong (Y, q)$. 
\end{definition}

In this way, $\bbP \calH(\mu)$ can be decomposed according to the 
isomorphism classes of the resulting singularities: 
$$ \bbP \calH(\mu) = \bigsqcup_{(Y,q)\in S(\mu)} \bbP\calH(\mu)_{(Y, q)}. $$

Later in Theorem~\ref{thm:filtration} (iv), we will see that $ \bbP \calH(\mu)_{(Y,q)}$ can be identified with the locus of negatively graded smooth deformations $\bbP {\rm Def}^{-}_s(Y, q)$ in the miniversal deformation space as introduced in~\eqref{eq:def_s}. 

\subsection{Properties of the construction}
\label{subsec:properties}

Next, we will describe the upshots of the preceding construction. To this end, we need to introduce the following filtration for the spaces of differentials: 
\begin{align}
\label{eq:filtration}
 \calF^{\lambda} H^0(C, m \omega_C^{\log})  \coloneqq \bigcap_{i=1}^n \left\{ s\in H^0 (C, m  \omega_C^{\log})  \mid  a_i \ord_{p_i} (s)  \geq \lambda \right\}.
\end{align}

\begin{theorem}
\label{thm:filtration}
In the above construction, the following statements hold: 
\begin{enumerate}
\item[\rm (i)] $Y$ is a smoothable curve with an isolated singularity at $q$ which is Gorenstein with $\bbG_m$-action, and $(\calC, \calL)$ is a test configuration of $(C, \omega_C^{\log})$. 
\item[\rm (ii)] There are $\bbG_m$-equivariant isomorphisms of graded rings: 
\begin{align}
\bigoplus_{m\in \bbN} H^0(\calC, m\calL) & \cong \bigoplus_{m\in \bbN} \bigoplus_{\lambda \in \bbZ}\calF^{\lambda} H^0(C, m \omega_C^{\log})t^{-\lambda},  \label{eq:iso-1} \\
\bigoplus_{m\in \bbN} H^0(Y, m \omega_Y^{\log}) & \cong \bigoplus_{m\in \bbN} \bigoplus_{\lambda \in \bbZ}\calF^{\lambda} H^0(C, m \omega_C^{\log}) / \calF^{\lambda + 1} H^0(C, m \omega_C^{\log}).  \label{eq:iso-2} 
\end{align}
In particular, if $\lambda < 0$, then $\calF^{\lambda} H^0(C, m \omega_C^{\log}) / \calF^{\lambda + 1} H^0(C, m \omega_C^{\log}) = 0$. 
\item[\rm (iii)] Consider the $\bbG_m$-action on $Y$ given by $\lambda\cdot t_i = \lambda^{-a_i} t_i$, where $t_i$ is the uniformizer of the normalization of each branch $Y_i$ of $Y$ at $q$. Let $\chi_m^{\log}$ and $\chi_m$ be the characters of this $\bbG_m$-action on $H^0(Y, m \omega_Y^{\log})$ and $H^0(Y, m \omega_Y)$, respectively. Then 
\begin{eqnarray}
\label{eq:log}
 \chi_1^{\log} (Y) = \chi_1 (Y) \quad {\rm and} \quad \chi_2^{\log} (Y) = \chi_2 (Y) + \sum_{i=1}^n a_i. 
 \end{eqnarray}
 \item[\rm (iv)] The space $\bbP\calH(\mu)_{(Y, q)}$ can be identified with the locus $\bbP{\rm Def}^{-}_s(Y, q)$ of smooth deformations with negative grading in the miniversal deformation space.  
 \item[\rm (v)] Let $y_i$ be the point at infinity of each rational branch $Y_i$ for $i = 1,\ldots, n$. Then the divisor at infinity for $Y$ defined in~\eqref{eq:infinity} is 
 \begin{align}
 \label{eq:Y-infinity}
 Y_{\infty} = \frac{1}{\ell} \cdot \sum_{i=1}^n (m_i+1) y_i \sim \frac{1}{\ell} \cdot \omega_{Y}^{\log}. 
 \end{align}
 Moreover, $\omega_{Y}^{\log}$ is an ample line bundle class. The same conclusion holds when replacing $(Y, y_1, \ldots, y_n)$ by the nearby fiber $(C, p_1, \ldots, p_n)$.  
\end{enumerate}
\end{theorem}

\begin{proof} 
The preceding construction yields a commutative diagram 
$$
\xymatrix{
&  \tilde{\calC} \ar[ld]_{f} \ar[rd]^{\rho} &  \\
\calC \ar@{-->}[rr] &  & C_{\bbA^1}
}
$$
where $\rho$ is the weighted blowup, $f$ is the birational contraction, and the central fiber of $\calC$ is the singular curve $Y$ with an isolated singularity at $q$. 

For the first claim, the family $\calC$ is a smoothing of $Y$. Since the ideal of the blowup is $\bbG_m$-equivariant with respect to the standard $\bbG_m$-action on $\bbA^1$, it induces a natural $\bbG_m$-action on $\calC$ which satisfies the definition of test configurations in \cite[Section 2]{BHJ17}; see in particular \cite[Example 2.5]{BHJ17}. Additionally, the relative holomorphic one-form on $\tilde{\calC}$ with zeros $\sum_{i=1}^n m_i \tilde{P}_i$ descends to $\calC$ as a local generator for the relative dualizing bundle of $\calC$ at $q$, which implies that $q$ is a Gorenstein singularity; see \cite[Section 2]{CC26} for a more detailed explanation.   
 
For the second claim, the isomorphisms of graded rings follow from the Rees construction; see \cite[Section 2.5]{BHJ17} and \cite[Section 4.1]{ABBDILW23}. The description of each graded component follows from \cite[Lemma 5.17]{BHJ17}. More precisely, for each irreducible component $Y_i$ of the central fiber $Y$ in $\calC$, we have  
$$ b_{Y_i}\coloneqq \ord_{Y_i} (\calC)  = 1 \quad {\rm and} \quad v_{Y_i} \coloneqq b_{Y_i}^{-1} r (\ord_{Y_i}) = r (\ord_{Y_i}) $$
where $r(\ord_{Y_i})$ is the restriction of $\ord_{Y_i}$ from $K(\calC) \cong K(C)(t) $ to $K(C)$. Then for $s\in H^0(C, m \omega_{C}^{\log})$, the valuation $v_{Y_i} (s) = a_i \ord_{p_i}(s)$ due to the weighted blowup. Moreover, let $L_{\bbA^1}$ be the trivial family of line bundles $\omega_C^{\log}$ on $C_{\bbA^1}$. Note that  
$$\rho^{*} L_{\bbA^1} =  \omega_{\tilde{\pi}}\left(\sum_{i=1}^n \tilde{P}_i\right) \quad {\rm and} \quad f^{*} \calL = \omega_{\tilde{\pi}}\left(\sum_{i=1}^n \tilde{P}_i + \ell \tilde{C}_0\right).$$ It follows that  
$$ f^{*} \calL  = \rho^{*}  L_{\bbA^1} (\ell \tilde{C}_0)  $$
where $\ord_{Y_i} (\ell \tilde{C}_0) = 0$ since the proper transform $\tilde{C}_0$ does not contain $Y_i$.  
By the proof of \cite[Lemma 5.17]{BHJ17} and~\cite[Equation (4.3)]{ABBDILW23}, we conclude that the subspace with weight $\geq \lambda$ in $H^0(C, m \omega_C^{\log})$ is  
$$\calF^{\lambda} H^0(C, m \omega_C^{\log}) = \bigcap_{i=1}^n \left\{ s\in H^0 (C, m  \omega_C^{\log})  \mid  a_i \ord_{p_i} (s) \geq \lambda \right\} $$
as defined in~\eqref{eq:filtration}. 

For the third claim, note that our sign convention of the $\bbG_m$-action $\lambda \cdot t_i \mapsto \lambda^{-a_i} t_i$ coincides with that of \cite[Section 3.1.1]{AFS16}. Then the desired identities thus follow from \cite[Corollary 3.3 and Corollary 3.6]{AFS16}. In contrast, if we change the sign of the $\bbG_m$-action as in \cite[Section 4.1]{AFS16}, then we would also need to change the sign of the corresponding characters; see \cite[Theorem 4.2]{AFS16}. 

For the fourth claim, we refer to \cite{P74} about the properties of the miniversal deformation spaces for singularities with $\bbG_m$-action; see also \cite[Appendix]{L84}. In our case, recall that $\bbP {\rm Def}^{-}_s(Y, q)$ is the part of smooth deformations with negative grading. If $s(C, p_1, \ldots, p_n) = (Y, q)$, then by our construction, it is clear that $(C, p_1, \ldots, p_n) \in \bbP {\rm Def}^{-}_s(Y, q)$, where the $\bbG_m$-action is attractive to the singularity $q$ and hence the corresponding deformation direction has negative grading in Pinkham's convention. Conversely, consider a smoothable Gorenstein curve singularity $(Y, q)$ with $\bbG_m$-action such that $Y$ is a projective curve with $n$ irreducible rational components $Y_i$ that meet at $q$ and such that $(Y, q)$ admits a $\bbG_m$-equivariant smoothing. A local generator $\eta$ of $\omega_Y$ at $q$ restricted to each branch $Y_i$ is a meromorphic differential $\eta_i = d t_i / t_i^{m_i+2}$ for some $m_i \geq 0$. Then after performing stable reduction, the associated (smooth) tail of $(Y, q)$ (in the sense of \cite[Definition 3.1]{H00} and \cite[Section 4.4]{FS13}) is a curve $C$ with a canonical divisor $\sum_{i=1}^n m_i p_i$. This curve $C$ is attached at each $p_i$ to a rational component $E_i$ that maps to $Y_i$ in the process of stable reduction; see \cite[Section 1.4]{B24} for more details. In particular, $(C, p_1, \ldots, p_n) \in \bbP \calH(\mu)$ for $\mu = (m_1, \ldots, m_n)$. Hence, the identification between $\bbP\calH(\mu)_{(Y, q)}$ and $\bbP{\rm Def}^{-}_s (Y, q)$ is established. 

For the last claim, the divisor at infinity in~\eqref{eq:infinity} along each branch $Y_i$ is defined by $t$ and $s_i = 1/t_i$. Due to the weighted blowup of the ideal $(x_i, t^{a_i})$, where $a_i = \ell / (m_i+1)$, the restriction of $(t, s_i)$ to $Y_i$ picks up the point $y_i$ defined by $s_i = 0$ with coefficient $1/a_i = (m_i+1) / \ell$, and the last equivalence relation of divisor classes follows from $\sum_{i=1}^n (m_i +1) y_i \sim \omega_{Y}^{\log}$ in the construction of the contraction map $f$. Alternatively, we can compute the divisor at infinity by using~\cite[Section 5, Formula ($\ast$)]{P77}. Indeed, in the family $\tilde{\calC}$, we can further blow up each node of type $x_i t_i = t^{a_i}$ in the central fiber for $i = 1,\ldots, n$, replacing it by a chain of rational curves of length $a_i$. The resulting family thus provides a minimal good resolution of $\calC$ in the setting of \cite[Section 2]{P77}, where the related invariants can be read off easily. Finally, note that $\omega_Y^{\log}$ is the restriction of the (relatively) ample line bundle $\calL$ to $Y$, where $f^{*}\calL \cong \tilde{\calL}$ induces the contraction as in~\eqref{eq:contraction}. Therefore, $\omega_Y^{\log}$ is ample. The same conclusion holds for $C$ and its divisor at infinity in the negatively graded miniversal deformation space of $(Y,q)$; see~\cite[(A.5)]{L84}. 
\end{proof}

\subsection{The boundary of $\bbP{\rm Def}^{-} (Y, q)$}
\label{subsec:boundary}

As seen in Example~\ref{ex:cusp}, besides smooth elliptic curves in the miniversal deformation space of the ordinary cusp, the rational nodal curve also appears. In general, given a Gorenstein curve singularity $(Y, q)$ with $\bbG_m$-action, a natural question is to determine which singular curves can appear in $\bbP{\rm Def}^{-} (Y, q)$. This question has been studied in various special situations, mostly for unibranched monomial singularities with symmetric semigroups, for local complete intersections, and for low genus cases; see~\cite{S93, CS13, CF23, CM25, D90, LP19, B22}. Utilizing the ideas in Theorem~\ref{thm:filtration}, we can provide a unified description to answer this question. 

Analogous to the usual stability condition for nodal curves, we will use the following notion for Gorenstein curves.  Let $(C, p_1, \ldots, p_n)$ be a reduced Gorenstein curve with $n$ marked points $p_1, \ldots, p_n$ contained in the smooth locus of $C$. We say that $(C, p_1, \ldots, p_n)$ is a {\em stable pointed Gorenstein curve} if $\omega_C^{\log} = \omega_C(p_1+\cdots + p_n)$ is an ample line bundle class. Furthermore, given a partition $\mu = (m_1, \ldots, m_n)$ of $2g-2$, we say that the pointed curve $(C, p_1, \ldots, p_n)$ is {\em of type $\mu$} if $\sum_{i=1}^n m_i p_i$ is a canonical divisor of $C$. 

We remark that the stability condition imposed to pointed Gorenstein curves is standard; see~\cite[Footnote 3)]{AFS16}. It also naturally fits the ampleness requirement for the divisor at infinity in the setting of \cite[(A.5)]{L84}, since in our case the divisor at infinity is $\omega^{\log}/\ell$. 

\begin{theorem}
\label{thm:boundary}
Let $(Y, q)$ be a stable pointed Gorenstein curve of type $\mu$ with good $\bbG_m$-action, where $q$ is the unique singularity. Then deformations in $\bbP{\rm Def}^{-} (Y, q)$ parameterize stable pointed Gorenstein curves $(C, p_1, \ldots, p_n)$ of type $\mu$ such that the isomorphism of graded rings in~\eqref{eq:iso-2} holds. 
\end{theorem}

As mentioned before, Looijenga \cite[(A.6)]{L84} revealed the functorial meaning of $\bbP{\rm Def}^{-}(Y, q)$ as the fine moduli space of deformations with good $\bbG_m$-action that are polarized by the divisor at infinity. Theorem~\ref{thm:boundary} thus provides an explicit geometric description for the deformations that can appear in this moduli space.  
\begin{proof}
First, suppose $C = C_t$ is a (possibly singular) curve that admits an isotrivial degeneration to $(Y, q)$ with respect to the $\bbG_m$-action, where the base parameter is $t$. Since being Gorenstein is an open condition in the moduli stack of curves, it follows that $C$ is Gorenstein, and hence the Hodge bundle of rank $g$ is well-defined over this family. Let $y_i$ be the point at infinity in each branch $Y_i$ of $Y$, where $\sum_{i=1}^n m_i y_i$ is a canonical divisor of $Y$ by assumption. Since the Hodge bundle parameterizing canonical divisors (up to a scaler multiple) has constant rank on each fiber curve, we can lift $\sum_{i=1}^n m_i y_i$ to a canonical divisor $D_t$ in the nearby fiber $C_t$ for $t\neq 0$. Since $C_t \cong C$, we can also view $\{D_t\}$ as a family of canonical divisors in $C$. As $t\to 0$, the points in the support of $D_t$ converge to $n$ distinct limit points, each with multiplicity $m_i$. Therefore, their limit $D$ as a canonical divisor in $C$ is of type $\mu$. Since each $y_i$ is a smooth point of $Y$, the support of $D$ lies in the smooth locus of $C$ as well. 

Conversely, suppose $(C, p_1, \ldots, p_n)$ is a stable pointed Gorenstein curve with a differential $\omega$ as a section of the dualizing bundle of $C$, where the zeros of $\omega$ yield the canonical divisor $\sum_{i=1}^n m_i p_i$ of type $\mu$. Since each $p_i$ is a smooth point of $C$, the construction of the weighted blowup in Section~\ref{subsec:construction} can be carried out in the same way.  Additionally, $\omega_C^{\log}$ is ample by assumption, which implies that the twisted relative log dualizing bundle $\tilde{\calL}$ can be defined as in~\eqref{eq:twisted} and remains relatively semiample. Indeed, let $\tilde{\eta}$ be the multi-scale differential associated to the limit of $\omega$ in the central fiber of the weighted blowup. Then $\tilde{\eta}$ as a relative section of $\tilde{\calL}$ restricted to the central fiber has zeros only at the proper transforms $y_i$ of $p_i$. On the other hand, we can take a general section of the ample line bundle  $\omega_C^{\log}$ such that it does not vanish at any $p_i$, and hence as a relative section of $\tilde{\calL}$ it does not vanish at any $y_i$. Since $\tilde{\calL}$ is a subbundle of $\omega_{\tilde{\pi}}^{\log}\otimes \calO_{\tilde{\calC}}(\ell(\tilde{\calC}_0 + \sum_{i=1}^n E_i))$ and the base of the family is a DVR, it follows that $\tilde{\pi}_{*}\tilde{\calL}^{k}$ is locally free for $k\gg 0$. Therefore, the contraction in~\eqref{eq:contraction} holds. 

We can also justify that the resulting singularity $q$ in the central fiber $Y$ after the contraction is Gorenstein. Indeed, the Residue Theorem holds for reduced singular curves as well; see~\cite[Lemma 1.1.2]{LP19}. Additionally, by assumption $\tilde{\calC}$ is Gorenstein in codimension $1$, and it also satisfies Serre's condition $S_2$ since the base is smooth and the fibers are reduced curves. It follows that the dualizing sheaf of $\tilde{\calC}$ can be identified with the sheaf associated to a canonical divisor; see~\cite[Section 5]{K13}.  Consequently, the same argument as in~\cite[Section 2]{CC26} implies that the multi-scale differential $\tilde{\eta}$ as the limit of $\omega$ can descend to $Y$ and locally generate its dualizing sheaf at the singularity $q$, since the zeros of $\tilde{\eta}$ are along the divisor at infinity which are disjoint from $q$. The construction of the blowup and contraction as in Section~\ref{subsec:construction} is $\bbG_m$-equivariant. Therefore, $(Y, q)$ is a Gorenstein curve with $\bbG_m$-action of type $\mu$, and its stability follows from Theorem~\ref{thm:filtration} (v). 

Finally, the isomorphism of graded rings as in~\eqref{eq:iso-2} follows from \cite[(A.5)]{L84}, where the class of the divisor at infinity is just $\omega^{\log}/\ell$. Indeed, the right-hand side of \eqref{eq:filtration} can be rewritten as 
\begin{align*}
    &\ \left\{ s\in H^0 (C, m  \omega_C^{\log})  \mathrel{\Big|}   (s)_0 \geq \sum_{i=1}^n \frac{\lambda}{a_i}p_i\right\} \\
   = &\ \left\{ s\in H^0 \left(C, m  \omega_C^{\log}-\sum_{i=1}^n\frac{\lambda}{a_i}p_i\right)  \right\} \\
   = &\ H^0\left(C, (m\ell - \lambda) C_{\infty}\right), 
\end{align*}
where $C_{\infty} = \sum_{i=1}^n (m_i+1)p_i / \ell = \sum_{i=1}^n p_i / a_i $ is the divisor at infinity for $C$. 
\end{proof}

As an example, consider the special case when $(Y,q)$ is a monomial singularity whose gap sequence contains $2g-1$; see Section~\ref{subsec:monomial} below for the related definitions. The following description of the boundary of $\bbP{\rm Def}^{-}(Y, q)$ was first established by St\"ohr \cite{S93} through a delicate study of Gr\"obner bases and syzygies for canonically embedded Gorenstein curves. Now using Theorem~\ref{thm:boundary}, we can give a concise and conceptual explanation as follows. 

\begin{corollary}
\label{cor:boundary}
Let $(Y,q)$ be a monomial curve singularity of genus $g$ whose gap sequence $G$ contains $2g-1$. Then the boundary of $\bbP{\rm Def}^{-}(Y, q)$ parameterizes stable pointed integral Gorenstein curves 
$(C, p)$ where the Weierstrass gap sequence of $p$ in $C$ is equal to $G$. 
\end{corollary}

\begin{proof}
In the case of a monomial singularity, the graded ring structure in \eqref{eq:iso-2} is determined by the corresponding semigroup that defines the singularity, whose complement is thus the gap sequence $G$; see Section~\ref{subsec:monomial} for more details. 
Note that $2g-1\in G$ is equivalent to that $(2g-2)p$ is a canonical divisor, which corresponds to $\mu = (2g-2)$ in our notation. Therefore, the desired claim follows as a special case of Theorem~\ref{thm:boundary}. Finally, the irreducibility of $C$ is due to the ampleness of 
$\omega_C(p)$; otherwise if $C$ is reducible, then at least one of its irreducible components does not contain $p$, and hence $\omega_C(p)\sim (2g-1)p $ would be trivial restricted to that component, which contradicts that $\omega_C(p)$ is ample on $C$. 
\end{proof}

\section{Examples of singularities and their invariants}
\label{sec:ex}

In this section, we describe various singularities that arise from the preceding construction. In particular, we compute the weights and characters for the associated $\bbG_m$-action on the spaces of differentials. 

Although most examples in this section have been considered in \cite{AFS16}, we highlight our uniform approach through filtrations of the Hodge bundle associated to the test configurations, which will lead to new applications in later sections. As mentioned in Theorem~\ref{thm:filtration} (ii), it suffices to consider $\lambda \geq 0$ for nontrivial contributions of the weights to the characters. We also recall the sign convention of the $\bbG_m$-action on $(Y,q)$ in Theorem~\ref{thm:filtration} (iii), for which we adapt the setting as in \cite[Section 3.1.1]{AFS16}. 

\subsection{Monomial singularities and $\bbP \calH(2g-2)$}
\label{subsec:monomial}

For $\mu = (2g-2)$, we have 
$$(2g-2)p \sim \omega_C, \quad \ell = 2g-1, \quad a = 1.$$
Hence, for $\lambda \geq 0$, we have 
$$ \calF^{\lambda} H^0(C, m \omega_C^{\log}) = H^0 (C, (m(2g-1) - \lambda) p), $$
which is determined by the Weierstrass gap sequence and semigroup of $p$. 

Suppose the {\em Weierstrass gap sequence} $G_p$ consists of 
$$1 = b_1 < b_2 < \cdots < b_g = 2g-1,$$ 
where $h^0(C, (k-1)p) = h^0(C, kp)$ holds for $k \in \bbN$ if and only if $k = b_i$ for some $i$. Additionally, since $2g-1 \in G_p$,  such gap sequences are called {\em symmetric}, namely, $2g-1 - k \in G_p $ if and only if $k$ 
is in the {\em Weierstrass semigroup} $H_p \coloneqq \bbN \setminus G_p $. In this case, $Y$ is the $\bbG_m$-equivariant compactification of the {\em monomial curve} 
$$\Spec \bbC [t^{k}: k\in H_p]$$
 and $q$ is the unibranch singularity at $t = 0$. In particular, $\bbP{\rm Def}^{-}_s (Y, q)$ is the locus of $(C, p)\in \bbP \calH(2g-2)$ where the semigroup of $p$ is generated by the exponents of the monomials defining $(Y, q)$; see \cite[Section 13]{P74}.  

Next, we compute the associated weights and characters. For $m = 1$, we have 
$$\dim  H^0(C, (2g-1 - \lambda)p) / H^0(C, (2g- 2 - \lambda)p) = 1\ \mbox{or}\ 0, $$
where the former occurs if and only if $2g-1 - \lambda \in H_p$, namely, $\lambda = b_i$ for some $i$.  
We thus obtain that 
$$  \chi_1^{\log} (Y)  = \chi_1(Y) = \sum_{i=1}^g b_i. $$ 

For $m = 2$, we have 
$$\dim H^0(C, (4g-2 - \lambda)p) / H^0(C, (4g-3 - \lambda)p) = 1\  \mbox{or}\ 0,$$ 
where the former occurs if and only if 
$4g-2 - \lambda \in H_p$, namely, $\lambda = 2g-1 + b_i$ for some $i$ or $\lambda \leq 2g-2$. 
 We thus obtain that 
$$ \chi_2^{\log}(Y) = (2g-1)^2 + \sum_{i=1}^g b_i \quad {\rm and}\quad  \chi_2 (Y) = (2g-1)^2 + \sum_{i=1}^g b_i - 1 $$
where the second identity follows from \eqref{eq:log}. 

These values are consistent with those in \cite[Section 3.1.7]{AFS16}. 

\subsubsection{$A_{2g}$-singularity and $\bbP \calH(2g-2)^{\hyp}$}
\label{subsec:A2g}
The $A_{2g}$-singularity is given by 
$$\Spec \bbC[x,y] / (y^2 - x^{2g+1}).$$ 
The corresponding gap sequence is $G = \{1, 3, \ldots, 2g-3, 2g-1\}$. In this case, the locus $\bbP{\rm Def}^{-}_s (A_{2g})$ is the {\em hyperelliptic} component $\bbP \calH(2g-2)^{\hyp}$ parameterizing $(C, p)$ where $C$ is hyperelliptic and $p$ is a Weierstrass point; see \cite[Section 3.1.2]{AFS16}. 

\subsubsection{$E_6$-singularity and $\bbP \calH(4)^{\odd}$}
\label{subsec:E6}
The $E_6$-singularity is given by 
$$\Spec \bbC [x, y] / (y^3 - x^4).$$ 
The corresponding semigroup and gap sequence are 
$H = \{ 3, 4, 6, 7, \ldots \}$ and $G = \{1, 2, 5 \}$.  
In this case, the locus $\bbP{\rm Def}^{-}_s (E_6)$ is the {\em odd spin} component $\bbP \calH(4)^{\odd}$ parameterizing $(C, p)$ where $C$ is a nonhyperelliptic curve of genus three with a hyperflex at $p$, i.e., $4p \sim \omega_C$ and $h^0(C, 2p) = 1$. By the preceding computations, we have 
$$ \chi_1^{\log} (E_6) = 8, \quad \chi_2^{\log} (E_6) = 33. $$ 
These values are consistent with those in \cite[Section 3.1.6]{AFS16}. 

\subsubsection{$E_8$-singularity and $\bbP \calH(6)^{\even}$}
\label{subsec:E8}
Similarly, the $E_8$-singularity is given by 
$$\Spec \bbC [x, y] / (y^3 - x^5).$$ 
The corresponding semigroup and gap sequence are 
$H = \{ 3, 5, 6, 8, 9, \ldots \}$ and $G = \{1, 2, 4, 7 \}$. 
The locus $\bbP{\rm Def}^{-}_s (E_8)$ is the {\em even spin} component $\bbP \calH(6)^{\even}$ parameterizing $(C, p)$ where $C$ is a nonhyperelliptic curve of genus four with an even theta characteristic given by $\calO(3p)$. By the preceding computations, we have 
$$ \chi_1^{\log} (E_8) = 14, \quad \chi_2^{\log} (E_8) = 63. $$
These values are consistent with those in \cite[Section 3.1.6]{AFS16}. 

\subsubsection{Unibranched planar singularities}
\label{subsec:unibranch}
In general, consider unibranched planar singularities of type 
$$\Spec \bbC [t^{pi + qj : i, j \in \bbN}],$$ 
where $p$ and $q$ are coprime. Its $\bbP{\rm Def}^{-}_s$ corresponds to the locus in the stratum $\bbP \calH(2g-2)$, where $g = (p-1)(q-1)/2$, with the gap sequence 
$G = \{ b_1, \ldots, b_g\}$ consisting of positive integers that cannot be expressed as $p i + q j $ with $i, j \geq 0$. In this case, 
$$\sum_{i=1}^g b_i  = \frac{1}{12} (p-1)(q-1) (2pq - p - q - 1); $$
see \cite[Section 3.1.8]{AFS16} for more details and related references.  It follows that 
\begin{align*}
\chi_1^{\log} (Y) & = \chi_1(Y) = \frac{1}{12} (p-1)(q-1) (2pq - p - q - 1), \\
\chi_2^{\log} (Y) & = (pq - p - q)^2 + \frac{1}{12} (p-1)(q-1) (2pq - p - q - 1). 
\end{align*}

\subsection{$A_{2g+1}$-singularity and $\bbP \calH(g-1, g-1)^{\hyp}$}
\label{subsec:A2g+1}

This singularity is given by 
$$\Spec \bbC[x, y] / (y^2 - x^{2g+2}),$$ 
which has two branches.  It corresponds to the {\em hyperelliptic} 
component $\bbP \calH(g-1, g-1)^{\hyp}$ parameterizing $(C, p_1, p_2)$, where $C$ is hyperelliptic and $p_1, p_2$ are hyperelliptic conjugate points. 

In this case, we have 
 $$\ell = g, \quad a_1 = a_2 = 1, \quad \omega_C^{\log} \sim g (p_1 + p_2)$$ 
and 
$$\calF^{\lambda} H^0(C, m\omega_C^{\log}) = H^0(C, (mg - \lambda) (p_1 + p_2)).$$  

For $m = 1$, we have 
$$\dim  H^0(C, (g - \lambda) (p_1 + p_2)) / H^0(C, (g - \lambda - 1) (p_1 + p_2)) = 1\  \mbox{or} \ 0,$$ 
where the former occurs if and only if $\lambda \leq g$. We thus obtain that 
$$ \chi_1^{\log} (Y)  = \chi_1(Y) = \sum_{\lambda = 1}^{g} \lambda  = \frac{g(g+1)}{2}. $$

For $m = 2$, $\dim  H^0(C, (2g - \lambda) (p_1 + p_2)) / H^0(C, (2g - \lambda - 1) (p_1 + p_2)) = 2$, $1$, or $0$, which occurs if $\lambda \leq g-1$, $g\leq \lambda \leq 2g$, or $\lambda > 2g$, respectively. We thus obtain that 
\begin{align*}
 \chi_2^{\log} (Y) & = 2 \left(\sum_{\lambda = 1}^{g-1} \lambda\right) + \sum_{\lambda = g}^{2g} \lambda = \frac{5g^2 + g}{2}, \\
\chi_2 (Y) & = \chi_2^{\log} (Y) - (1+1) = \frac{5g^2 + g}{2}  - 2.  
\end{align*}

These values are consistent with those in \cite[Section 3.1.3]{AFS16}. 

\subsection{$D_{2g+1}$-singularity and $\bbP \calH(0,  2g-2)^{\hyp}$}
\label{subsec:D2g+1}

This singularity is given by 
$$\Spec \bbC[x, y] / x(y^2 - x^{2g-1}).$$ 
It corresponds to the stratum $\bbP \calH(0,  2g-2)^{\hyp}$ parameterizing $(C, p_1, p_2)$, where $C$ is hyperelliptic, $p_1$ is an ordinary marked point, and $p_2$ is a Weierstrass point. 

In this case, we have 
$$\ell = 2g-1, \quad a_1 = 2g-1, \quad a_2 = 1.$$ 
Since 
 $\omega_C^{\log} \sim p_1 + (2g-1)p_2$, we have 
$$\calF^{\lambda} H^0(C, m\omega_C^{\log}) = H^0\left(m \omega_C^{\log} - \left\lceil \frac{\lambda}{2g-1}\right\rceil p_1 - \lambda p_2\right). $$

For $m = 1$, if $\lambda \geq 2g$, then $\calF^{\lambda} H^0(C, \omega_C^{\log})  = 0$. Suppose $1\leq \lambda \leq 2g-1$. Then $\calF^{\lambda} H^0(C, \omega_C^{\log}) = H^0(C, (2g-1 - \lambda) p_2)$. Consequently, 
$\dim H^0(C, (2g-1 - \lambda) p_2) / H^0 (C, (2g-2 - \lambda) p_2) = 1$ or $0$, where the former occurs if and only if 
$\lambda = 1, 3, \ldots, 2g-1$. We thus obtain that 
$$  \chi_1^{\log} (Y)  = \chi_1(Y) = 1 + 3 + \cdots + (2g-1)  = g^2. $$

For $m = 2$, if $\lambda \geq 4g-1$, then $\calF^{\lambda} H^0(C, 2\omega_C^{\log})  = 0$. Next, suppose 
$2g \leq \lambda \leq 4g-2$. Then $\calF^{\lambda} H^0(C, 2\omega_C^{\log}) = H^0(C, (4g-2 - \lambda) p_2)$. Consequently,  $\dim H^0(C, (4g-2 - \lambda) p_2) / H^0(C, (4g- 3 - \lambda) p_2) = 1$ for $\lambda = 2g, 2g+2, \ldots, 4g-2$. Finally, suppose $1 \leq \lambda \leq 2g-1$. Then 
$\calF^{\lambda} H^0(C, 2\omega_C^{\log}) = H^0(C, p_1 + (4g-2 - \lambda) p_2)$ whose dimension is $3g-\lambda$. Consequently,  $\dim H^0(C, p_1 + (4g-2 - \lambda) p_2) / H^0(C, p_1 + (4g- 3 - \lambda) p_2) = 1$ for all $\lambda$ in this range. We thus obtain that 
\begin{align*}
 \chi_2^{\log} (Y) & = 2 \sum_{i=g}^{2g-1} i + \sum_{j=1}^{2g-1} j = 5g^2 - 2g, \\
\chi_2 (Y) & = \chi_2^{\log} (Y) - (2g-1 + 1) = 5g^2 - 4g. 
\end{align*}

These values are consistent with those in \cite[Section 3.1.4]{AFS16}. 

\subsection{$D_{2g+2}$-singularity and $\bbP \calH(0,g-1,g-1)^{\hyp}$}
\label{subsec:D2g+2}

This singularity is given by 
$$\Spec \bbC[x, y] / x(y^2 - x^{2g}).$$ 
It corresponds to the stratum $\bbP \calH(0,g-1,g-1)^{\hyp}$ parameterizing $(C, p_1, p_2, p_3)$, where $C$ is hyperelliptic, $p_1$ is an ordinary marked point, and $p_2, p_3$ are hyperelliptic conjugate points. 

In this case, 
$$\ell = g, \quad a_1 = g, \quad a_2 =  a_3 = 1.$$ 
Since 
 $\omega_C^{\log} \sim p_1 + g (p_2+p_3)$, we have 
$$\calF^{\lambda} H^0(C, m\omega_C^{\log}) =  H^0\left(m \omega_C^{\log} - \left\lceil \frac{\lambda}{g}\right\rceil p_1 - \lambda (p_2+p_3)\right). $$ 

For $m = 1$, if $\lambda \geq g+1$, then $\calF^{\lambda} H^0(C, \omega_C^{\log}) = 0$. Next, suppose $1\leq \lambda \leq g$. Then 
$\calF^{\lambda} H^0(C, \omega_C^{\log}) = H^0(C, (g- \lambda) (p_1 + p_2))$ which is of dimension $g-\lambda + 1$.  
Consequently, $\dim H^0(C, (g- \lambda) (p_1 + p_2)) / H^0(C, (g- \lambda - 1) (p_1 + p_2)) = 1$ for all $\lambda$ in this range. We thus obtain that 
$$  \chi_1^{\log} (Y)  = \chi_1(Y) = \sum_{\lambda = 1}^g \lambda = \frac{g^2 + g}{2}. $$

For $m = 2$, if $\lambda \geq 2g+1$, then $\calF^{\lambda} H^0(C, 2\omega_C^{\log}) = 0$. Next, we suppose 
$g+1 \leq \lambda \leq 2g$. Then $\calF^{\lambda} H^0(C, 2\omega_C^{\log}) = H^0(C, (2g-\lambda)(p_2 + p_3))$ 
which is of dimension $2g-\lambda + 1$. Consequently, $\dim \calF^{\lambda} H^0(C, 2\omega_C^{\log}) / \calF^{\lambda+1} H^0(C, 2\omega_C^{\log}) = 1$ for all $\lambda$ in this range. Finally, consider $1\leq \lambda \leq g$. Then 
$\calF^{\lambda} H^0(C, 2\omega_C^{\log}) = H^0(C, p_1 + (2g-\lambda)(p_2 + p_3))$ which is of dimension 
$3g - 2\lambda + 2$. Consequently, $\dim \calF^{\lambda} H^0(C, 2\omega_C^{\log}) / \calF^{\lambda+1} H^0(C, 2\omega_C^{\log}) = 2$ for all $\lambda$ in this range. We thus obtain that 
\begin{align*}
 \chi_2^{\log} (Y) & =  \sum_{\lambda = g+1}^{2g} \lambda + 2 \sum_{\lambda = 1}^g \lambda = \frac{5g^2 + 3g }{2}, \\
 \chi_2 (Y) & = \chi_2^{\log} (Y) - (g + 1 + 1) = \frac{5g^2 + g - 4}{2}.   
\end{align*}

These values are consistent with those in \cite[Section 3.1.5]{AFS16}. 

\subsection{$E_7$-singularity and $\bbP \calH(1,3)$}
\label{subsec:E7}

This singularity is given by 
$$\Spec \bbC[x, y] / y (y^2 - x^{3}).$$ 
It corresponds to the stratum $\bbP \calH(1,3)$ where 
$p_1 + 3p_2 \sim \omega_C$. 

In this case, we have 
$$\ell = 4, \quad a_1 = 2, \quad a_2 = 1.$$ 
It follows that 
$$\calF^{\lambda} H^0(C, m\omega_C^{\log}) = H^0\left( \left(2m -  \left\lceil \frac{\lambda}{2}\right\rceil p_1\right) + (4m - \lambda) p_2\right). $$

For $m = 1$, we have the filtration of $H^0(C, \omega_C^{\log}) = H^0(C, 2p_1 + 4p_2)$ given by 
\begin{align*} 
 & \calF^{0}\colon H^0(C, 2p_1 + 4p_2) \cong \bbC^4 \supsetneq \calF^{1}\colon H^0(C, p_1 + 3p_2) \cong \bbC^3 \supsetneq  \calF^{2}\colon H^0(C, p_1 + 2p_2) \\
 & \cong \bbC^2  \supsetneq \calF^{3}\colon H^0(C, p_2) = \calF^{4}\colon H^0(C, \calO) \cong \bbC \supsetneq  \calF^{5}  = 0. 
\end{align*}
The nontrivial weights are $1$, $2$, and $4$. We thus obtain that  
$$ \chi_1^{\log}(E_7) = \chi_1 (E_7) = 7. $$

For $m = 2$,  the filtration of $H^0(C, 2\omega_C^{\log}) = H^0(C, 4p_1 + 8p_2)$ is given by 
\begin{align*} & \calF^{0} \colon H^0(C, 4p_1 + 8p_2) \cong \bbC^{10} \supsetneq \calF^{1} \colon H^0(C, 3p_1 + 7p_2) \cong \bbC^8 \supsetneq  \calF^{2}\colon H^0(C, 3p_1 + 6p_2) \cong \bbC^7 \\
& \supsetneq \calF^{3} \colon H^0(C, 2p_1 + 5p_2) \cong \bbC^5 \supsetneq  \calF^{4} \colon H^0(C, 2p_1 + 4p_2) \cong \bbC^4  \supsetneq  \calF^{5} \colon H^0(C, p_1 + 3p_2) \cong \bbC^3   \\ 
& \supsetneq \calF^{6} \colon H^0(C, p_1 + 2p_2) \cong \bbC^2 \supsetneq  \calF^{7} \colon H^0(C, p_2) =\calF^{8} \colon H^0(C, \calO) \cong \bbC \supsetneq \calF^9 = 0. 
\end{align*}
The nontrivial weights are $1$, $2$, $2$, $3$, $4$, $5$, $6$, and $8$. We thus obtain that  
$$ \chi_2^{\log}(E_7) = 31, \quad \chi_2 (E_7) = 31 - (2+1) = 28. $$

These values are consistent with those in \cite[Section 3.1.6]{AFS16}. 

\subsection{Elliptic $n$-fold points and $\bbP \calH(0^{n})$}
\label{subsec:elliptic}

For $n\geq 3$, this singularity is given by the union of $n$ general lines through a point in $\bbA^{n-1}$.  The corresponding stratum is $\bbP \calH(0^{n})$ in genus one, where $0^n$ denotes $n$ ordinary marked points, i.e., zeros of order $0$.  

In this case, we have $\ell = 1$ and $a_i = 1$ for all $i$. It follows that 
$$\calF^{\lambda} H^0(C, m\omega_C^{\log}) = H^0(C, (m- \lambda) (p_1 + \cdots + p_n))$$ 
which is of dimension $0$, $1$, or $(m-\lambda)n$ for $\lambda > m$, $\lambda = m$, or $\lambda < m$, respectively. 

For $m = 1$, we have  
$$\dim \calF^{\lambda} H^0(C, \omega_C^{\log}) / \calF^{\lambda+1} H^0(C, \omega_C^{\log}) = 1$$ 
for $\lambda = 1$. We thus obtain that 
$$ \chi_1^{\log}(Y) = \chi_1(Y) = 1. $$

For $m = 2$, we have 
$$\dim \calF^{\lambda} H^0(C, 2\omega_C^{\log}) / \calF^{\lambda+1} H^0(C, 2\omega_C^{\log}) = n-1 \ \mbox{or}\ 1$$ 
for $\lambda = 1$ or $2$, respectively. We thus obtain that 
\begin{align*}
   \chi_2^{\log}(Y) & = n+1, \\
\chi_2(Y) & = (n+1) - n = 1. 
\end{align*}

The obtained values are consistent with those in \cite[Section 3.1.9]{AFS16}. 

\subsection{The principal strata}
\label{subsec:principal}

Consider $(C, p_1, \ldots, p_{2g-2})\in \bbP \calH (1^{2g-2})$. In this case, $\ell = 2$ and $a_i = 1$ for $i =1, \ldots, 2g-2$. Therefore, we have 
$$\calF^{\lambda} H^0(C, m\omega_C^{\log}) = H^0(C, (2m - \lambda) (p_1 + \cdots + p_{2g-2})) = H^0 (C, (2m-\lambda)K).$$ 
Note that $h^0(C, (2m-\lambda)K) = 0$, $1$, $g$, or $(4m-2\lambda - 1) (2g-2)$ 
 if $\lambda \geq 2m+1$, $\lambda = 2m$, $\lambda = 2m-1$, or $\lambda \leq 2m-2$, respectively. 
 
 For $m = 1$, we have 
 $\dim \calF^{\lambda} H^0(C, \omega_C^{\log}) / \calF^{\lambda+1} H^0(C, \omega_C^{\log}) = g-1$ or $1$ for $\lambda = 1$ or $2$, respectively. We thus obtain that 
 $$ \chi_1^{\log} (Y) =  \chi_1 (Y) = g+1. $$
 
 For $m = 2$, we have 
 $\dim \calF^{\lambda} H^0(C, 2\omega_C^{\log}) / \calF^{\lambda+1} H^0(C, 2\omega_C^{\log}) = 2g-2$, $2g-3$, $g-1$, or $1$ for $\lambda = 1$, $2$, $3$, or $4$, respectively. We thus obtain that 
 \begin{align*}
   \chi_2^{\log} (Y) & = (2g-2)\cdot 1 + (2g-3)\cdot 2 + (g-1)\cdot 3 + 1 \cdot 4 = 9g - 7, \\
  \chi_2 (Y) & = \chi_2^{\log} (Y) - (2g-2) = 7g -5. 
  \end{align*}
  
  \begin{remark}
  \label{rem:cross-ratio}
  Although all $(Y,q)\in S(1^{2g-2})$ behave the same numerically, $S(1^{2g-2})$ can contain infinitely many non-isomorphic singularity classes. For example, for $g = 3$, consider the canonical divisor $p_1 + p_2 + p_3 + p_4$ cut out by a line $L$ with a plane quartic curve $C$. Through the construction of the deformation to the normal cone for the divisor at infinity $p_1 + p_2 + p_3 + p_4$, the resulting $Y$ consists of four lines joining each $p_i$ to a common point $q\in \bbP^2$. Hence, the {\em cross-ratio} of $p_1$, $p_2$, $p_3$, and $p_4$ in $L\cong \bbP^1$ is an invariant for the singularity $(Y, q)$; see \cite[Remark 2.1]{B24} for more details. In this case, $\bbP\calH(1,1,1,1)$ admits a slicing over the one-dimensional base of cross-ratios. We will describe the corresponding singularities and their deformations in detail in Section~\ref{subsec:(1,1,1,1)}.  
\end{remark}

\subsection{Toric singularities}
\label{subsec:toric}

Consider the singularity given by 
$$\Spec \bbC [x, y] / (x^p - y^q),$$ 
where $p, q \geq 2$. Let $b = \gcd (p, q)$, $p_0 = p / b$ and $q_0 = q/b$. The case $b = 1$ is a unibranch planar singularity which was analyzed earlier. The $\bbG_m$-action is given by $\lambda \cdot (x, y) = ( \lambda^{q_0}x, \lambda^{p_0} y)$. The corresponding $\bbP{\rm Def}^{-}_s$ lies in the stratum 
$\bbP \calH ( k^b )$, where the genus is determined by  
$$ 2g - 2 = pq - p - q - b  = b^2 p_0 q_0 - bp_0 - bq_0 - b, $$
$k =  b p_0 q_0 - p_0 - q_0 - 1$, and $ \omega_C \sim k (p_1 + \cdots + p_b)$; see~\cite[Theorem 6.3]{H00} and \cite[Section 4.2]{AFS16} for more details and related references. 

In this case, we have $\ell = k+1$ and $a_i = 1$ for all $i$. It follows that  
$$\calF^{\lambda} H^0(C, m\omega_C^{\log}) = H^0 (C, ((k+1)m - \lambda) (p_1 + \cdots + p_b)). $$

For $m = 1$, if $k+1 - \lambda = 0, 1, \ldots, k$, namely, if $\lambda = 1, \ldots, k, k+1$, then the dimension of $\calF^{\lambda} H^0(C, \omega_C^{\log})$ is 
$$\# \{ (i, j) : i, j \geq 0\ {\rm and}\ qi + pj \leq (k+1 - \lambda) b \}.$$  
It follows that 
\begin{align*}
 \chi_1^{\log} (Y) & =  \chi_1 (Y)  \\
& =    \sum_{\lambda = 1}^{k+1} \lambda \cdot \Big(\# \{ (i, j) : i, j \geq 0\ {\rm and}\ qi + pj \leq (k+1 - \lambda) b \} \\
& \quad - \# \{ (i, j) : i, j \geq 0\ {\rm and}\ qi + pj \leq (k - \lambda) b \} \Big)  \\
 & =   \sum_{\lambda = 1}^{k+1} \# \{ (i, j) : i, j \geq 0\ {\rm and}\ qi + pj \leq (k+1 - \lambda) b \} \\
  & =  \sum_{n = 0}^{k} \# \{ (i, j) : i, j \geq 0\ {\rm and}\ qi + pj \leq n b \} 
 \end{align*}
 where $n = k+1 - \lambda$ in the last identity. By \cite[Proposition 4.3]{AFS16}, this yields an interesting combinatorial identity 
 \begin{align*}
   & \frac{1}{12b} \big( (pq - p - q)^2 + pq (pq - p - q + 1) - b^2 \big)  \\
   = & \sum_{n = 0}^{\frac{1}{b}(pq - p - q - b)} \# \{ (i, j) : i, j \geq 0\ {\rm and}\ qi + pj \leq n b \}. 
   \end{align*}
 
 For $m = 2$, if $2k+2 - \lambda = 0, 1, \ldots, k$, namely, if $\lambda = k+2, \ldots, 2k+2$, then 
 the dimension of $\calF^{\lambda} H^0(C, 2\omega_C^{\log})$ is 
$$\# \{ (i, j) : i, j \geq 0\ {\rm and}\ qi + pj \leq (2k+2 - \lambda) b \}.$$  
If $2k+2 - \lambda \geq k+1$, namely, $\lambda \leq k+1$, then 
the dimension of $\calF^{\lambda} H^0(C, 2\omega_C^{\log})$ is $1 - g + (2k+2 - \lambda)b$. From these values, 
one can also compute $\chi_2^{\log}(Y)$ explicitly. 
  
\section{Numerical properties of weights and characters}
\label{sec:numerical}

In this section, we analyze the numerical properties for the weights and characters of the $\bbG_m$-action on $H^0(C, m \omega_C^{\log})$ for $(C, p_1, \ldots, p_n) \in \bbP\calH(\mu)$. Specifically, we will show that the weights for $m = 1$ can determine the weights for all $m \geq 2$. Additionally, for later use, we review the relation of the characters of the $\bbG_m$-action and the $\alpha$-invariants of the singularities in the context of the log minimal model program for $\BM_g$. 

\subsection{Weights and characters}
\label{subsec:numerical}

Recall that for $\mu = (m_1, \ldots, m_n)$, we have 
$$\ell = \lcm (m_1 + 1, \ldots, m_n + 1)\quad {\rm and}\quad a_i = \frac{\ell}{m_i+1}$$
for all $i$. For $\lambda \geq 0$ and $i = 1,\ldots, n$, we define 
$$ l_{\lambda, i} \coloneqq \left\lceil \frac{\lambda}{a_i} \right\rceil. $$
Then we have 
$$ \calF^{\lambda} H^0(C, m\omega_C^{\log})  = H^0 \left(C, m \omega_C^{\log} - \sum_{i=1}^n l_{\lambda, i} p_i\right). $$
Denote by $ N_{m, \lambda}$ the number of weights equal to $\lambda$ for the $\bbG_m$-action on $H^0(C, m \omega_C^{\log})$. It follows that 
 \begin{align}
 \label{eq:weight}
  N_{m, \lambda} &=  h^0 \left(C, m \omega_C^{\log} - \sum_{i=1}^n l_{\lambda, i} p_i\right) - h^0 \left(C, m \omega_C^{\log} - \sum_{i=1}^n l_{\lambda+1, i} p_i\right). 
  \end{align}
  
In particular, for $m = 1$ and $\lambda = 0$, we have 
$$ N_{1,0} = h^0(C, \omega_C^{\log}) - h^0 (C, \omega_C) = n-1. $$
 Therefore, the weights of the $\bbG_m$-action on $H^0(C, \omega_C^{\log})$ are of the following form: 
 $$0, \ldots, 0, w_1, \ldots, w_{g} $$
where they are listed with multiplicities and the first $n-1$ weights are $0$. 

For $\lambda \geq 1$, we have $l_{\lambda, i} \geq 1$. It follows that 
$$H^0\left(C, \omega_C^{\log} - \sum_{i=1}^n l_{\lambda, i} p_i\right) = H^0\left(C, \omega_C - \sum_{i=1}^n (l_{\lambda, i}-1) p_i\right). $$
Then, $w_1, \ldots, w_g$ can be identified with the weights of the induced $\bbG_m$-action on $H^0 (C, \omega_C)$. 

Below we will show that $w_1, \ldots, w_g$ determine completely the weights of the $\bbG_m$-action on $H^0 (C, m\omega_C^{\log})$ for all $m$. 

\begin{proposition}
\label{prop:weight}
For $m \geq 2$, the weights and characters of the $\bbG_m$-action on $H^0 (C, m\omega_C^{\log})$ can be described as follows:  
\begin{enumerate}
\item[\rm (i)] The number of weights equal to $(m-1) \ell$ is $N_{m, (m-1) \ell} = n-1$. 
\item[\rm (ii)] For $ 0 \leq \lambda < (m-1) \ell$, the number of weights equal to $\lambda$ is $N_{m, \lambda} = \sum_{i=1}^n (l_{\lambda+1, i} - l_{\lambda, i})$. 
\item[\rm (iii)] The weights that are larger than $(m-1) \ell$ are $(m-1)\ell+w_1, \ldots, (m-1)\ell + w_g$.  
\item[\rm (iv)] $\chi_m^{\log} = \frac{1}{2} (m-1)m (2g-2+n) \ell +  \sum_{i=1}^g w_i$. 
\end{enumerate}
\end{proposition}

\begin{proof}
Note that $l_{(m-1) \ell, i} = (m-1)(m_i + 1)$ and $l_{(m-1) \ell+1, i} = (m-1)(m_i + 1) + 1$. 
Therefore, we have 
$$m \omega_{C}^{\log} - \sum_{i=1}^n l_{(m-1) \ell, i} p_i \sim \omega_C^{\log} \quad {\rm and }\quad m \omega_{C}^{\log} - \sum_{i=1}^n l_{(m-1)\ell+1, i} p_i \sim \omega_C. $$ 
It follows that $N_{m, (m-1)\ell} = n-1$.  

For $\lambda < (m-1)\ell$, we have 
$$\deg m \omega_{C}^{\log} - \sum_{i=1}^n l_{\lambda+1, i} p_i \geq \deg \omega_C^{\log} > 2g-2.$$ 
The Riemann--Roch formula implies that 
$$ N_{m, \lambda} = \sum_{i=1}^n l_{ \lambda+1, i} -  \sum_{i=1}^n l_{\lambda, i}. $$

For $\lambda > (m-1)\ell$, we have 
$$ H^0 \left(C, m \omega_C^{\log} - \sum_{i=1}^n l_{ \lambda, i} p_i\right) = H^0\left(C, \omega_C^{\log} - \sum_{i=1}^n (l_{ \lambda, i} - (m-1)\ell) p_i\right). $$ 
Therefore, the weights that are larger than $(m-1)\ell$ are given by $(m-1)\ell+w_1, \ldots, (m-1)\ell + w_g$. 

Finally, we have 
\begin{align*}
\chi_m^{\log} & = (n-1)(m-1)\ell + \sum_{\lambda=0}^{(m-1)\ell - 1} \sum_{i=1}^n \lambda (l_{\lambda+1, i} - l_{\lambda, i}) + g(m-1)\ell + \sum_{i=1}^g w_i \\
& = (m-1) \left((g+n-1)\ell +  ((m-1)\ell-1)(2g-2+n)\right) -\sum_{i=1}^n\sum_{\lambda = 1}^{(m-1)\ell -1} l_{\lambda, i}  + \sum_{i=1}^g w_i \\
& = (m-1) \left((g+n-1)\ell +  ((m-1)\ell-1)(2g-2+n)\right) \\
&\quad - \sum_{i=1}^n (1 + 2 + \cdots + (m-1)(m_i+1)) \frac{\ell}{m_i+1} +  (m-1)(2g-2+n)  + \sum_{i=1}^g w_i  \\
& = \frac{1}{2} (m-1)m  (2g-2+n) \ell +  \sum_{i=1}^g w_i. 
\end{align*}
\end{proof}

\subsection{The $\alpha$-invariants}
\label{subsec:alpha}

Given a singularity $(Y, q)\in S(\mu)$, its {\em $\alpha$-invariant} is determined by the following relation; see \cite[Definition 1.4]{AFS16}: 
\begin{align}
\label{eq:alpha}
 \frac{\chi_2^{\log}}{\chi_1^{\log}} & = \frac{13 (1-\alpha)}{2-\alpha}. 
 \end{align}
Geometrically speaking, it predicts that the singular curve formed by gluing a smooth curve nodally with {\em every} branch of $Y$ appears in the log minimal model program for the moduli space of curves at the stage of $K + \alpha \delta$; see \cite[Proposition 1.3]{AFS16}. As $\alpha$ decreases, the types of singularities become more and more complicated. Therefore, it is meaningful to classify the singularities whose $\alpha$-invariants are bounded from below by a given critical value. 

To this end, we will prove the classification result of Theorem~\ref{thm:3/8} in Section~\ref{sec:3/8}. Before we do that, we first explain a variant in the process of gluing a smooth curve nodally to $Y$. If we glue it to every branch of $Y$, then we say that this is the case {\em without dangling branches}, as named in \cite{AFS16} and also called {\em $\widehat{\calO}$-atoms} in \cite[Definition 2.6]{AFS16}. Alternatively, one can choose a subset $Q\subset \{1, \ldots, n\}$ such that each branch of $Y$ with index in $Q$ is not attached nodally to the rest of the curve, where such a loose branch is called {\em dangling} and this kind of nodal attaching is called an {\em $\widehat{\calO}^Q$-atom} in \cite[Definition 2.9]{AFS16}.  

Indeed, knowing the weights and characters in the case without dangling branches determine all other dangling cases: $\chi_1^{\log} = w_1 + \cdots + w_g$ remains the same, while $\chi_2^{\log}$ is modified by subtracting $\sum_{i\in Q}\frac{\ell}{m_i+1}$; see~\cite[Corollary 3.3 and Corollary 3.6]{AFS16}. Therefore, we choose to focus on the case 
without dangling branches, and one can use the aforementioned relation to recover the dangling cases as well.  

\section{Singularities of the nonvarying strata}
\label{sec:nonvarying}

A stratum $\bbP \calH(\mu)$ (or one of its connected components) is called {\em nonvarying}, if all Teichm\"uller curves contained in it have the same slope; equivalently, if they have the same area Siegel--Veech constant, or the same sum of nonnegative Lyapunov exponents. The hyperelliptic strata and a number of strata in low genus are known to be nonvarying; see \cite{CM12, YZ13}. They are also expected to be the only nonvarying strata based on numerical evidences. Additionally, the nonvarying strata possess special properties in affine geometry and intersection theory; see \cite{C24}. From the viewpoint of this paper, it is natural to speculate that a stratum $\bbP \calH(\mu)$ is nonvarying if and only if $S(\mu)$ consists of a {\em unique} isomorphism class $(Y,q)$ of Gorenstein singularities with $\bbG_m$-action. In other words, in this case $\bbP \calH(\mu)$ would be the locus of smooth deformations in the miniversal deformation space of $(Y,q)$.  

In this section, we prove the above speculation 
for all known nonvarying strata in \cite{CM12, YZ13}, as recalled in Theorem~\ref{thm:nonvarying}. Additionally, we classify the unique singularity class for each of them, describe their defining equations as well as weighted projective embeddings, and compute their weights, characters, and $\alpha$-invariants. Finally, we prove the remaining parts of Theorem~\ref{thm:nonvarying} in Theorem~\ref{thm:T2=0} and Theorem~\ref{thm:hyp-nonvarying}, respectively, as well as prove Theorem~\ref{thm:US} through Theorem~\ref{thm:ordinary}. 

By the descriptions in Section~\ref{sec:ex}, the uniqueness claim holds for the hyperelliptic strata: $\bbP \calH(2g-2)^{\hyp}$ and $\bbP \calH(g-1, g-1)^{\hyp}$, as well as for the minimal strata up to genus five: $\bbP\calH(4)^{\odd}$, $\bbP\calH(6)^{\odd}$, $\bbP\calH(6)^{\even}$, $\bbP\calH(8)^{\odd}$, and $\bbP\calH(8)^{\even}$, where the Weierstrass semigroup of the unique zero is nonvarying in each case.  
Therefore, it suffices to consider the remaining nonvarying strata in \cite[Figures 3 and 4]{CM12}: 
\begin{align*}
g = 3 \colon & (3,1), (2,2)^{\odd}, (2,1,1), \\
g = 4\colon & (5,1), (4,2)^{\even}, (4,2)^{\odd}, (3,3)^{\nonhyp}, (3,2,1), (2,2,2)^{\odd}, \\
g = 5\colon & (6,2)^{\odd}, (5,3). 
\end{align*}

We remark that the claim for the nonvarying strata in genus three also follows from the classification in \cite[Section 2]{B24} (by setting the $\bbG_a$-parameters to be zero in the {\em crimping data} in order to have the $\bbG_m$-action well-defined). For completeness, we include the discussion for them as well.  

To this end, we utilize the description in \cite[Section 1.3]{B24}. Suppose $(R, \mathfrak{m})$ is the complete local ring of the concerned isolated curve singularity $(Y, q)$ with $b$ branches. Let $\delta$ be the {\em $\delta$-invariant} of the singularity. The genus, $\delta$-invariant, and the number of branches satisfy the following relation: 
$$ g = \delta - b + 1. $$

Let $(\wR, \wfkm)\cong (\bbC[\![t_1]\!] \oplus \cdots \oplus \bbC[\![t_b]\!], \langle t_1, \ldots, t_b \rangle )$ be the normalization of $(R, \mathfrak{m})$, where each $t_i$ is the parameter on the corresponding branch $Y_i$. Then $\wR/R$ is a $\bbZ$-graded $R$-module, where the $i$th graded piece is 
$$ (\wR / R)_i \coloneqq \wfkm^i / ( \wfkm^i \cap R) + \wfkm^{i+1}. $$
The following observations hold; see~\cite[Section 1.3.4]{B24}, \cite[Section 2]{B22}, and \cite[Appendix A]{S11}: 
\begin{enumerate}[(1)]
\item $\delta = \sum_{i\geq 0} \dim_{\bbC} (\wR/R)_i$. 
\item $g = \sum_{i\geq 1} \dim_{\bbC} (\wR/R)_i$. 
\item If $(\wR/R)_i = (\wR/R)_j = 0$, then $(\wR/R)_{i+j} = 0$. 
\item $ \sum_{i\geq j} (\wR/R)_i$ is a grading of $\wfkm^j / (\wfkm^j \cap R)$. 
\item There are short exact sequences of $\bbC$-modules: 
$$ 0 \to \frac{\wfkm^i \cap R}{\wfkm^{i+1}\cap R} \to \frac{\wfkm^i}{\wfkm^{i+1}}\to (\wR/R)_i \to 0. $$
\end{enumerate}

Let $\alpha_i = \dim_{\bbC} (\wR / R)_i$ for $i\geq 1$. The sequence 
$$[\alpha_1\ \alpha_2 \ \alpha_3 \cdots]$$ 
(or the subsequence consisting of the nonzero entries) is called the {\em gap sequence} of the singularity, where 
$$\sum_{i\geq 1} \alpha_i = g.$$ 
Moreover, if $\alpha_i = 0$ and $\alpha_j = 0$, then $\alpha_{i+j} = 0$. We caution the reader not to confuse the gap sequence of a singularity with the Weierstrass gap sequence of a smooth point.  

Let $\fkc = {\rm Ann}_{R}(\wR / R) $ denote the {\em conductor ideal}. The singularity being Gorenstein is equivalent to 
$$ {\rm len} (R / \fkc) = {\rm len} (\wR / R) = \delta. $$
 
If $(Y, q)\in S(\mu)$ for $\mu = (m_1, \ldots, m_n)$ is Gorenstein with $\bbG_m$-action, the following additional observations hold: 
 \begin{enumerate}
\item[(G1)] $t_i\not\in R$ for all $i$.  
\item[(G2)] If $t_i$ and $t_j$ appear in a (non-decomposable) generator of $\fkm / \fkm^2$, then their exponents are proportional as  
$m_i + 1 : m_j + 1$. In particular, each $t_i$ can appear at most one time in a generator of $\fkm / \fkm^2$.
\item[(G3)] The germ of the dualizing bundle $\omega_Y$ at $q$ is generated by $dt_i / t_i^{m_i+2}$ along each branch $Y_i$. In particular, $\wfkm^{\max\{m_i+2\}_{i=1}^n}\subset R$. 
\item[(G4)] Up to rescaling each $t_i$, suppose $\sum_{i=1}^n u_i (dt_i / t_i^{m_i+2})$ is a local generator of $\omega_Y$ at $q$, where $u_i \neq 0$ for all $i$. Then $ u_j t_i^{m_i+1} - u_i t_j^{m_j+1} \in R$. 
\item[(G5)] The last nonzero entry in the gap sequence is $\alpha_{\max\{m_i+1\}_{i=1}^n} = 1$. 
\end{enumerate}

In the above, (G1) follows from the fact that a Gorenstein curve singularity, except for the node, is indecomposable; see \cite[Proposition 2.1]{AFS16}. (G2) follows from the fact that the $\bbG_m$-action on the branch of $t_i$ has weight $\ell / (m_i+1)$, where $\ell = \lcm (m_1+1, \ldots, m_n+1)$. The first part of (G3) has been verified in the proof of Theorem~\ref{thm:filtration} (iv). The second part of (G3) as well as (G4) and (G5) follow from the fact that the duality residue pairing is perfect, whose germs at the singularity form a graded pairing.  

In what follows, we adapt the strategy in \cite[Section 2]{B24} to describe the gap sequence and the generators of $\fkm / \fkm^2$ in each case.  

\subsection{The stratum $\bbP \calH(3,1)$}
\label{subsec:(3,1)}

This case was already studied in Section~\ref{subsec:E7} and the resulting singularity is $E_7$. Below we will use the above approach to provide an alternative argument. 

In this case, we have $g = 3$,  $b = 2$, and $\delta = 4$. According to the preceding observations, the gap sequence is determined by $[\alpha_1 \ \alpha_2\ \alpha_3\ 1]$, where $\alpha_4 = 1$ is the last nonzero entry and $\sum_{i\geq 1}\alpha_i = 3$. Since $\alpha_4 \neq 0$, it implies that $\alpha_1\neq 0$ and $\alpha_2\neq 0$, and hence the gap sequence is $[1\ 1\ 0\ 1]$. 

We want to describe the generators of $\fkm / \fkm^2 \pmod{\wfkm^5}$. Note that the exponents of nonzero terms of $t_1$ and $t_2$ in every (non-decomposable) generator are proportional as $4:2 = 2:1$. Moreover, $t_1^4\oplus 0\not\in R$ and $0\oplus t_2^2 \not\in R$. 
Therefore, up to rescaling the parameters, we can assume that the linear generator is $t_1^2 \oplus t_2\in R$ and another cubic generator is $t_1^3\oplus 0\in R$. 

In summary, $\fkm / \fkm^2 \pmod{\wfkm^5}$ can be generated by  
\begin{align*}
x & = t_1^2 \oplus t_2, \\
y & = t_1^3\oplus 0. 
\end{align*}
 In terms of these generators, the defining equation of this singularity is
$$ y (y^2- x^3). $$

Since $x$ and $y$ have weights $3$ and $2$ under the $\bbG_m$-action, respectively, the standard projectivization of the corresponding Gorenstein curve $Y$ as well as its deformations $(C, p_1, p_2)\in \bbP \calH(3,1)$ can be embedded in the weighted projective space $\bbP(3,2,1)$, where the divisor at infinity introduced in~\eqref{eq:infinity} is cut out by the hyperplane of weight $1$ as $ p_1 + \frac{1}{2} p_2$. 

Note that $R/{\fkc}$ is spanned by $\langle 1, x, y, x^2\rangle $. It has dimension $4 = \delta$, which verifies that this singularity is Gorenstein. The germ of the dualizing bundle at $q$ is generated by $dt_1 / t_1^5 - dt_2 / t_2^3$.  

The defining equation of this singularity coincides with the description in Section~\ref{subsec:E7}. In this case, we have 
$$ \chi_1^{\log}(3,1) = 7, \quad  \chi_2^{\log}(3,1) = 31, \quad \alpha(3,1) = \frac{29}{60}. $$

\subsection{The stratum $\bbP \calH(2,2)^{\odd}$}
\label{subsec:(2,2)}

In this case, we have $g = 3$, $b = 2$, and $\delta = 4$. The gap sequence is determined by $[\alpha_1 \ \alpha_2\ 1]$, where $\alpha_3 = 1$ is the last nonzero entry and $\sum_{i\geq 1}\alpha_i = 3$. Let $y_i$ be the point at infinity in the rational branch $Y_i$ for $i = 1, 2$. The spin parity of this case is determined by $\dim H^0(Y, y_1 + y_2)\pmod{2}$, where $t_1^{n_1} \oplus t_2^{n_2} \in H^0(Y, y_1 + y_2)$ if and only if $n_1 \leq 1$ and $n_2 \leq 1$. Since we want $H^0(Y, y_1 + y_2)$ to be one-dimensional spanned by the constant function only, there is no linear generator in $R$. Consequently, $\alpha_1 = 2$ and the gap sequence is 
$[2 \ 0\ 1]$. 

We want to describe the generators of $\fkm / \fkm^2 \pmod{\wfkm^4}$. Note that the exponents of nonzero terms of $t_1$ and $t_2$ in every (non-decomposable) generator are proportional as $3:3 = 1:1$. Since $\alpha_2 = 0$, we have the quadratic generators $t_1^2\oplus 0\in R$ and $0\oplus t_2^2\in R$. Up to rescaling the parameters, we can assume that $t_1^3 \oplus t_2^3\in R$ is the cubic generator. 
 
In summary, $\fkm / \fkm^2 \pmod{\wfkm^4}$ can be generated by  
\begin{align*}
x & = t_1^2 \oplus 0, \\
y & = 0 \oplus t_2^2, \\
z & = t_1^3\oplus t_2^3. 
\end{align*}
 In terms of these generators, the defining equations of this singularity are 
 $$ ( xy, x^3 + y^3 - z^2 ). $$ 
 This singularity is denoted by $T_7$ in Giusti's classification of simple space curve singularities; see~\cite[(7.22)]{L84Book} and \cite[Table 2]{S15}. 
 
 Since $x$, $y$, and $z$ have weights $2$, $2$, and $3$ under the $\bbG_m$-action, respectively, the standard projectivization of the corresponding Gorenstein curve $Y$ as well as its deformations $(C, p_1, p_2)\in \bbP \calH(2,2)^{\odd}$ can be embedded in the weighted projective space $\bbP(2,2,3,1)$, where the divisor at infinity is cut out by the hyperplane of weight $1$ as $ p_1 +  p_2$. 
 
 Note that $R/{\fkc}$ is spanned by $\langle 1, x, y, z\rangle $. It has dimension $4 = \delta$, which verifies that this singularity is Gorenstein. The germ of the dualizing bundle at $q$ is generated by $dt_1 / t_1^4 - dt_2 / t_2^4$.  

 Finally, we compute the weights and characters for this singularity by using~\eqref{eq:filtration}. In this case, we have $\ell = 3$ and $a_1 = a_2 = 1$. Since $\omega_C^{\log}\sim 3p_1 + 3p_2$, it follows that 
$$ \calF^{\lambda} H^0(C, \omega_C^{\log}) = H^0(C,  (3 - \lambda)p_1 + (3 - \lambda) p_2).$$  
Therefore, the filtration of $H^0(C, \omega_C^{\log})$ is given by 
\begin{align*} 
& \calF^{0}\colon H^0(C, 3p_1 + 3p_2)  \cong \bbC^4 \supsetneq \calF^{1}\colon H^0(C, 2p_1 + 2p_2) \cong \bbC^3 \\
& \supsetneq  \calF^{2} \colon H^0(C, p_1 + p_2) =  \calF^{3} \colon H^0(C, \calO) \cong \bbC   \supsetneq \calF^4 = 0. 
\end{align*}
The nontrivial weights are $1$, $1$, and $3$. Combining with Proposition~\ref{prop:weight} (iv) and~\eqref{eq:alpha},  
we thus obtain that 
$$ \chi_1^{\log}(2,2)^{\odd} = 5, \quad  \chi_2^{\log}(2,2)^{\odd} = 23, \quad \alpha(2,2)^{\odd} = \frac{19}{42}. $$

\subsection{The stratum $\bbP \calH(2,1,1)$}
\label{subsec:(2,1,1)}

In this case, we have $g = 3$,  $b = 3$, and $\delta = 5$. The gap sequence is determined by $[\alpha_1 \ \alpha_2\ 1]$, where $\alpha_3 = 1$ is the last nonzero entry and $\sum_{i\geq 1}\alpha_i = 3$.  Note that the exponents of nonzero terms of $t_1$, $t_2$, and $t_3$ in every (non-decomposable) generator are proportional as $3:2:2$. Since each individual $t_i\not\in R$, there is only one linear generator in $R$, which implies that $\alpha_1 = 2$. Therefore, the gap sequence is $[2 \ 0\ 1]$. 

We want to describe the generators of $\fkm / \fkm^2 \pmod{\wfkm^4}$. Up to rescaling the parameters, we can assume that the linear generator is $0 \oplus t_2 \oplus t_3\in R$. Since $0\oplus t_2^2\oplus 0\not\in R$ and $0 \oplus 0\oplus t_3^2\not\in R$, the other quadratic generators can be written as $t_1^2\oplus 0 \oplus 0\in R$ and $t_1^3\oplus t_2^2 \oplus 0\in R$. It follows that 
$0 \oplus t_2^3 \oplus 0 \in R$, $0 \oplus 0 \oplus t_3^3\in R$, and hence $\alpha_3 = 1$ is confirmed.  

In summary, $\fkm / \fkm^2 \pmod{\wfkm^4}$ can be generated by  
\begin{align*}
x & = 0 \oplus t_2 \oplus t_3, \\
y & = t_1^2\oplus 0 \oplus 0, \\
z & = t_1^3\oplus t_2^2 \oplus 0. 
\end{align*}
 In terms of these generators, the defining equations of this singularity are 
$$ (xy, z^2 - y^3 - x^2z ). $$ 
This singularity is denoted by $T_8$ in Giusti's classification of simple space curve singularities; see~\cite[(7.22)]{L84Book} and \cite[Table 2]{S15}.  
 
 Since $x$, $y$, and $z$ have weights $3$, $4$, and $6$ under the $\bbG_m$-action, respectively, the standard projectivization of the corresponding Gorenstein curve $Y$ as well as its deformations $(C, p_1, p_2, p_3)\in \bbP \calH(2,1,1)$ can be embedded in the weighted projective space $\bbP(3,4,6,1)$, where the divisor at infinity is cut out by the hyperplane of weight $1$ as $ \frac{1}{2}p_1 + \frac{1}{3} p_2 + \frac{1}{3}p_3$. 

Note that $R/{\fkc}$ is spanned by $\langle 1, x, y, z, x^2\rangle $. It has dimension $5 = \delta$, which verifies that this singularity is Gorenstein. The germ of the dualizing bundle at $q$ is generated by $dt_1 / t_1^4 - dt_2 / t_2^3 + dt_3 / t_3^3 $.  

Finally, we compute the weights and characters for this singularity. In this case, we have $\ell = 6$, $a_1 = 2$, and $a_2 = a_3 = 3$. 
Since $\omega_C^{\log}\sim 3p_1 + 2p_2+2p_3$, it follows that 
$$ \calF^{\lambda} H^0(C, \omega_C^{\log}) = H^0(C,  (3 - \lceil\lambda / 2\rceil)p_1 + (2 - \lceil\lambda / 3\rceil) p_2 + (2 - \lceil\lambda / 3\rceil) p_3).$$  
Therefore, the filtration of $H^0(C, \omega_C^{\log})$ is given by 
\begin{align*} 
& \calF^{0}\colon H^0(C, 3p_1 + 2p_2 + 2p_3)  \cong \bbC^5 \supsetneq \calF^{1}\colon H^0(C, 2p_1 + p_2+p_3) \\
& = \calF^{2} \colon H^0(C, 2p_1 + p_2 + p_3) \cong \bbC^3  \supsetneq \calF^{3}\colon H^0(C, p_1 + p_2 + p_2)\cong \bbC^2   \\
 & \supsetneq \calF^{4} \colon H^0(C, p_1) =  \calF^{5}\colon  H^0(C, \calO)  = \calF^{6} \colon H^0(C, \calO) \cong   \bbC \supsetneq  \calF^{7}  = 0. 
\end{align*}
The nontrivial weights are $2$, $3$, and $6$. We thus obtain that 
$$ \chi_1^{\log}(2,1,1) = 11, \quad  \chi_2^{\log}(2,1,1) = 53, \quad \alpha(2,1,1) = \frac{37}{90}. $$

\subsection{The stratum $\bbP \calH(5,1)$}
\label{subsec:(5,1)}
 
In this case, we have $g = 4$,  $b = 2$, and $\delta = 5$. The gap sequence is determined by $[\alpha_1 \ \alpha_2\ \alpha_3 \ \alpha_4 \ \alpha_5\ 1]$, where $\alpha_6 = 1$ is the last nonzero entry. Moreover, $\sum_{i\geq 1}\alpha_i = 4$ and $\alpha_1, \alpha_2, \alpha_3 \neq 0$. Therefore, the only possibility is $[1\ 1\ 1\ 0\ 0\ 1]$. 
 
We want to describe the generators of $\fkm / \fkm^2 \pmod{\wfkm^7}$. Since $\alpha_1 = 1$ and the exponents of nonzero terms of $t_1$ and $t_2$ in every (non-decomposable) generator are proportional as $6:2 = 3:1$, up to rescaling the parameters, 
the linear generator can be written as $ t_1^3 \oplus t_2$. Moreover, $\alpha_4 = \alpha_5 = 0$ and $\alpha_6 = 1$ implies that $t_1^4\oplus 0 \in R$, $ t_1^5\oplus 0\in R$, but $t_1^6 \oplus 0 \not\in R$. 

In summary, $\fkm / \fkm^2 \pmod{\wfkm^7}$ can be generated by  
\begin{align*}
x & = t_1^3 \oplus t_2, \\
y & = t_1^4\oplus 0, \\
z & = t_1^5\oplus 0. 
\end{align*}
 In terms of these generators, the defining equations of this singularity are 
 $$ (xz - y^2, x^2y - z^2). $$
This singularity is denoted by $W_9$ in Giusti's classification of simple space curve singularities; see~\cite[(7.22)]{L84Book} and \cite[Table 2]{S15}.  
 
 Since $x$, $y$, and $z$ have weights $3$, $4$, and $5$ under the $\bbG_m$-action, respectively, the standard projectivization of the corresponding Gorenstein curve $Y$ as well as its deformations $(C, p_1, p_2)\in \bbP \calH(5,1)$ can be embedded in the weighted projective space $\bbP(3,4,5,1)$, where the divisor at infinity is cut out by the hyperplane of weight $1$ as $ p_1 +  \frac{1}{3}p_2$. 
 
Note that $R/{\fkc}$ is spanned by $\langle 1, x, y, z, x^2\rangle $. It has dimension $5 = \delta$, which verifies that this singularity is Gorenstein. The germ of the dualizing bundle at $q$ is generated by $dt_1 / t_1^7 - dt_2 / t_2^3$.  

Finally, we compute the weights and characters for this singularity. In this case, we have $\ell = 6$, $a_1 = 1$, and $a_2 = 3$. Since $\omega_C^{\log}\sim 6p_1 + 2p_2$, it follows that 
$$ \calF^{\lambda} H^0(C, \omega_C^{\log}) = H^0(C,  (6 - \lambda)p_1 + (2 - \lceil\lambda / 3\rceil) p_2).$$  
Therefore, the filtration of $H^0(C, \omega_C^{\log})$ is given by 
\begin{align*} 
& \calF^{0}\colon H^0(C, 6p_1 + 2p_2)  \cong \bbC^5 \supsetneq \calF^{1}\colon H^0(C, 5p_1 + p_2) \cong \bbC^4 \\
 & \supsetneq  \calF^{2} \colon H^0(C, 4p_1 + p_2) \cong \bbC^3  \supsetneq \calF^{3}\colon H^0(C, 3p_1 + p_2)\cong \bbC^2  \\ 
 & \supsetneq \calF^{4} \colon H^0(C, 2p_1) =  \calF^{5}\colon  H^0(C, p_1)  = \calF^{6} \colon H^0(C, \calO) \cong   \bbC \supsetneq  \calF^{7}  = 0. 
\end{align*}
The nontrivial weights are $1$, $2$, $3$, and $6$. We thus obtain that 
$$ \chi_1^{\log}(5,1) = 12, \quad  \chi_2^{\log}(5,1) = 60, \quad \alpha(5,1) = \frac{3}{8}. $$

\subsection{The strata $\bbP \calH(4, 2)^{\even}$ and $\bbP \calH(4, 2)^{\odd}$}
\label{subsec:(4,2)}
 
In this case, we have $g = 4$, $b=2$, and $\delta = 5$.  The gap sequence is determined by $[\alpha_1 \ \alpha_2\ \alpha_3 \ \alpha_4 \ 1]$, where $\alpha_5 = 1$ is the last nonzero entry and $\sum_{i\geq 1}\alpha_i = 4$. 

We want to describe the generators of $\fkm / \fkm^2 \pmod{\wfkm^6}$. Since the exponents of nonzero terms of $t_1$ and $t_2$ in every (non-decomposable) generator are proportional as $5:3$, we have $\alpha_1 = 2$, and hence $\alpha_2 = 1$ or $0$. 

\begin{enumerate}
\item[(0)] Suppose $\alpha_2 = 0$. Then $\alpha_4 = 0$ and $\alpha_3 = 1$. In this case, the gap sequence is $[2 \ 0\ 1 \ 0 \ 1]$. It follows that $t_1^2 \oplus 0 \in R$, $0 \oplus t_2^2 \in R$, and $a t_1^{5} \oplus t_2^3\in R$ for some $a\neq 0$. Up to rescaling $t_1$, we can assume that $a = 1$ and 
$t_1^{5} \oplus t_2^3\in R$. 

In summary, $\fkm / \fkm^2 \pmod{\wfkm^6}$ can be generated by 
\begin{align*}
x & = t_1^2 \oplus 0, \\
y & = 0 \oplus t_2^2, \\
z & = t_1^5\oplus t_2^3. 
\end{align*} 
 In terms of these generators, the defining equations of this singularity are 
 $$ (xy, x^5 + y^3 - z^2). $$
 This singularity is denoted by $T_9$ in Giusti's classification of simple space curve singularities; see~\cite[(7.22)]{L84Book} and \cite[Table 2]{S15}.  
 
 Since $x$, $y$, and $z$ have weights $6$, $10$, and $15$ under the $\bbG_m$-action, respectively, the standard projectivization of the corresponding Gorenstein curve $Y$ as well as its deformations $(C, p_1, p_2)$ can be embedded in the weighted projective space $\bbP(6,10,15,1)$, where the divisor at infinity is cut out by the hyperplane of weight $1$ as $ \frac{1}{3}p_1 + \frac{1}{5} p_2$.  

Note that $R/{\fkc}$ is spanned by $\langle 1, x, y, z, x^2 \rangle $. It has dimension $5 = \delta$, which verifies that this singularity is Gorenstein. The germ of the dualizing bundle at $q$ is generated by $dt_1 / t_1^6 - dt_2 / t_2^4$. 

Additionally, let $y_i$ be the point at infinity in the rational branch $Y_i$ for $i = 1, 2$. The spin parity of this case is determined by $\dim H^0(Y, 2y_1 + y_2)\pmod{2}$, where $t_1^{n_1} \oplus t_2^{n_2} \in H^0(Y, 2y_1 + y_2)$ if and only if $n_1 \leq 2$ and $n_2 \leq 1$. We see that $H^0(Y, 2y_1 + y_2)$ is spanned by $\langle 1, x \rangle$, which is $2$-dimensional. Therefore, this case corresponds to $\bbP \calH(4, 2)^{\even}$. 

Finally, we compute the weights and characters for this singularity. In this case, we have $\ell = 15$, $a_1 = 3$, and $a_2 = 5$. 
Since $\omega_C^{\log}\sim 5p_1 + 3p_2$, it follows that 
$$ \calF^{\lambda} H^0(C, \omega_C^{\log}) = H^0(C,  (5 - \lceil\lambda/3\rceil )p_1 + (3 - \lceil\lambda / 5\rceil) p_2).$$  
Therefore, the filtration of $H^0(C, \omega_C^{\log})$ is given by 
\begin{align*} 
& \calF^{0}\colon H^0(C, 5p_1 + 3p_2) \cong \bbC^5 \supsetneq \calF^{1}\colon H^0(C, 4p_1 + 2p_2)  =  \calF^{2}\colon H^0(C, 4p_1 + 2p_2) \\
& = \calF^{3}\colon H^0(C, 4p_1 + 2p_2) \cong \bbC^4  \supsetneq \calF^{4}\colon H^0(C, 3p_1 + 2p_2) = \calF^{5}\colon H^0(C, 3p_1 + 2p_2)  \cong \bbC^3 \\
& \supsetneq \calF^{6}\colon H^0(C, 3p_1 + p_2) = \calF^{7} \colon H^0(C, 2p_1 + p_2) = \calF^{8}\colon  H^0(C, 2p_1 + p_2)  \\ 
& = \calF^{9}\colon H^0(C, 2p_1 + p_2)  \cong \bbC^2  \supsetneq \calF^{10} \colon H^0(C, p_1 + p_2) = \calF^{11}\colon H^0(C, p_1) = \calF^{12} \colon H^0(C, p_1)  \\
& = \calF^{13} \colon H^0(C, \calO) =  \calF^{14} \colon H^0(C, \calO) = \calF^{15} \colon H^0(C, \calO) \cong \bbC\supsetneq  \calF^{16} = 0. 
\end{align*}
The nontrivial weights are $3$, $5$, $9$, and $15$. We thus obtain that 
$$ \chi_1^{\log}(4,2)^{\even} = 32, \quad  \chi_2^{\log}(4,2)^{\even} = 152 , \quad \alpha(4,2)^{\even} = \frac{14}{33}. $$

\item[(1)] Suppose $\alpha_2 = 1$. Then the gap sequence is $[2 \ 1\ 0 \ 0 \ 1]$. It follows that $t_1^3 \oplus 0\in R$, $t_1^5 \oplus t_2^3\in R$ (up to rescaling $t_1$), $t_1^4\oplus 0\in R$, and $0\oplus t_2^4\in R$. Since $t_1^5\oplus 0\not\in R$, the quadratic generator of $\fkm / \fkm^2$ 
must be $0\oplus t_2^2$. 

In summary, $\fkm / \fkm^2 \pmod{\wfkm^6}$ can be generated by  
\begin{align*}
x & = 0 \oplus t_2^2, \\
y & = t_1^3 \oplus 0, \\
z & = t_1^5\oplus t_2^3, \\
w & = t_1^4\oplus 0.  
\end{align*} 
 In terms of these generators, the defining equations of this singularity are 
 $$ (xy, xw, yz - w^2, zw - y^3, x^3 + y^2w - z^2). $$

 Since $x$, $y$, $z$, and $w$ have weights $10$, $9$, $15 $, and $12$ under the $\bbG_m$-action, respectively, the standard projectivization of the corresponding Gorenstein curve $Y$ as well as its deformations $(C, p_1, p_2)$ can be embedded in the weighted projective space $\bbP(10,9,15,12, 1)$, where the divisor at infinity is cut out by the hyperplane of weight $1$ as $ \frac{1}{3}p_1 + \frac{1}{5} p_2$. 

Note that $R/{\fkc}$ is spanned by $\langle 1, x, y, z, w \rangle $. It has dimension $5 = \delta$, which verifies that this singularity is Gorenstein. The germ of the dualizing bundle at $q$ is generated by $dt_1 / t_1^6 - dt_2 / t_2^4$. 

Additionally, $H^0(Y, 2y_1 + y_2)$ is spanned by the constant function $1\oplus 1$, which is $1$-dimensional. Therefore, this case corresponds to $\bbP \calH(4, 2)^{\odd}$. 

Finally, we compute the weights and characters for this singularity. In this case, we have $\ell = 15$, $a_1 = 3$, and $a_2 = 5$. 
Since $\omega_C^{\log}\sim 5p_1 + 3p_2$, it follows that 
$$ \calF^{\lambda} H^0(C, \omega_C^{\log}) = H^0(C,  (5 - \lceil\lambda/3\rceil )p_1 + (3 - \lceil\lambda / 5\rceil) p_2).$$  
Therefore, the filtration of $H^0(C, \omega_C^{\log})$ is given by 
\begin{align*} 
& \calF^{0}\colon H^0(C, 5p_1 + 3p_2) \cong \bbC^5 \supsetneq \calF^{1}\colon H^0(C, 4p_1 + 2p_2)  =  \calF^{2}\colon H^0(C, 4p_1 + 2p_2) \\
& = \calF^{3}\colon H^0(C, 4p_1 + 2p_2) \cong \bbC^4  \supsetneq \calF^{4}\colon H^0(C, 3p_1 + 2p_2) = \calF^{5}\colon H^0(C, 3p_1 + 2p_2)  \cong \bbC^3 \\
& \supsetneq \calF^{6}\colon H^0(C, 3p_1 + p_2) \cong \bbC^2 \supsetneq  \calF^{7} \colon H^0(C, 2p_1 + p_2) = \calF^{8}\colon  H^0(C, 2p_1 + p_2)  \\ 
& = \calF^{9}\colon H^0(C, 2p_1 + p_2) = \calF^{10} \colon H^0(C, p_1 + p_2) = \calF^{11}\colon H^0(C, p_1) = \calF^{12} \colon H^0(C, p_1)  \\
& = \calF^{13} \colon H^0(C, \calO) =  \calF^{14} \colon H^0(C, \calO) = \calF^{15} \colon H^0(C, \calO) \cong \bbC\supsetneq  \calF^{16} = 0. 
\end{align*}
The nontrivial weights are $3$, $5$, $6$, and $15$. We thus obtain that 
$$ \chi_1^{\log}(4,2)^{\odd} = 29, \quad  \chi_2^{\log}(4,2)^{\odd} = 149, \quad \alpha(4,2)^{\odd} = \frac{79}{228}. $$
 \end{enumerate}
 
\subsection{The stratum $\bbP \calH(3, 3)^{\nonhyp}$}
\label{subsec:(3,3)}
 
In this case, we have $g = 4$, $b=2$, and $\delta = 5$.  The gap sequence is determined by $[\alpha_1 \ \alpha_2\ \alpha_3 \ 1]$, where $\alpha_4 = 1$ is the last nonzero entry. Moreover, $\sum_{i\geq 1}\alpha_i = 4$, $\alpha_1 \neq 0$, and $\alpha_2\neq 0$. Therefore, the gap sequence has only two possibilities: $[1\ 1\ 1\ 1]$ or $[2\ 1\ 0\ 1]$. Note that $\alpha_1 = 1$ or $2$ corresponds to $\dim H^0(Y, y_1 + y_2) = 2$ or $1$, respectively. Since here we consider the nonhyperelliptic component, its gap sequence is $[2\ 1\ 0\ 1]$. 

We want to describe the generators of $\fkm / \fkm^2 \pmod{\wfkm^5}$. Note that the exponents of nonzero terms of $t_1$ and $t_2$ in every (non-decomposable) generator are proportional as $4:4 = 1:1$. Up to rescaling the parameters, we can assume that the quadratic generator is $t_1^2 \oplus t_2^2$. Moreover, $t_1^3\oplus 0\in R$ and $0\oplus t_2^3\in R$.  

In summary, $\fkm / \fkm^2 \pmod{\wfkm^5}$ can be generated by 
\begin{align*}
x & = t_1^2 \oplus t_2^2, \\
y & = t_1^3 \oplus 0, \\
z & = 0 \oplus t_2^3.  
\end{align*} 
 In terms of these generators, the defining equations of this singularity are 
 $$ (yz, x^3 - y^2 - z^2). $$
 This singularity is denoted by $Z_9$ in Giusti's classification of simple space curve singularities; see~\cite[(7.22)]{L84Book} and \cite[Table 2]{S15}.  
 
  Since $x$, $y$, and $z$ have weights $2$, $3$, and $3$ under the $\bbG_m$-action, respectively, the standard projectivization of the corresponding Gorenstein curve $Y$ as well as its deformations $(C, p_1, p_2)\in \bbP \calH(3,3)^{\nonhyp}$ can be embedded in the weighted projective space $\bbP(2,3,3,1)$, where the divisor at infinity is cut out by the hyperplane of weight $1$ as $ p_1 +  p_2$. 
 
Note that $R/{\fkc}$ is spanned by $\langle 1, x, y, z, x^2 \rangle $. It has dimension $5 = \delta$, which verifies that this singularity is Gorenstein. The germ of the dualizing bundle at $q$ is generated by $dt_1 / t_1^5 - dt_2 / t_2^5$. 

Finally, we compute the weights and characters for this singularity. In this case, we have $\ell = 4$ and $a_1 = a_2 = 1$.  Since $\omega_C^{\log}\sim 4p_1 + 4p_2$, it follows that 
$$ \calF^{\lambda} H^0(C, \omega_C^{\log}) = H^0(C,  (4 - \lambda )p_1 + (4 - \lambda) p_2).$$  
Therefore, the filtration of $H^0(C, \omega_C^{\log})$ is given by 
\begin{align*} 
& \calF^{0}\colon H^0(C, 4p_1 + 4p_2) \cong \bbC^5 \supsetneq \calF^{1}\colon H^0(C, 3p_1 + 3p_2) \cong \bbC^4  
 \supsetneq \calF^{2}\colon H^0(C, 2p_1 + 2p_2) \\
& = \bbC^2 \supsetneq \calF^{3}\colon H^0(C, p_1 + p_2) = \calF^{4}\colon H^0(C, \calO) \cong \bbC \supsetneq \calF^{5} = 0. 
\end{align*}
The nontrivial weights are $1$, $1$, $2$, and $4$. We thus obtain that 
$$ \chi_1^{\log}(3,3)^{\nonhyp} = 8, \quad  \chi_2^{\log}(3,3)^{\nonhyp} = 40, \quad \alpha(3,3)^{\nonhyp} = \frac{3}{8}. $$

\subsection{The stratum $\bbP \calH(3,2,1)$}
\label{subsec:(3,2,1)}

In this case, we have $g = 4$, $b=3$, and $\delta = 6$.  The gap sequence is determined by $[\alpha_1 \ \alpha_2\ \alpha_3 \ 1]$, where $\alpha_4 = 1$ is the last nonzero entry. Moreover, $\sum_{i\geq 1}\alpha_i = 4$ and $\alpha_1 \geq \alpha_2 > 0$. Therefore, the gap sequence is $[2\ 1\ 0\ 1]$. 

We want to describe the generators of $\fkm / \fkm^2 \pmod{\wfkm^5}$. The exponents of nonzero terms of $t_1$, $t_2$, and $t_3$ in every (non-decomposable) generator are proportional as $4:3:2$. Up to rescaling the parameters, we can assume that the linear generator is $t_1^2 \oplus 0 \oplus t_3$. Similarly, we can assume that $0 \oplus t_2^2 \oplus 0$ is a quadratic generator, and 
$t_1^3\oplus 0 \oplus 0$ and $t_1^4 \oplus t_2^3 \oplus 0$ are cubic generators.  

In summary, $\fkm / \fkm^2 \pmod{\wfkm^5}$ can be generated by 
\begin{align*}
x & = t_1^2 \oplus 0 \oplus t_3, \\
y & = 0 \oplus t_2^2 \oplus 0, \\
z & = t_1^3\oplus 0 \oplus 0, \\
w & = t_1^4 \oplus t_2^3 \oplus 0.
\end{align*} 
 In terms of these generators, the defining equations of this singularity are 
 $$ (xy, yz, xw - z^2, w^2 - y^3 - xz^2, z(w - x^2)). $$
 
Since $x$, $y$, $z$, and $w$ have weights $6$, $8$, $9$, and $12$ under the $\bbG_m$-action, respectively, the standard projectivization of the corresponding Gorenstein curve $Y$ as well as its deformations $(C, p_1, p_2, p_3)\in \bbP \calH(3,2,1)$ can be embedded in the weighted projective space $\bbP(6, 8, 9, 12, 1)$, where the divisor at infinity is cut out by the hyperplane of weight $1$ as $ \frac{1}{3}p_1 + \frac{1}{4} p_2 + \frac{1}{6}p_3$. 
 
Note that $R/{\fkc}$ is spanned by $\langle 1, x, y, z, w, x^2 \rangle $. It has dimension $6 = \delta$, which verifies that this singularity is Gorenstein. The germ of the dualizing bundle at $q$ is generated by $dt_1 / t_1^5 - dt_2 / t_2^4 - dt_3 / t_3^3$. 

Finally, we compute the weights and characters for this singularity. In this case, we have $\ell = 12$, $a_1 = 3$, $a_2 = 4$, and $a_3 = 6$. 
Since $\omega_C^{\log}\sim 4p_1 + 3p_2 + 2p_3$, it follows that 
$$ \calF^{\lambda} H^0(C, \omega_C^{\log}) = H^0(C,  (4 - \lceil\lambda/3\rceil )p_1 + (3 - \lceil\lambda/4\rceil) p_2 + (2 - \lceil\lambda/6\rceil )p_3).$$  
Therefore, the filtration of $H^0(C, \omega_C^{\log})$ is given by 
\begin{align*} 
& \calF^{0}\colon H^0(C, 4p_1 + 3p_2+2p_3) \cong \bbC^6 \supsetneq \calF^{1}\colon H^0(C, 3p_1 + 2p_2+p_3)  =  \calF^{2}\colon H^0(C, 3p_1 + 2p_2+p_3) \\
& = \calF^{3}\colon H^0(C, 3p_1 + 2p_2+p_3) \cong \bbC^4 \supsetneq \calF^{4}\colon H^0(C, 2p_1 + 2p_2+p_3)  \cong \bbC^3 \\
&  \supsetneq   \calF^{5}\colon H^0(C, 2p_1 + p_2+p_3)  = \calF^{6}\colon H^0(C, 2p_1 + p_2+p_3)\cong \bbC^2  \\
& \supsetneq \calF^{7}\colon H^0(C, p_1 + p_2)  =  \calF^{8}\colon H^0(C, p_1 + p_2) = \calF^{9}\colon H^0(C, p_1) \\
& =   \calF^{10}\colon H^0(C, \calO) =  \calF^{11}\colon H^0(C, \calO) =  \calF^{12}\colon H^0(C, \calO)\cong \bbC  \supsetneq \calF^{13} = 0. 
  \end{align*}
The nontrivial weights are $3$, $4$, $6$, and $12$. We thus obtain that 
$$ \chi_1^{\log}(3,2,1) = 25, \quad  \chi_2^{\log}(3,2,1) = 133, \quad \alpha(3,2,1) = \frac{59}{192}. $$

\subsection{The stratum $\bbP \calH(2,2,2)^{\odd}$}
\label{subsec:(2,2,2)-odd}
 
In this case, we have $g = 4$, $b=3$, and $\delta = 6$.  The gap sequence is determined by $[\alpha_1 \ \alpha_2 \ 1]$, where $\alpha_3 = 1$ is the last nonzero entry. Moreover, $\sum_{i\geq 1}\alpha_i = 4$, $\alpha_1 \neq 0$, and $\alpha_1 \geq \alpha_2$. 
The gap sequence has two possibilities: $[2\ 1\ 1]$ or $[3\ 0\ 1]$. Note that $\alpha_1 = 2$ or $3$ corresponds to $\dim H^0(Y, y_1 + y_2 + y_3) = 2$ or $1$, respectively. Since here we consider the odd spin component, its gap sequence is $[3\ 0\ 1]$. 

We want to describe the generators of $\fkm / \fkm^2\pmod{\wfkm^4}$. The exponents of nonzero terms of $t_1$, $t_2$, and $t_3$ in every (non-decomposable) generator are proportional as $1:1:1$. The quadratic generators are $t_1^2 \oplus 0 \oplus 0$,  $0 \oplus t_2^2 \oplus 0$, and $0 \oplus 0 \oplus t_3^2$. Up to rescaling the parameters, we can assume that the cubic generators are $t_1^3 \oplus t_2^3 \oplus 0$ and $t_1^3 \oplus 0 \oplus t_3^3$. 

In summary, $\fkm / \fkm^2 \pmod{\wfkm^4}$ can be generated by 
\begin{align*}
x & = t_1^2 \oplus 0 \oplus 0, \\
y & = 0 \oplus t_2^2 \oplus 0, \\
z & = 0 \oplus 0 \oplus t_3^2, \\
w & = t_1^3 \oplus t_2^3 \oplus 0, \\
u & =  t_1^3 \oplus 0 \oplus t_3^3. 
\end{align*} 
 In terms of these generators, the defining equations of this singularity are 
 $$ (xy, xz, yz, yu, zw, x(w-u), x^3 + y^3 - w^2, x^3 + z^3 - u^2, y^3 + z^3 - (w-u)^2). $$
 
Since $x$, $y$, $z$, $w$, and $u$ 
 have weights $2$, $2$, $2$, $3$, and $3$ under the $\bbG_m$-action, respectively, the standard projectivization of the corresponding Gorenstein curve $Y$ as well as its deformations $(C, p_1, p_2, p_3)\in \bbP \calH(2,2,2)^{\odd}$ can be embedded in the weighted projective space $\bbP(2,2,2,3,3,1)$, where the divisor at infinity is cut out by the hyperplane of weight $1$ as $ p_1 +  p_2 + p_3$. 

Note that $R/{\fkc}$ is spanned by $\langle 1, x, y, z, w, u \rangle $. It has dimension $6 = \delta$, which verifies that this singularity is Gorenstein. The germ of the dualizing bundle at $q$ is generated by $dt_1 / t_1^4 - dt_2 / t_2^4 - dt_3 / t_3^4$. 

Finally, we compute the weights and characters for this singularity. In this case, we have $\ell = 3$ and $a_1 = a_2 = a_3 = 1$. 
Since $\omega_C^{\log}\sim 3p_1 + 3p_2 + 3p_3$, it follows that 
$$ \calF^{\lambda} H^0(C, \omega_C^{\log}) = H^0(C,  (3 -\lambda )p_1 + (3 - \lambda) p_2 + (3 - \lambda)p_3).$$  
Therefore, the filtration of $H^0(C, \omega_C^{\log})$ is given by 
\begin{align*} 
& \calF^{0}\colon H^0(C, 3p_1 + 3p_2+3p_3) \cong \bbC^6 \supsetneq \calF^{1}\colon H^0(C, 2p_1 + 2p_2+2p_3) \cong \bbC^4\\ 
&\supsetneq \calF^{2}\colon H^0(C, p_1 + p_2+ p_3) =   \calF^{3}\colon H^0(C, \calO) \cong \bbC \supsetneq \calF^{4} = 0. 
  \end{align*}
The nontrivial weights are $1$, $1$, $1$, and $3$. We thus obtain that 
$$ \chi_1^{\log}(2,2,2)^{\odd} = 6, \quad  \chi_2^{\log}(2,2,2)^{\odd} = 33, \quad \alpha(2,2,2)^{\odd} = \frac{4}{15}. $$

\subsection{The stratum $\bbP \calH(6, 2)^{\odd}$}
\label{subsec:(6,2)}
 
 In this case, we have $g = 5$, $b=2$, and $\delta = 6$.  The gap sequence is determined by 
 $[\alpha_1 \ \alpha_2 \ \alpha_3\ \alpha_4\ \alpha_5\ \alpha_6\ 1]$, where $\alpha_7 = 1$ is the last nonzero entry and $\sum_{i\geq 1}\alpha_i = 5$.  Since here we consider the odd spin component, $\dim H^0(Y, 3y_1 + y_2) = 1$, and hence any $t_1^{n_1}\oplus t_2^{n_2}\in R$ with $n_1\leq 3$ and $n_2\leq 1$ must be the constant function. Note that the exponents of nonzero terms of $t_1$ and $t_2$ in every (non-decomposable) generator of $\fkm / \fkm^2$ are proportional as $7:3$. Hence, $t_1^{i}\oplus 0 \not\in R$ for $i = 1, 2, 3$ and $0\oplus t_2 \not\in R$. It implies that the gap sequence is $[2\ 1\ 1\ 0\ 0\ 0\ 1]$.    

We want to describe the generators of $\fkm / \fkm^2 \pmod{\wfkm^8}$. Up to rescaling the parameters, we can assume that 
$t_1^7 \oplus t_2^3\in R$ is the cubic generator. Since $t_1^5\oplus 0\in R$ but $t_1^7\oplus 0\not \in R$, the quadratic generator can only be $0\oplus t_2^2$.  
 
 In summary, $\fkm / \fkm^2\pmod{\wfkm^8}$ can be generated by 
 \begin{align*}
x & = 0 \oplus t_2^2, \\
y & = t_1^7 \oplus t_2^3, \\
z & = t_1^4 \oplus 0, \\
w & = t_1^5 \oplus 0, \\
u & =  t_1^6 \oplus 0. 
\end{align*} 
In terms of these generators, the defining equations of this singularity are 
$$ (xz, xw, xu, zu - w^2, yw - u^2, yz - wu, u^2 - z^3, yu - z^2w, y^2 - x^3 - zw^2). $$

Since $x$, $y$, $z$, $w$, and $u$ have weights $14$, $21$, $12$, $15$, and $18$ under the $\bbG_m$-action, respectively, the standard projectivization of the corresponding Gorenstein curve $Y$ as well as its deformations $(C, p_1, p_2)\in \bbP \calH(6,2)^{\odd}$ can be embedded in the weighted projective space $\bbP(14,21,12,15,18,1)$, where the divisor at infinity is cut out by the hyperplane of weight $1$ as $ \frac{1}{3}p_1 + \frac{1}{7} p_2$. 

Note that $R/{\fkc}$ is spanned by $\langle 1, x, y, z, w, u \rangle $. It has dimension $6 = \delta$, which verifies that this singularity is Gorenstein. The germ of the dualizing bundle at $q$ is generated by $dt_1 / t_1^8 - dt_2 / t_2^4$. 

Finally, we compute the weights and characters for this singularity. In this case, we have $\ell = 21$, $a_1 = 3$, and 
$a_2 = 7$. Since $\omega_C^{\log}\sim 7p_1 + 3p_2$, it follows that 
$$ \calF^{\lambda} H^0(C, \omega_C^{\log}) = H^0(C,  (7 - \lceil\lambda/3\rceil )p_1 + (3 - \lceil\lambda/7\rceil) p_2).$$  
Therefore, the filtration of $H^0(C, \omega_C^{\log})$ is given by 
\begin{align*} 
& \calF^{0}\colon H^0(C, 7p_1 + 3p_2) \cong \bbC^6 \supsetneq \calF^{1}\colon H^0(C, 6p_1 + 2p_2)  = \calF^{2}\colon H^0(C, 6p_1 + 2p_2)\\ 
& =  \calF^{3}\colon H^0(C, 6p_1 + 2p_2) \cong \bbC^5  \supsetneq \calF^{4}\colon H^0(C, 5p_1 + 2p_2) = \calF^{5}\colon H^0(C, 5p_1 + 2p_2) \\
& = \calF^{6}\colon H^0(C, 5p_1 + 2p_2) \cong \bbC^4  \supsetneq \calF^{7}\colon H^0(C, 4p_1 + 2p_2) \cong \bbC^3  \supsetneq \calF^{8}\colon H^0(C, 4p_1 + p_2) \\
& = \calF^{9}\colon H^0(C, 4p_1 + p_2)   \cong \bbC^2 \supsetneq \calF^{10}\colon H^0(C, 3p_1 + p_2) = \calF^{11}\colon H^0(C, 3p_1 + p_2)  \\
& = \cdots = \calF^{20} \colon H^0(C, \calO) = \calF^{21} \colon H^0(C, \calO) = \bbC \supsetneq \calF^{22} = 0. 
  \end{align*}
 The nontrivial weights are $3$, $6$, $7$, $9$, and $21$. We thus obtain that 
$$ \chi_1^{\log}(6,2)^{\odd} = 46, \quad  \chi_2^{\log}(6,2)^{\odd} = 256, \quad \alpha(6,2)^{\odd} = \frac{43}{171}. $$ 

\subsection{The stratum $\bbP \calH(5,3)$}
\label{subsec:(5,3)}
 
 In this case, we have $g = 5$, $b=2$, and $\delta = 6$.  The gap sequence is determined by 
 $[\alpha_1 \ \alpha_2 \ \alpha_3\ \alpha_4\ \alpha_5\ 1]$, where $\alpha_6 = 1$ is the last nonzero entry and $\sum_{i\geq 1}\alpha_i = 5$. Since $\alpha_6\neq 0$, we have $\alpha_1, \alpha_2, \alpha_3 \neq 0$. Moreover, the exponents of nonzero terms of $t_1$ and $t_2$ in every (non-decomposable) generator of $\fkm / \fkm^2$ are proportional as $6:4 = 3:2$. Therefore, $\alpha_1 = 2$, and hence the gap sequence is $[2\ 1\ 1\ 0\ 0\ 1]$.  
  
 We want to describe the generators of $\fkm / \fkm^2 \pmod{\wfkm^7}$. Up to rescaling the parameters, we can assume that 
 $t_1^6 \oplus t_2^4 \in R$, while $t_1^6 \oplus 0\not\in R$ and $0 \oplus t_2^4\not\in R$. Therefore, the quadratic generator is $t_1^3 \oplus t_2^2$ and the cubic generator is $0 \oplus t_2^3$. We also need $t_1^4\oplus 0\in R$ and $t_1^5\oplus 0\in R$. 
 
 In summary, $\fkm / \fkm^2 \pmod{\wfkm^7}$ can be generated by 
 \begin{align*}
x & = t_1^3 \oplus t_2^2, \\
y & = 0 \oplus t_2^3, \\
z & = t_1^4 \oplus 0, \\
w & = t_1^5 \oplus 0. 
\end{align*} 
In terms of these generators, the defining equations of this singularity are 
$$ ( yz, yw, xw - z^2, x^3 - y^2 - zw, x^2z - w^2).  $$ 

Since $x$, $y$, $z$, and $w$ have weights $6$, $9$, $8$, and $10$ under the $\bbG_m$-action, respectively, the standard projectivization of the corresponding Gorenstein curve $Y$ as well as its deformations $(C, p_1, p_2)\in \bbP \calH(5,3)$ can be embedded in the weighted projective space $\bbP(6,9,8,10,1)$, where the divisor at infinity is cut out by the hyperplane of weight $1$ as $ \frac{1}{2}p_1 + \frac{1}{3} p_2$. 

Note that $R/{\fkc}$ is spanned by $\langle 1, x, y, z, w, x^2 \rangle $. It has dimension $6 = \delta$, which verifies that this singularity is Gorenstein. The germ of the dualizing bundle at $q$ is generated by $dt_1 / t_1^7 - dt_2 / t_2^5$. 

Finally, we compute the weights and characters for this singularity. In this case, we have $\ell = 12$, $a_1 = 2$, and 
$a_2 = 3$. Since $\omega_C^{\log}\sim 6p_1 + 4p_2$, it follows that 
$$ \calF^{\lambda} H^0(C, \omega_C^{\log}) = H^0(C,  (6 - \lceil\lambda/2\rceil )p_1 + (4 - \lceil\lambda/3\rceil) p_2).$$  
Therefore, the filtration of $H^0(C, \omega_C^{\log})$ is given by 
\begin{align*} 
& \calF^{0}\colon H^0(C, 6p_1 + 4p_2) \cong \bbC^6 \supsetneq \calF^{1}\colon H^0(C, 5p_1 + 3p_2)  = \calF^{2}\colon H^0(C, 5p_1 + 3p_2) \\
& \cong \bbC^5 \supsetneq \calF^{3}\colon H^0(C, 4p_1 + 3p_2) \cong \bbC^4 \supsetneq \calF^{4}\colon H^0(C, 4p_1 + 2p_2) \cong \bbC^3 \\
& \supsetneq \calF^{5}\colon H^0(C, 3p_1 + 2p_2) = \calF^{6}\colon H^0(C, 3p_1 + 2p_2) \cong \bbC^2 \\
& \supsetneq \calF^{7}\colon H^0(C, 2p_1 + p_2) = \cdots = \calF^{12}\colon H^0(C, \calO) \cong \bbC \supsetneq \calF^{13} = 0. 
  \end{align*}
 The nontrivial weights are $2$, $3$, $4$, $6$, and $12$. We thus obtain that 
$$ \chi_1^{\log}(5,3) = 27, \quad  \chi_2^{\log}(5,3) = 147, \quad \alpha(5,3) = \frac{19}{68}. $$ 

\subsection{The {\em varying} stratum $\bbP \calH(1,1,1,1)$}
\label{subsec:(1,1,1,1)}

As mentioned in Remark~\ref{rem:cross-ratio}, it was already observed in \cite[Remark 2.1]{B24} that the singularity classes corresponding to the varying stratum $\bbP \calH(1,1,1,1)$ in genus three have one-dimensional moduli, determined by the cross-ratio of the four zeros in the line section of the canonical map. There is a special case when the underlying curve is hyperelliptic and the four zeros form two pairs of hyperelliptic conjugate points. In what follows we analyze the two cases separately. 

\begin{enumerate}[(1)]
\item First, the canonical model of a non-hyperelliptic smooth curve $C$ of genus three is a plane quartic, where a canonical divisor $p_1+p_2+p_3+p_4$ in $C$ is cut out by a line $L$. The resulting singularity $(Y,q)$ corresponds to a cone of four lines through $p_1, p_2, p_3$, and $p_4$ that meet at $q$. Without loss of generality, we can write down the defining equation of $(Y,q)$ as  
$$ xy (x-y)(x- cy), $$
where $c\neq 0,1,\infty$ is determined by the cross-ration of 
the four zeros in $L$. 

Since in this case $(Y,q)$ is a planar curve singularity, its miniversal deformation space can be described by 
\begin{align*}
& xy (x-y)(x- cy) + \lambda_1 x^2y^2 + \lambda_2 x^2y + \lambda_3 xy^2 \\
+\ & \lambda_4 x^2 + \lambda_5 xy + \lambda_6 y^2 + \lambda_7 x + \lambda_8 y + \lambda_9. 
\end{align*}
Since $x$ and $y$ both have weight $1$, the weights of the deformation parameters $\lambda_1, \ldots, \lambda_9$ are 
$0, 1, 1, 2, 2, 2, 3, 3, 4$, respectively. In particular, 
$\bbP{\rm Def}^{-}(Y,q)$ is isomorphic to the weighted projective space $\bbP(1, 1, 2, 2, 2, 3, 3, 4)$. Additionally, the deformation space $\bbP{\rm Def}^{0}(Y,q)$ of weight $0$ given by the direction of $\lambda_1$ parameterizes equisingular deformations that alter the cross-ratio of the four zeros in $L$. 

\item Next, suppose $(Y, q)$ arises from the degeneration of 
a hyperelliptic curve $C$. Up to reordering the zeros, we can assume that $p_1+p_2\sim p_3+p_4$ is the hyperelliptic $g^1_2$. Let $t_i$ be the parameter of the rational branch attached to each $p_i$. Up to rescaling the parameters, 
\begin{align*}
x & = t_1\oplus {\rm i} t_2\oplus 0 \oplus 0, \\
y & = 0\oplus 0 \oplus t_3\oplus {\rm i} t_4, \\
z & = t_1^2 \oplus 0 \oplus (-t_3^2)\oplus 0 
\end{align*}
generate the local ring of the singularity $(Y, q)$. In this case, the defining equations of $(Y, q)$ are 
$$ (xy, z(z - x^2 - y^2)). $$
Note that it is {\em not} a planar singularity. 

By using the software \textsf{Singular}, we can describe the miniversal deformation space of this hyperelliptic singularity as follows: 
\begin{align*}
&\begin{pmatrix}
    xy\\
    z^2-x^2z-y^2z\\
\end{pmatrix}+\lambda_1\begin{pmatrix}
    z\\
  0\\
\end{pmatrix}+\lambda_2\begin{pmatrix}
      y\\
   0\\
\end{pmatrix}+\lambda_3\begin{pmatrix}
      x\\
   0\\
\end{pmatrix}+\lambda_4\begin{pmatrix}
      1\\
   0\\
\end{pmatrix}\\
+\ & \lambda_5\begin{pmatrix}
      0\\
   z\\
\end{pmatrix}+\lambda_6\begin{pmatrix}
      0\\
   y^2\\
\end{pmatrix}+\lambda_7\begin{pmatrix}
      0\\
   y\\
\end{pmatrix}+\lambda_8\begin{pmatrix}
      0\\
   x\\
\end{pmatrix}+\lambda_9\begin{pmatrix}
      0\\
   1\\
\end{pmatrix}. 
\end{align*}
Since the weights of $x$, $y$, and $z$ are $1$, $1$, and $2$, respectively, the weights of the deformation parameters $\lambda_1, \ldots, \lambda_9$ are 
$0, 1, 1, 2, 2, 2, 3, 3, 4$, respectively. Interestingly, $\bbP{\rm Def}^{-}(Y,q)$ is still isomorphic to the weighted projective space $\bbP(1, 1, 2, 2, 2, 3, 3, 4)$. Additionally, the deformation space $\bbP{\rm Def}^{0}(Y,q)$ of weight $0$ given by the direction of $\lambda_1$ deforms this hyperelliptic singularity back to the preceding case of planar singularities by using the relation $xy + \lambda_1 z = 0$. 
\end{enumerate}

In summary, the union of $\bbP{\rm Def}^{-}(Y, q)$ for $(Y,q)\in S(1,1,1,1)$ forms a weighted projective bundle with fiber $\bbP(1, 1, 2, 2, 2, 3, 3, 4)$ over the one-dimensional base of cross-ratios, where the degenerate cross-ratios correspond to the three hyperelliptic loci given by $p_1+p_2\sim p_3+p_4$, $p_1+p_3\sim p_2+p_4$, or $p_1+p_4\sim p_2+p_3$. In particular, $\bbP\calH(1,1,1,1)$ is the complement of the discriminant locus of singular deformations in this weighted projective bundle. 

\subsection{The {\em varying} stratum $\bbP \calH(2,2,2)^{\even}$}
\label{subsec:(2,2,2)-even}

For completeness, here we consider another varying stratum $\bbP \calH(2,2,2)^{\even}$ in genus four. We will show that its singularity classes are {\em not} unique. 
 
In this case, we have $g = 4$, $b=3$, and $\delta = 6$.  Recall the discussion in Section~\ref{subsec:(2,2,2)-odd}, which implies that the gap sequence for this even spin component is $[2\ 1\ 1]$. 

We want to describe the generators of $\fkm / \fkm^2\pmod{\wfkm^4}$. The exponents of nonzero terms of $t_1$, $t_2$, and $t_3$ in every (non-decomposable) generator are proportional as $1:1:1$. Up to rescaling and reordering the parameters, we can assume that the linear generator is either $t_1 \oplus t_2 \oplus 0$ or $t_1 \oplus t_2 \oplus  t_3$. 
 
\begin{enumerate}[(1)]
\item First, suppose $t_1 \oplus t_2 \oplus  t_3 \in R$. Since each individual $t_i^3\not\in R$, up to reordering and rescaling the parameters, we can assume that a quadratic generator is $t_1^2 \oplus a t_2^2 \oplus 0$, where $a\neq 0, 1$. 

In this case, $\fkm / \fkm^2\pmod{\wfkm^4}$ can be generated by 
\begin{align*}
x & = t_1 \oplus t_2 \oplus t_3, \\
y & = t_1^2 \oplus (at_2^2) \oplus 0. 
\end{align*} 
 In terms of these generators, the defining equation of this singularity is 
 $$ y(y-x^2)(y- ax^2).   $$
 
 Since $x$ and $y$ have weights $1$ and $2$ under the $\bbG_m$-action, respectively, the standard projectivization of the corresponding Gorenstein curve $Y$ as well as its deformations $(C, p_1, p_2, p_3)$ can be embedded in the weighted projective space $\bbP(1,2,1)$, where the divisor at infinity is cut out by the hyperplane of weight $1$ as $ p_1 +  p_2 + p_3$. 

Note that $R/{\fkc}$ is spanned by $\langle 1, x,  y, x^2, xy, x^3 \rangle $. It has dimension $6 = \delta$, which verifies that this singularity is Gorenstein. The germ of the dualizing bundle at $q$ is generated by $ dt_1 / t_1^4 + dt_2 / t_2^4 - 2 (dt_3 / t_3^4)$. 

Moreover, note that $h^0(C, p_i + p_j) = 1$ for all $i$ and $j$, which implies that in this case the underlying curve is not hyperelliptic.  

\item Next, suppose $t_1\oplus t_2 \oplus 0 \in R$. Then $t_1^2 \oplus t_2^2 \oplus 0 \in R$. Since each individual $t_i^3\not\in R$, the other quadratic generator must be $0\oplus 0 \oplus t_3^2$. Up to rescaling $t_3$, the other cubic generator can be written as $ t_1^3 \oplus 0 \oplus t_3^3$. 

In this case, $\fkm / \fkm^2 \pmod{\wfkm^4}$ can be generated by 
\begin{align*}
x & = t_1 \oplus t_2 \oplus 0, \\
y & = 0 \oplus 0 \oplus t_3^2, \\
z & = t_1^3 \oplus 0 \oplus t_3^3. 
\end{align*} 
 In terms of these generators, the defining equations of this singularity are 
$$ (xy, z^2 - y^3 - x^3z).  $$

Since $x$, $y$, and $z$ have weights $1$, $2$, and $3$ under the $\bbG_m$-action, respectively, the standard projectivization of the corresponding Gorenstein curve $Y$ as well as its deformations $(C, p_1, p_2, p_3)$ can be embedded in the weighted projective space $\bbP(1,2,3,1)$, where the divisor at infinity is cut out by the hyperplane of weight $1$ as $ p_1 +  p_2 + p_3$. 

Note that $R/{\fkc}$ is spanned by $\langle 1, x, y, z, x^2, x^3 \rangle $. It has dimension $6 = \delta$, which verifies that this singularity is Gorenstein. The germ of the dualizing bundle at $q$ is generated by $dt_1 / t_1^4 - dt_2 / t_2^4 - dt_3 / t_3^4$. 

In this case, $H^0(Y, y_1 + y_2)$ is spanned by $\langle 1\oplus 1, x\rangle$ and $H^0(Y, 2y_3)$ is spanned by $\langle 1\oplus 1, y\rangle$. Therefore, this case corresponds to the hyperelliptic locus, where the three zeros consist of a pair of hyperelliptic conjugate points together with a Weierstrass point. 
\end{enumerate}

Finally, we observe that although the resulting singularities are different for the hyperelliptic and nonhyperelliptic loci in $\bbP\calH(2,2,2)^{\even}$, the filtrations of $H^0(C, \omega_C^{\log})$ are the same for both cases as follows: 
\begin{align*} 
& \calF^{0}\colon H^0(C, 3p_1 + 3p_2 + 3p_3) \cong \bbC^{6} \supsetneq \calF^{1}\colon H^0(C, 2p_1 + 2p_2 + 2p_3) \cong 
\bbC^4 \\
& \supsetneq \calF^{2}\colon H^0(C, p_1 + p_2 + p_3) \cong \bbC^2  \supsetneq  \calF^{3}\colon H^0(C, \calO) \cong 
\bbC  \supsetneq \calF^{4} = 0. 
\end{align*}
The nontrivial weights are $1$, $1$, $2$, and $3$. We thus obtain that 
$$ \chi_1^{\log}(2,2,2)^{\even} = 7, \quad  \chi_2^{\log}(2,2,2)^{\even} = 34, \quad \alpha(2,2,2)^{\even} = \frac{23}{57}. $$

\subsection{Nonvarying strata without obstructions}
\label{subsec:T2}

Since we have determined the defining equation of the unique singularity for each nonvarying stratum, by using the software \textsf{Singular}, we can compute the dimension of the obstruction space $T^2$ for each of them; see~\cite[Remark 1.32.1]{GLS} for a few examples of running such a program.  In particular, this leads to the following result. 

\begin{theorem}
\label{thm:T2=0}
Let $(Y, q)$ be the unique Gorenstein singularity with $\bbG_m$-action that corresponds to one of the following nonvarying strata: 
\begin{align*}
 \mu =\ & (3,1), (2,2)^{\odd}, (2,1,1), (5,1), (4,2)^{{\rm even}/{\rm odd}}, (3,3)^{\nonhyp}, (3,2,1), (5,3) \\
  & (2g-2)^{{\rm even}/{\rm odd}}\ {\rm for}\ g\leq 5, (2g-2)^{\hyp}\ {\rm and}\ (g-1, g-1)^{\hyp}\ {\rm for\ all}\ g.  
 \end{align*}
 Then the space of obstructions $T^2$ for $(Y, q)$ is $\{0\}$ and the corresponding stratum $\bbP\calH(\mu)$ can be realized as a hypersurface complement in a weighted projective space that can be described explicitly. Additionally, $\bbP\calH(\mu)$ is a rational variety and the rational Chow group of $\bbP\calH(\mu)$ is trivial in any positive degrees. 
  \end{theorem}

\begin{proof}
As mentioned before, we can verify that $\dim T^2 = 0$ for these strata by using \textsf{Singular}.  Alternatively, for each of the strata listed above, the defining ideal of the corresponding singularity requires at most four parameters, and Gorenstein singularities with embedding codimension at most three have no infinitesimal obstructions; see~\cite{H80}. Consequently, $\bbP\calH(\mu) = \bbP{\rm Def}^{-}_s(Y,q)$ is an open subset of the weighted projective space $\bbP(T^{1,-})$; see Section~\ref{subsec:def}. Note that these nonvarying strata are known to be affine; see~\cite[Theorem 1.1]{C24}. It implies that the complement of each $\bbP\calH(\mu)$ in its $\bbP(T^{1,-})$ is a hypersurface of pure codimension $1$; see~\cite[Proposition 1]{G69}, where the hypersurface is the discriminant locus parameterizing singular deformations in $\bbP{\rm Def}^{-}(Y,q)$. Regarding the explicit description of the corresponding weighted projective space, see Example~\ref{ex:wps} below for a demonstration. The remaining claims thus follow from the facts that weighted projective spaces are rational varieties and the rational Chow ring of a weighted projective space is generated by the class of a hypersurface; see~\cite[Lemma 4.7]{I22}. 
\end{proof}

\begin{example}
\label{ex:wps}
Consider the nonvarying stratum $\bbP\calH(2,1,1)$. The defining equations of the corresponding singularity $(Y,q)$ have been given in Section~\ref{subsec:(2,1,1)}.  
With the defining equations at hand, we can compute its miniversal deformation space by using the software \textsf{Singular} as follows: 
\begin{itemize}
        \item[] \texttt{LIB "sing.lib";}
        \item[] \texttt{ring R = 0,(x,y,z),dp;}
        \item[] \texttt{ideal I = xy, z2-y3-x2z;}
        \item[] \texttt{deform(I);}
    \end{itemize} 
\begin{align*}
\begin{pmatrix}
    xy\\
    z^2-y^3-x^2z\\
\end{pmatrix} & +\lambda_1\begin{pmatrix}
    0\\
    z\\
\end{pmatrix}+\lambda_2\begin{pmatrix}
      0\\
    y\\
\end{pmatrix}+\lambda_3\begin{pmatrix}
      0\\
    x\\
\end{pmatrix}+\lambda_4\begin{pmatrix}
      0\\
    1\\
\end{pmatrix} \\
& +\lambda_5\begin{pmatrix}
      z\\
    0\\
\end{pmatrix}+\lambda_6\begin{pmatrix}
      y\\
    0\\
\end{pmatrix}+\lambda_7\begin{pmatrix}
      x\\
    0\\
\end{pmatrix}+\lambda_8\begin{pmatrix}
      1\\
    0\\
\end{pmatrix}.
\end{align*}

The weights of $x$, $y$, and $z$ are $3$, $4$, and $6$, respectively. The weights of the deformation parameters  $(\lambda_1,\ldots,\lambda_8)$ can be read off from those of $(z,y^2,xz,y^3,xy/z,x,y,xy)$ which are $(6,8,9,12,1,3,4,7)$. Therefore, $\bbP\calH(2,1,1)$ can be compactified by the miniversal deformation space of $(Y, q)$ as the weighted projective space 
$${\rm Def}(Y, q)\cong \bbP(1,3,4,6,7,8,9,12).$$ 
\end{example} 

\begin{remark}
\label{rem:T2}
The structure of $\bbP{\rm Def}^{-}(Y,q)$ for monomial singularities with fixed semigroups has also been understood fairly well in low genus; see~\cite{N08, S23}. 

Comparing to the first list in Theorem~\ref{thm:nonvarying}, the two nonvarying strata with signature $(2,2,2)^{\odd}$ and $(6,2)^{\odd}$ are missing in Theorem~\ref{thm:T2=0}. Indeed, we have verified by using \textsf{Singular} that $T^2\neq 0$ for both of the corresponding singularities. Additionally, the defining ideal for each of the two singularities requires five parameters, and hence the embedding codimension is four, which is beyond the aforementioned range of having no obstructions. Although we expect that the same results on their rationality and Chow ring structure should still hold, some extra ideas are needed for addressing these two exceptional cases. 

Moreover, as mentioned in Section~\ref{subsec:ordinary}, we can add entries of $0$ into the signature $\mu$, i.e., zeros of order $0$, to generalize the setting of nonvarying strata. Nevertheless, increasing the number of zero entries can lead to more complicated deformations for the corresponding singularities. For example, if $(Y, q)$ is the elliptic $n$-fold point, then $T^2$ is nonzero for $n\geq 6$ and $\bbP{\rm Def}^{-}(Y, q)$ is {\em not} a weighted projective space for $n = 6$; see~\cite[Proposition 1.2]{S25} and~\cite[Theorem 1.5.7]{LP19}.  

Finally, we remark that for large genus, the birational type of a (nonhyperelliptic) stratum $\bbP\calH(\mu)$ is expected to be of {\em general type}; see~\cite{CCM24}. 
\end{remark}

\subsection{Nonvarying strata with ordinary marked points}
\label{subsec:ordinary}

We can generalize the definition of nonvarying strata by including zeros of order $0$, i.e., ordinary marked points. Let $\mu' = (m_1, \ldots, m_n, 0, \ldots, 0)$, where $\mu = (m_1, \ldots, m_n)$ is one of the above nonvarying signatures and 
the new parameters $s_1, \ldots, s_k$ correspond to the branches of the last $k$ ordinary points. Up to rescaling the parameters, we can assume that 
$$dt_1 / t_1^{m_1+2} - \cdots - ds_1 / s_1^2 - \cdots - ds_k / s_k^2$$ 
generates the dualizing bundle at the Gorenstein singularity. Then, we just need to add $t_1^{m_1+1} \oplus s_j$ into $R$ for $j = 1, \ldots, k$.  

For example, consider $\mu' = (3,1, 0, \ldots, 0)$. Based on the description in Section~\ref{subsec:(3,1)}, $\fkm / \fkm^2$ for the unique singularity class in $S(\mu')$ can be generated by  
\begin{align*}
x & = t_1^2 \oplus t_2 \oplus 0 \cdots \oplus 0, \\
y & = t_1^3\oplus 0 \oplus 0 \cdots \oplus 0,\ \\
z_j & = t_1^4 \oplus 0 \oplus \cdots \oplus 0 \oplus s_j \oplus 0 \oplus \cdots \oplus 0, \quad j = 1, \ldots, k. 
\end{align*}

In general, we have 
$$\chi_1^{\log}(\mu') = \chi_1^{\log}(\mu), \quad \chi_2^{\log}(\mu')  = \chi_2^{\log}(\mu) + k \ell,$$ 
where $\ell = \lcm (m_1+1, \ldots, m_n+1)$.

The above observation also leads to an application towards the $K(\pi,1)$-conjecture reviewed in Section~\ref{subsec:Kpi1}. The following is a refined formulation of Theorem~\ref{thm:US}. 

\begin{theorem}
\label{thm:ordinary}
Given $(Y, q)\in S(\mu)$, let $(Y', q')\in S(\mu,0)$ be the singularity class produced by adding an ordinary marked point away from the zeros $p_1,\ldots, p_n$ in the construction of $(Y,q)$. Then $\bbP{\rm Def}^{-}_s(Y',q')$ is the universal curve 
punctured at $p_1,\ldots, p_n$ over $\bbP{\rm Def}^{-}_s(Y,q)$. Moreover, if $\bbP{\rm Def}^{-}_s(Y,q)$ is $K(\pi,1)$, then $\bbP{\rm Def}^{-}_s(Y',q')$ is also $K(\pi,1)$. 

In particular, the $K(\pi,1)$-property holds for the discriminant complement in the miniversal deformation space of the singularities $U_7$, $U_8$, $U_9$, $S_{2g+2}$, and $S_{2g+3}$, which arise from the nonvarying strata for $\mu' = 
(4,0)^{\odd}$, $(3,1,0)$, $(6,0)^{\even}$, $(2g-2,0,0)^{\hyp}$, and $(g-1,g-1,0,0)^{\hyp}$, respectively. 
\end{theorem}

We refer to \cite[(7.22)]{L84Book} and \cite[Table 2]{S15} for the labeling of the above singularities and their defining ideals. 

\begin{proof}
The first claim follows from the preceding description about how to add the parameter $s$ into the parameterization for the new branch attached to the ordinary marked point. 

For the second claim, note that for $2g-2+n > 0$, a curve of genus $g$ punctured at $n$ points is $K(\pi,1)$. Then the desired claim follows from the standard exact homotopy sequence of fibration if the base is also $K(\pi, 1)$. 

Finally, note that the $K(\pi,1)$-property holds for $ADE$ singularities by \cite{B73} and \cite{D72}, which include those from the nonvarying strata by deleting the entries of $0$ from $\mu'$. Moreover, the concerned singularities are isolated complete intersections and their Milnor numbers are equal to the dimensions of the corresponding strata of differentials; see \cite[Theorem 7.2.22]{G20} and Remark~\ref{rem:smoothing}. Therefore, $\bbP{\rm Def}^{-}_s = \bbP{\rm Def}_s$ holds for them, which implies the last claim. 

For completeness, we list the defining ideals and parameterizations for these singularities as follows: 
\begin{align*}
U_7\colon & (xz - y^2, yz - x^3)  \\
& x = t^3 \oplus 0, \\
& y = t^4 \oplus 0, \\
& z = t^5 \oplus s.  
\end{align*}
\begin{align*}
U_8\colon & (xz - y^2, yx^2 - yz)  \\
& x = t_1^2 \oplus t_2 \oplus 0, \\
& y = t_1^3 \oplus 0 \oplus 0, \\
& z = t_1^4 \oplus 0 \oplus s.  
\end{align*}
\begin{align*}
U_9\colon & (xz - y^2, yz - x^4)  \\
& x = t^3 \oplus 0, \\
& y = t^5 \oplus 0, \\
& z = t^7 \oplus s.  
\end{align*}
\begin{align*}
S_{2g+2}\colon & (xz_1 - xz_2, z_1z_2 - x^{2g-1})  \\
& x = t^2 \oplus 0 \oplus 0, \\
& y = t^{2g+1} \oplus 0 \oplus 0, \\
& z_1 = t^{2g-1} \oplus s_1\oplus 0, \\
& z_2 = t^{2g-1} \oplus 0 \oplus s_2.  
\end{align*}
\begin{align*}
S_{2g+3}\colon & (xz_1 - xz_2, x^gz_2- z_1z_2)  \\
& x = t_1 \oplus t_2 \oplus 0 \oplus 0, \\
& y = t_1^{g+1}\oplus (-t_2^{g+1}) \oplus 0 \oplus 0, \\
& z_1 = t_1^{g} \oplus 0 \oplus s_1\oplus 0, \\
& z_2 = t_1^{g} \oplus 0 \oplus 0\oplus s_2.
\end{align*}
We remark that for $S_{2g+2}$ and $S_{2g+3}$, the original parameter $y$ in each case can be generated by the other parameters, so $y$ does not appear as a generator in the defining ideal. 
\end{proof}

\subsection{Nonvarying hyperelliptic loci}
\label{subsec:hyp-nonvarying}

Besides the hyperelliptic components of the strata, in general, one can define the locus of {\em hyperelliptic differentials} $\calH(\mu)^{\hyp}$ in $\calH(\mu)$ for 
$\mu = (2a_1, \ldots, 2a_m, b_1, b_1, \ldots, b_n, b_n)$, where differentials in $\calH(\mu)^{\hyp}$ are defined on  hyperelliptic curves such that each zero of order $2a_i$ is a Weierstrass point and each pair of zeros of order $b_j$ is hyperelliptic conjugate. 

\begin{theorem}
\label{thm:hyp-nonvarying}
The construction in Section~\ref{subsec:construction} produces isomorphic singularities for all hyperelliptic differentials in $\calH(\mu)^{\hyp}$, i.e., $S(\mu)^{\hyp}$ consists of a unique singularity class (up to relabeling the marked zeros). 
\end{theorem}

\begin{proof}
Let $t_i$ be the parameter of the rational branch associated to the zero of order $a_i$, and let $r_j$ and $s_j$ be the parameters associated to the pair of zeros of order $b_j$. 
By the condition (G4), up to rescaling the parameters, we can assume that all pairwise differences of $t_i^{2a_i+1}$, $r_j^{b_j+1}$, and $s_k^{b_k+1}$ are contained in $R$. Additionally, by the hyperelliptic assumption, we have $t_i^2 \in R$ and $r_j + \zeta_js_j \in R$ for all $i$ and $j$, where $\zeta_j = e^{\pi {\rm i}/(b_j+1)}$ such that 
$(r_j + \zeta_js_j)^{b_j+1} = r_j^{b_j+1} - s_{j}^{b_j+1}$. 

Note that all even powers of $t_i$ and all odd powers of $t_i$ with exponents $\geq 2a_i+3$ are contained in $R$. Moreover, all powers of $r_j$ and $s_j$ with exponents $\geq b_j+2$ as well as $r_j^c + (\zeta_js_j)^c$ for all exponents $c$ are contained in $R$. It follows that the sum $|\alpha|$ of the entries in the gap sequence $[\alpha_1, \alpha_2, \ldots]$ 
of the singularity is bounded above by 
$$ \sum_{i=1}^m a_i + \sum_{j=1}^n b_j + 1 = g,$$
where we used the fact that $\sum_{i=1}^m 2a_i + \sum_{j=1}^n 2b_j = 2g-2$ and the last one is from $\alpha_{\max\{2a_i+1, b_j+1\}} = 1$ by (G5). Since $|\alpha|$ has to be $g$, 
the above functions thus generate $R$, which determines the corresponding singularity uniquely. 
\end{proof}

\begin{remark}
\label{rem:(2,1,1)}
 Theorem~\ref{thm:hyp-nonvarying} implies that $\bbP\calH(\mu)^{\hyp}\subset \bbP{\rm Def}^{-}_s (Y, q)$, where $(Y,q)$ is the unique singularity class in $S(\mu)^{\hyp}$. However, $\bbP{\rm Def}^{-}_s (Y, q)$ may contain non-hyperelliptic differentials, i.e., the preceding containment can be strict in general. For example, $\bbP\calH(2,1,1)^{\hyp}$ is a proper subspace of $\bbP\calH(2,1,1)$, where the latter is a nonvarying stratum and equal to $\bbP{\rm Def}^{-}_s (Y, q)$ by Theorem~\ref{thm:nonvarying}. 

Additionally, up to a finite choice of ordering the marked zeros, the hyperelliptic locus $\bbP\calH(\mu)^{\hyp}$ can be identified with $\calM_{0,2g+2+n}$ by marking the images of the Weierstrass points and the conjugate pairs. Note that $\calM_{0,2g+2+n}$ can be interpreted as the configuration space of points in the real plane and hence is $K(\pi, 1)$. If we know $\bbP\calH(\mu)^{\hyp} = \bbP{\rm Def}^{-}_s (Y, q)$, then it would imply that the locus of smooth deformations in $\bbP{\rm Def}^{-} (Y, q)$, viewed as the complement of the discriminant locus of singular deformations, is also $K(\pi, 1)$. 
\end{remark}

It is thus a meaningful question to determine when $\bbP\calH(\mu)^{\hyp}= \bbP{\rm Def}^{-}_s (Y, q)$. We will provide a criterion as follows. Denote by $w_i$ the zero of order $2a_i$ and $u_j, v_j$ the 
pair of zeros of order $b_j$. Let $C$ be a smooth deformation contained in $\bbP{\rm Def}^{-}_s (Y, q)$. Recall that the divisor at infinity is the $\bbQ$-Cartier divisor 
$$C_{\infty} = \frac{1}{\ell}\left(\sum_{i=1}^m (2a_i+1)w_i + \sum_{j=1}^n (b_j+1)(u_i + v_i)\right),$$
where $\ell = \lcm \{2a_1+1, \ldots, 2a_m+1, b_1 + 1, \ldots, b_n + 1 \}$. 

\begin{proposition}
\label{prop:hyp=def}    
Given $\mu = (2a_1, \ldots, 2a_m, b_1, b_1, \ldots, b_n, b_n)$, suppose there exists a value of $k$ such that $\lfloor kC_{\infty}\rfloor$ as a linear combination of $w_i$ and $u_j, v_j$ has an even coefficient for each $w_i$, has the same coefficients for each pair $u_j, v_j$, and satisfies $0 < \deg \lfloor kC_{\infty}\rfloor < 2g-2$. Then $\bbP\calH(\mu)^{\hyp}= \bbP{\rm Def}^{-}_s (Y, q)$, where $(Y,q)$ is the unique singularity class in $S(\mu)^{\hyp}$ (modulo the ordering of the marked zeros).
\end{proposition}

\begin{proof}
By Clifford's theorem, if there is a divisor $D$ on $C$ such that $0 <\deg D < 2g-2$ and $h^0(C, D) = \frac{1}{2}\deg D + 1$, then $C$ is hyperelliptic, and moreover, $D$ is linearly equivalent to an effective linear combination of Weierstrass points and conjugacy pairs, where the coefficients of the Weierstrass points are even. 

As discussed in Section~\ref{sec:test}, note that 
$h^0(C, \lfloor \lambda C_{\infty}\rfloor)$ is the same for all deformations in $\bbP{\rm Def}^{-}_s (Y, q)$, including those hyperelliptic deformations in the subspace $\bbP\calH(\mu)^{\hyp}$. Therefore, taking $D = \lfloor kC_{\infty}\rfloor$ as in the assumption gives a divisor that satisfies the hyperelliptic Clifford's bound.  
\end{proof}

\begin{example}
Suppose $2a_i+1$ and $b_j+1$ are divisible by an integer $d \geq 4$ for all $i$ and $j$. Let $k = 2\ell / d$. Then 
$$k C_\infty = \sum_{i=1}^m \frac{2(2a_i+1)}{d} w_i + \sum_{j=1}^n \frac{2(b_j+1)}{d} (u_i + v_i) $$
is a Cartier divisor. Since $d\geq 4$, we have the even coefficient of $w_i$ equal to $2(2a_i+1)/d < 2a_i$, and hence $k C_\infty$ is strictly less effective than the canonical divisor $\sum_{i=1}^m 2a_i w_i + \sum_{j=1}^n b_j (u_i + v_i)$, which implies that $\deg kC_{\infty} < 2g-2$. Therefore, we can apply Proposition~\ref{prop:hyp=def} to conclude that in this case all deformations in $\bbP{\rm Def}^{-}_s (Y, q)$ are hyperelliptic. 
\end{example}

\begin{example}
Without loss of generality, we reorder the zeros such that $a_1 \geq \cdots \geq a_m $. Moreover, suppose $a_1\geq 2$ and $2a_1 + 1 > 2(2a_i+1)$ for all $a_i\neq a_1$. Let $k = 2\ell / (2a_1+1)$. Then we have 
$$\lfloor k C_{\infty}\rfloor = \sum_{i=1}^{m} \left\lfloor\frac{2(2a_i+1)}{2a_1+1}\right\rfloor w_i + \sum_{j=1}^n \left\lfloor\frac{2(b_j+1)}{2a_1+1}\right\rfloor (u_j + v_j). $$
Note that the coefficient of $w_i$ is $2$ if $a_i = a_1$ and is $0$ if $a_i\neq a_1$. Since $2a_1 > 2$, the divisor $\lfloor k C_{\infty}\rfloor$ is strictly less effective than the corresponding canonical divisor. Therefore, we can apply Proposition~\ref{prop:hyp=def} to conclude that in this case all deformations in $\bbP{\rm Def}^{-}_s (Y, q)$ are hyperelliptic. 
\end{example}

\section{Singularities with bounded $\alpha$-invariants}
\label{sec:3/8}

Suppose a singularity $(Y,q)\in S(\mu)$ has $\alpha$-invariant $\geq 3/8$ (without dangling branches); see \cite[Problem 3.15]{AFS16} for the significance of this critical value. The classification of such singularities is described in Theorem~\ref{thm:3/8}. The remaining section is devoted to prove this theorem. 

We first reduce the problem to an inequality about $\chi_1^{\log}$. By the relation~\eqref{eq:alpha}, the desired bound $\alpha\geq 3/8$ is equivalent to $\chi_2^{\log} / \chi_1^{\log} \leq 5$. Recall that for $\mu = (m_1, \ldots, m_n)$, we have defined $\ell = \lcm (m_1+1, \ldots, m_n+1)$, $a_i = \ell / (m_i + 1)$, and 
$l_{\lambda, i} = \lceil \lambda / a_i \rceil $ for $1\leq \lambda \leq \ell$. Denote by $D_{\lambda}$ the following divisor class: 
\begin{align*}
 D_{\lambda} & \coloneqq \omega_C^{\log} - \sum_{i=1}^{n} l_{\lambda, i}p_i  = \sum_{i=1}^n (m_i + 1 - l_{\lambda, i}) p_i.  
 \end{align*}
By Proposition~\ref{prop:weight} (iv), the inequality $\chi_2^{\log} / \chi_1^{\log} \leq 5$ is further equivalent to 
\begin{align}
\label{eq:3/8-inequality}
\chi_1^{\log} = \sum_{\lambda=1}^{\ell} h^0(C, D_\lambda) \geq \frac{1}{4}(2g-2+n) \ell. 
\end{align}
We will also use the following identities: 
 \begin{align}
 \label{eq:sum-ell}
\sum_{\lambda=1}^{\ell} \ell_{\lambda, i}  = \frac{1}{2} (m_i + 2)\ell \quad {\rm and}\quad \sum_{\lambda=1}^{\ell}\sum_{i=1}^n \ell_{\lambda, i}  = (g-1+n) \ell. 
\end{align}
 
 By Clifford's theorem, we have 
 $$h^0(C, D_{\lambda}) \leq 1 + \frac{1}{2}\deg D_{\lambda}. $$ 
 Combining it with~\eqref{eq:sum-ell} implies that 
\begin{align}
\label{eq:clifford}
 \chi_1^{\log} & \leq \sum_{\lambda=1}^{\ell} \left( 1 + \frac{1}{2}\left(2g-2+n - \sum_{i=1}^n l_{\lambda, i} \right)  \right) = \frac{1}{2}(g+1)\ell. 
 \end{align}
 Comparing~\eqref{eq:clifford} with \eqref{eq:3/8-inequality}, we conclude that $n \leq 4$. 
 
In particular, this already settles the case for genus one.  From now on, we assume that $g\geq 2$. 

Additionally, since adding an ordinary marked point (i.e., a zero of order $0$) does not change the value of $\chi_1^{\log}$, we can first classify the cases where every entry in $\mu$ is positive, and then add zero entries into $\mu$ up to the bound given by  \eqref{eq:3/8-inequality}. 

In what follows, we carry out this procedure case by case, depending on the possible values of the nonzero entries in $\mu$ as well as the underlying curves being hyperelliptic or not.  

\subsection{The case $n = 4$}
\label{subsec:n=4}
 
In this case, the inequality in \eqref{eq:clifford} is an equality. In particular, Clifford's bound 
$h^0(C, D_\lambda) \leq 1 + \frac{1}{2}\deg D_\lambda$ attains equality for all $\lambda = 1, \ldots, \ell$. 

Up to reordering the zeros, we assume that $m_1 \geq \cdots \geq m_n$. Then $a_1 \leq \cdots \leq a_n$, where $a_i = \ell / (m_i+1)$. Note that 
$D_\lambda = K_C$ if and only if  $1\leq \lambda \leq a_1$ and $D_\lambda = \calO_C$ if and only if $\ell - a_1 +1 \leq \lambda \leq \ell$. 

First, consider the case $a_1 = \ell - a_1$, i.e., $m_1 = 1$. In this case, $\mu = (1,1,1,1)$, $\chi_1^{\log}(1,1,1,1) = 4$, and \eqref{eq:3/8-inequality} attains equality.  

Next, consider the case $m_1 \geq 2$. Then Clifford's bound holds for some $D_\lambda$ where $0 < \deg D_{\lambda} < 2g-2$, and hence the underlying curve $C$ is hyperelliptic. If a zero $z_i$ is a Weierstrass point, the divisor class $D_{a_i+1}$ is of the form $K_C - z_i - E_i$, where $E_i$ is an effective divisor not containing $z_i$. Then $D_{a_i+1}$ does not achieve the maximum of Clifford's bound, leading to a contradiction. Therefore, the remaining possibility is that the four zeros form two hyperelliptic conjugacy pairs, where $m_1 = m_2 = a$ and $m_3 = m_4 = b$ (up to reordering). In this case, \eqref{eq:3/8-inequality} attains equality.  

\subsection{The hyperelliptic case with $n \leq 3$}
\label{subsec:hyp}

In this subsection, we assume that the underlying curve is hyperelliptic. 

For $n = 3$,  consider first the case when one of the zeros, say, $z_1$, is a Weierstrass point, and the other two zeros form a hyperelliptic conjugacy pair. Note that $l_{\lambda, 1} = \lceil \lambda / a_1\rceil$ varies from $1$ to $m_1+1$ for $\lambda = 1, \ldots, \ell$, where each value appears $a_1$ times and $m_1+1$ is odd.  Therefore, the number of even values of $l_{\lambda, 1}$, denoted by $N_1^+$,  is equal to $(\ell - a_1) /2$. It follows that 
$$\chi_1^{\log} = \frac{1}{2}(g+1)\ell - \frac{1}{2}N_1^+ $$
and the inequality in~\eqref{eq:3/8-inequality} is equivalent to $N_1^+ = (\ell - a_1) /2\leq \ell / 2$, which holds true. If there is an additional ordinary marked point, we need 
$\frac{1}{2}(g+1)\ell - \frac{1}{2}N_1^+ \geq \frac{1}{2}(g+1)\ell$. It implies that $a_1 = \ell = (m_1+1)a_1$, which is impossible since we have assumed that $m_1 > 0$. 

Next, consider the case when all the three zeros are Weierstrass points. In particular, the zero orders $m_i$ are even and $m_i + 1 \geq 3$ for all $i$. Following the above notation and argument, we have 
$$\chi_1^{\log} = \frac{1}{2}(g+1)\ell - \frac{1}{2}(N_1^+ + N_2^+ + N_3^+) $$
and the inequality in~\eqref{eq:3/8-inequality} is equivalent to $N_1^+ + N_2^+ + N_3^+= (3\ell - a_1-a_2-a_3) /2\leq \ell / 2$.
Since $a_i = \ell / (m_i + 1) \leq \ell / 3$, this is impossible. 

For $n = 2$, consider first the case where both zeros are Weierstrass points, which implies that $m_1$ and $m_2$ are both even. By the same reasoning as above, we have 
$$\chi_1^{\log} = \frac{1}{2}(g+1)\ell - \frac{1}{2}N_1^+ - \frac{1}{2} N_2^+$$ 
and  the inequality in~\eqref{eq:3/8-inequality} is equivalent to $N_1^+ + N_2^+ \leq \ell$, which holds true. If there is an additional ordinary marked point, we need $\frac{1}{2}(g+1)\ell - \frac{1}{2}N_1^+ - \frac{1}{2} N_2^+ \geq \frac{1}{4}(2g+1)\ell$. It implies that $N_1^+ + N_2^+\leq \ell / 2$, i.e., $1/(m_1+1) + 1/(m_2+1) \geq 1$, which is impossible since $m_1$ and $m_2$ are positive even numbers. 

Next, consider $\mu = (g-1, g-1)^{\hyp}$, where the two zeros form a hyperelliptic conjugacy pair. This was already computed in Section~\ref{subsec:A2g+1}, and~\eqref{eq:3/8-inequality} attains equality. Therefore, no ordinary marked points can be added further.  

For $n = 1$, the unique zero is a Weierstrass point. In this case, $\chi_1^{\log} = g^2$, and the inequality in \eqref{eq:3/8-inequality} holds for $\mu = (2g-2)^{\hyp}$, $(2g-2, 0)^{\hyp}$, and $(2g-2, 0, 0)^{\hyp}$, where the ordinary marked points can be any points in $C$. 
 
From now on, it suffices to consider the nonhyperelliptic case with $n\leq 3$. 
 
\subsection{The nonhyperelliptic case with $n = 3$}
\label{subsec:nonhyp-3}
 
We first make the following observation that will be used frequently for all nonhyperelliptic cases. 
If the underlying curve $C$ is nonhyperelliptic, then for any divisor class $D$ with $0 < \deg D < 2g-2$, we have 
\begin{align}
\label{eq:Clifford-nonhyp}
 h^0(C, D)\leq \begin{cases}  \frac{1}{2}(\deg D + 1) & {\rm if}\ \deg D\ {\rm is\ odd}, \\
  \frac{1}{2}\deg D & {\rm if}\ \deg D\ {\rm is\ even}.
 \end{cases}
 \end{align}
For convenience, we call~\eqref{eq:Clifford-nonhyp} the nonhyperelliptic Clifford's bound. 

Additionally, let $N^+$ be the number of even values of $\sum_{i=1}^n (l_{\lambda, i} - 1)$ for $\lambda = a_1+1, \ldots, \ell-a_1$. Here we exclude the ranges where $1\leq \lambda \leq a_1$ and $\ell - a_1+1\leq \lambda \leq \ell$, because the corresponding $D_\lambda$ is $K_C$ and $\calO_C$, respectively, which achieves the maximum of the hyperelliptic Clifford's bound. With that said, by using~\eqref{eq:Clifford-nonhyp}, we can strengthen the inequality~\eqref{eq:clifford} as 
\begin{align}
\label{eq:3/8-nonhyp}
 \chi_1^{\log} & \leq \sum_{\lambda = 1}^{\ell} \frac{1}{2} \left( 2g-1+n - \sum_{i=1}^n l_{\lambda, i}  \right) - \frac{1}{2}N^+ + a_1 \\  \nonumber 
 & = \frac{1}{2} (g\ell - N^+) + a_1. 
  \end{align}
 
For the case $n = 3$, we assume that $m_1\geq m_2 \geq m_3$ as before, and then $a_1 \leq a_2 \leq a_3$.  Comparing~\eqref{eq:3/8-nonhyp} with \eqref{eq:3/8-inequality}, it follows that $2N^+ + \ell \leq 4a_1$, and hence $m_1 \leq 3$. It is straightforward to check the remaining possibilities $\mu = (3,3,2)$, $(3,2,1)$, $(2,2,2)$, and $(2,1,1)$ explicitly. Only $\mu = (2,1,1)$ and $(2,2,2)^{\even}$ work, whose $\alpha$-invariants have been computed in Sections~\ref{subsec:(2,1,1)} and~\ref{subsec:(2,2,2)-even}, respectively. Additionally, if we add an ordinary marked point to them, by Section~\ref{subsec:ordinary}, it would make their $\alpha$-invariants less than $3/8$. 
   
\subsection{The nonhyperelliptic case with $n = 2$}
\label{subsec:nonhyp-2}
  
 In this case, arguing as above, the condition $N^+ \leq 2a_1$ needs to hold. 
 
 Consider first the case $(g-1, g-1)$, where $\ell = g$ and 
 $a_1 = a_2 = 1$. Then $N^+ = g-2\leq 2$ and $g\leq 4$. The cases are $\mu = (2,2)^{\odd}$ and $(3,3)^{\nonhyp}$, 
 whose $\alpha$-invariants have been computed in Sections~\ref{subsec:(2,2)} and~\ref{subsec:(3,3)}, respectively. In both cases, no ordinary marked points can be added to them. 
 
Next, consider the case $m_1 > m_2$. Write $m_1 + 1 = dk_1$ and $m_2 + 1 = d k_2$, where 
$\gcd (k_1, k_2) = 1$, $k_1 > k_2$, $a_1 = k_2 < a_2 = k_1$, $\ell = d k_1 k_2$, and $d(k_1+k_2) = 2g$ is even. 
In this case, $N^+$ is the number of even values of $\lceil \lambda / k_2 \rceil + \lceil \lambda / k_1 \rceil$ for $k_2 + 1\leq \lambda \leq dk_1 k_2 - k_2$. Since $k_1$ and $k_2$ are relatively prime, we can evaluate $N^+$ by using the following observation: 
\begin{itemize}
\item If $k_1$ and $k_2$ are both odd, then there are $ (k_1k_2 + 1) / 2$ even values of $\lceil \lambda / k_2 \rceil + \lceil \lambda / k_1 \rceil$ for $1\leq \lambda \leq k_1k_2$. 
\item If exactly one of $k_1$ and $k_2$ is odd, then there are $k_1k_2 / 2$ even values of $\lceil \lambda / k_2 \rceil + \lceil \lambda / k_1 \rceil$ for $1\leq \lambda \leq k_1k_2$. 
\end{itemize}
The first statement follows from the Chinese remainder theorem, and the other can be verified by paring $\lambda$ with $k_1 k_2 + 1 - \lambda$. As a consequence, we obtain that 
$$2k_2 \geq N^+ \geq \frac{d}{2} k_1 k_2 - 2k_2, $$
which implies that $dk_1\leq 8$ and $d\leq 4$ since $k_1 > k_2\geq 1$. 
 
If $d = 4$, then $k_2 = 1$ and $k_1 = 2$, which is the case $\mu = (7,3)$ and the nonhyperelliptic Clifford's bound in~\eqref{eq:Clifford-nonhyp} attains its maximum for all divisor classes $D_\lambda$ that appear in the range $0 < \deg D_{\lambda} < 2g-2$. This is equivalent to $h^0(C, 2p_1 + p_2) = 2$ (and $C$ is not hyperelliptic by the assumption). Using the method in Section~\ref{sec:nonvarying}, we find that $\fkm/\fkm^2 \pmod{\wfkm^9}$ is generated by $x = t_1^2 \oplus t_2$ and $y = t_1^5 \oplus 0$, and the defining equation of the singularity is $y (y^2 - x^5)$. Alternatively, the condition $h^0(C, 2p_1 + p_2) = 2$ implies that $h^0(C, 5p_1 + p_2) = 3$. We can take a plane sextic defined by $y^3 f(x, y) - x^5$, where $f(x,y)$ is a general polynomial of degree three. The line $y = 0$ cuts out a divisor of type $5p_1 + p_2$, where $p_1$ is at the origin and $p_2$ is at infinity.  Moreover, the projection from $p_1$ induces a $g^1_3$ with a fiber given by $2p_1 + p_2$. Then the normalization of this plane sextic gives an element in $\bbP\calH(7,3)$ which further satisfies $h^0(C, 2p_1 + p_2) = 2$. 
 
If $d = 3$, then $1\leq k_2 < k_1 \leq 2$ and $k_1 + k_2$ is even, which is impossible. 
 
If $d = 2$, then $1\leq k_2 < k_1 \leq 4$. First, consider the case when $k_1$ and $k_2$ are both odd. The only possible case is $k_1 = 3$ and $k_2 = 1$, i.e., $\mu = (5,1)$, whose $\alpha$-invariant is exactly $3/8$ as we have computed in Section~\ref{subsec:(5,1)}. In this case, no additional ordinary marked points can be added. 
 
Next, consider the case where exactly one of $k_1$ and $k_2$ is odd. The possible cases are $(k_1, k_2) = (4,3)$, $(4,1)$, $(3,2)$, and $(2,1)$. Consequently, $\mu = (7,5)$, $(7,1)$, $(5,3)$, and $(3,1)$. We have computed that $\mu = (5,3)$ does not satisfy the desired bound (Section~\ref{subsec:(5,3)}), while $\mu = (3,1)$ does (Section~\ref{subsec:(3,1)}), which also allows to add at most one ordinary marked point. 

For $\mu = (7,5)$, every $D_\lambda$ that appears needs to satisfy the nonhyperelliptic Clifford's bound. This implies that $h^0(C, 2p_1 + p_2) = 2$ and $h^0(C, 3p_1 + 2p_2) = 3$. However, the $g^2_5$ given by $|3p_1 + 2p_2|$ realizes $C$ as a plane quintic of arithmetic genus $6$, which is smaller than the geometric genus $7$, leading to a contradiction.  

For $\mu = (7,1)$, similarly, every $D_\lambda$ that appears needs to satisfy the nonhyperelliptic Clifford's bound. This implies that $h^0(C, 3p_1) = 2$ (which is equivalent to its canonical residual $h^0(C, 4p_1 + p_2) = 3$). One can realize this case by using the following model. Take an irreducible nodal plane quintic $C'$ such that  the tangent line to one branch of the node cuts out the divisor $4p_1 + p_2$ in the normalization $C$ of $C'$. By the adjunction formula, $K_{C'}\sim (K_{\bbP^2} + C')|_{C'} = \calO_{\bbP^2}(2)|_{C'} \sim \calO_{C'} (8p_1 + 2p_2)$. Since $K_{C'}$ pulls back to $K_{C}(p_1 + p_2)$, it follows that $K_{C} \sim \calO_{C}(7p_1 + p_2)$. We can also describe the singularity by using the method in Section~\ref{sec:nonvarying}, where $\fkm/\fkm^2 \pmod{\wfkm^9}$ is generated by 
$x = t_1^4 \oplus t_2$ and $y = t_1^3 \oplus 0$, and the defining equation of the singularity is $y (y^4 - x^3)$. In this case, the $\alpha$-invariant is exactly $3/8$, and no additional ordinary marked points can be added. 
 
 If $ d= 1$, then $k_1$ and $k_2$ are both odd. By the first case of the preceding observation, we need 
 $$2k_2 \geq N^{+} = \frac{1}{2} (k_1k_2 + 1) - 2k_2. $$
The possible cases are $\mu = (6, 4)$, $(6, 2)$, and $(4,2)$ (where the underlying curves are not hyperelliptic as in the assumption).  We have seen that $(4,2)^{\odd}$ (Section~\ref{subsec:(4,2)}) and $(6,2)^{\odd}$ (Section~\ref{subsec:(6,2)}) do not satisfy the bound, while $(4,2)^{\even}$ does but no ordinary marked points can be added.  The remaining cases are 
$(6,2)^{\even}$, $(6,4)^{\odd}$, and $(6,4)^{\even}$. 

For $\mu = (6,2)^{\even}$, by the even parity, $h^0(C, 3p_1 + p_2) = 2$ holds and attains the maximum of the nonhyperelliptic Clifford's bound. In this case, $a_1 - N^{+} / 2 = 1/2 < 1$. Therefore, every $D_\lambda$ that appears need to satisfy the nonhyperelliptic Clifford's bound. In particular, we have $h^0(C, 2p_1 + p_2) = 2$ and $h^0(4p_1 + p_2) = 3$. Then the $g^2_5$ given by $|4p_1 + p_2|$ realizes $C$ as a plane quintic with an ordinary cusp at the image of $p_1$ by the genus reason. However, the cuspidal tangent line cuts out $4p_1 + p_2$, where the multiplicity at $p_1$ is bigger than $3$, leading to a contradiction.  Alternatively, we can apply the method in Section~\ref{sec:nonvarying} to show the nonexistence of such singularities. 

For $\mu = (6,4)$, we have $a_1 - N^{+} / 2 = 1$. Therefore, there can be at most one $D_\lambda$ that does not satisfy the 
nonhyperelliptic Clifford's bound. In particular, since $2p_1 + p_2$ and $4p_1 + 2p_2$ both appear and are canonical residual to each other, we have $h^0(C, 2p_1 + p_2) = 2$ and $h^0(C, 4p_1 + 2p_2) = 3$.  

For $\mu = (6,4)^{\odd}$, by the odd parity, we further have $h^0(C, 3p_1 + 2p_2) = 3$. (Otherwise it would be $1$ and we would lose $2$ from the nonhyperelliptic Clifford's bound.) The $g^2_5$ given by $|3p_1 + 2p_2|$ is an embedding of $C$ into $\bbP^2$ as a plane quintic. The projection from $p_1$ induces a base point free $g^1_4$ with a fiber given by $2p_1 + 2p_2$. However, $h^0(2p_1 + p_2) = 2$ implies that $p_2$ should be a base point of this $g^1_4$, leading to a contradiction. 

For $\mu = (6,4)^{\even}$, we have $h^0(C, 3p_1 + 2p_2) = 2$, which is $1$ smaller than the nonhyperelliptic Clifford's bound. Since this divisor class appears multiple times in the expression of $D_\lambda$ for $16\leq \lambda \leq 20$, we can check that $\chi_1^{\log}$ violates the inequality~\eqref{eq:3/8-inequality}. 
    
\subsection{The nonhyperelliptic case with $n = 1$}
\label{subsec:n=1}

This is the case of monomial singularities for $\mu = (2g-2)$ as studied in Section~\ref{subsec:monomial}. Recall that 
$$G_p = \{ 1 = b_1 < \cdots < b_g = 2g-1\}$$ 
is the Weierstrass gap sequence of $p$, and $H_p = 
 \bbN \setminus G_p$ is the Weierstrass semigroup. We write 
 $$H_p = \{ 0 = k_1 < \cdots < k_g = 2g-2 < 2g < 2g+1 < 2g+2 < \cdots \}.$$ 
 In this case, we have seen that $ \chi_1^{\log} = \sum_{i=1}^g b_i$, and hence the inequality~\eqref{eq:3/8-inequality} is equivalent to 
$$ \sum_{i=1}^g b_i \geq \left\lceil \frac{1}{4}(2g-1)^2\right\rceil = g^2 - g +1, $$
which is further equivalent to 
\begin{align}
\label{eq:H-upper}
 \sum_{i=1}^g k_i \leq g^2 - 1. 
\end{align}
By the definition of $H_p$, for $i\geq 1$, we have $h^0(C, (k_i-1)p) = i-1$ and $h^0(C, k_i p) = i$. Since we have assumed that the underlying curve $C$ is not hyperelliptic, by Clifford's theorem, we have $h^0(C, k_i p) = i \leq (k_i+1)/2$, which implies that $k_i \geq 2i - 1$ for $ 2\leq i \leq g-1$. It follows that 
 \begin{align}
 \label{eq:H-lower}
  \sum_{i=1}^g k_i \geq \sum_{i=2}^{g-1} (2i-1) + (2g-2) = g^2 - 2. 
  \end{align}
Combining~\eqref{eq:H-upper} and~\eqref{eq:H-lower}, we conclude that 
$$\sum_{i=1}^g k_i = g^2-2\quad {\rm or} \quad g^2 - 1. $$

First, consider the case that $\sum_{i=1}^g k_i = g^2 - 2$. Then $H_p = \{ 0, 3, 5, \ldots, 2g-3, 2g-2, 2g, 2g+1, \ldots \}$. Since $H_p$ is a semigroup, it implies that $6 = 3+3 \geq 2g-2$, and hence $g\leq 4$. 

Next, consider the case that $\sum_{i=1}^g k_i = g^2 - 1$. Then exactly one of the elements of $H_p$ can jump by one compared to the previous case, and there are no other jumps.  

If $k_2 = 3\in H_p$ and if $6 \leq 2g-4$ (i.e., $g\geq 5$), then 
$H_p = \{0, 3, 6, 7, 9, \ldots, 2g-5, 2g-3, 2g-2, 2g, 2g+1, \ldots \}$. It implies that $10=3+7 \geq 2g-2$ and $g\leq 6$. In particular, for $g = 6$, the only possible case is $H_p = \{ 0, 3, 6, 7, 9, 10, 12, 13, \ldots\}$ generated by $3$ and $7$, where 
$G_p = \{1,2,4,5,8,11 \}$. In this case, the defining equation of the singularity is $y^3 - x^7$. We remark that it was verified in \cite[Table 1]{S23} that the locus of smooth pointed curves with this semigroup is non-empty. 

If $k_2 = 4\in H_p$, then $H_p = \{0, 4, 5, 7, \ldots, 2g-5, 2g-3, 2g-2, 2g, 2g+1, \ldots \}$. It implies that $8 = 4+4 \geq 2g-2$ 
and $g\leq 5$. 

Note that the minimal strata with $g\leq 5$ are all nonvarying. We can check their semigroups explicitly and conclude the desired claim.  

\section{Slopes of singularities and effective divisors}
\label{sec:slope}

Motivated by~\cite[(1.2) and Theorem 4.2]{AFS16}, given a singularity class $(Y, q) \in S(\mu)$ for a partition $\mu = (m_1, \ldots, m_n)$ of $2g-2$, we define the {\em slope} of $(Y, q)$ to be 
\begin{align}
\begin{split}
\label{eq:slope}
 s(Y, q) & \coloneqq \frac{13\chi_1(Y) - \chi_2(Y) }{\chi_1(Y)} \\
 & = 12 - \frac{(2g-2+n - \sum_{i=1}^n \frac{1}{m_i+1}) \ell}{\chi_1(Y)}, 
 \end{split}
\end{align} 
where the second equality follows from~\eqref{eq:log} and Proposition~\ref{prop:weight} (iv). We remark that the coefficient 
of $\ell$ in~\eqref{eq:slope} is the famous constant $\kappa_\mu$ (up to a scalar multiple of $1/12$) that appears in the relation 
of sums of Lyapunov exponents, area Siegel--Veech constants, and slopes; see~\cite[Theorem 1]{EKZ14} and \cite[Proposition 4.5]{CM12}. 

The above definition mimics the slope for a one-parameter family of stable curves $B$ in $\BM_g$, which is defined by $s(B) = \deg_{B} \delta / \deg_{B} \lambda_1$, as well as the slope for a divisor class $D = a\lambda_1 - b\delta$ in $\BM_g$, which is defined by $s(D) = a / b$; see~\cite{HM90}. 
Additionally, by Theorem~\ref{thm:filtration} (ii), smooth pointed curves $(C, p_1, \ldots, p_n) \in \bbP{\rm Def}^{-}_s(Y, q)$ have the same graded quotients of Rees filtrations, which determine $\chi_m(Y)$ completely for all $m$. Therefore, we also define the slope of the canonical divisor $\sum_{i=1}^n m_i p_i$ in $C$ to be $s(Y, q)$. 

\begin{example}[Slopes of the nonvarying strata]
\label{ex:nonvarying}
As shown in Section~\ref{sec:nonvarying}, each nonvarying stratum therein has a unique singularity class $(Y, q)$, whose $\bbP{\rm Def}^{-}_s(Y, q)$ covers the entire stratum. In this case, we can take $B$ to be a {\em Teichm\"uller curve} in the nonvarying stratum. The slope of $B$ thus yields the nonvarying slope for all Teichm\"uller curves in the stratum, which is further equal to $12 c_{\area}(B) / L(B)$, where $L$ is the sum of nonnegative Lyapunov exponents and $c_{\area}$ is the area Siegel--Veech constant; see~\cite[Proposition 4.5]{CM12}. 
\end{example}

\begin{example}[Slopes of the principal strata]
\label{ex:BN}
Consider $(Y, q)\in S(1^{2g-2})$ and $(C, p_1, \ldots, p_{2g-2})\in \bbP{\rm Def}^{-}_s(Y, q)\subset \bbP\calH(1^{2g-2})$. In this case, we have $\ell = 2$ and $a_i = 1$ for all $i$. In particular, 
$$ \chi_1 (Y) = h^0(C, K) + h^0(C, \calO) = g+1. $$
It follows that 
$$ s(Y,q) = 6 + \frac{12}{g+1}. $$
Notably, this equals the celebrated slope of the {\em Brill--Noether divisors}; see~\cite[Theorem (6.62)]{HM}.  Additionally, this slope value is independent of the choices of $(Y, q)\in S(1^{2g-2})$. 
\end{example}

\begin{example}[Generic slopes of the odd spin strata]
\label{ex:spin-odd}
Let $(Y, q)\in S(2^{g-1})^{\odd}$ and $(C, p_1, \ldots, p_{g-1})\in \bbP{\rm Def}^{-}_s(Y, q)\subset \bbP\calH(2^{g-1})^{\odd}$, chosen generically. In this case, we have $\ell = 3$ and $a_i = 1$ for all $i$. In particular, 
$$ \chi_1 (Y) = h^0(C, K) + h^0 (C, p_1 + \cdots + p_{g-1}) + h^0(C, \calO) = g+2, $$
where we used the generic assumption for $h^0 (C, p_1 + \cdots + p_{g-1}) = 1$. It follows that 
$$ s(Y,q) =  4 + \frac{24}{g+2}.  $$ 
In this case, we can apply the method in Section~\ref{subsec:(2,2,2)-odd} to describe the generators of $\mathfrak{m}/\mathfrak{m}^2 \pmod{\tilde{\mathfrak{m}}^4}$ for the singularity $(Y, q)$ as follows: 
\begin{align*}
x_i & = 0\oplus \cdots \oplus 0\oplus t_i^2\oplus 0 \oplus  \cdots \oplus 0, \quad i = 1, \ldots, g-1, \\
y_j & = 0 \oplus \cdots \oplus 0 \oplus t_j^3 \oplus 0 \oplus \cdots \oplus 0 \oplus t_{g-1}^3, \quad j = 1, \ldots, g-2. 
\end{align*}
We remark that if $(Y, q)\in S(2^{g-1})^{\odd}$  is from a special choice such that $h^0(C, p_1 + \cdots + p_{g-1}) \geq 3$, then the value of $s(Y, q)$ and the structure of the singularity can be different from the above generic case. 
\end{example}

\begin{example}[Generic slopes of the strata of Weierstrass points]
\label{ex:weierstrass}
Let $(Y, q)\in S(g, 1^{g-2})$ and $(C, p_1, \ldots, p_{g-1})\in \bbP{\rm Def}^{-}_s(Y, q)\subset \bbP\calH(g, 1^{g-2})$, chosen generically. In other words, the zero $p_1$ of order $g$ is a {\em normal} Weierstrass point, and $p_2+ \cdots + p_{g-1}$ is the canonical residual divisor of $gp_1$. 

We demonstrate for the case when $g$ is odd.  In this case, we have $\ell = g+1$, $a_1 = 1$, and $a_i = (g+1)/2$ for $2\leq i \leq g-1$. In particular, 
\begin{align*}
\chi_1 (Y) & = \sum_{\lambda = 1}^{\ell} h^0\left(C, \omega_C^{\log} - \sum_{i=1}^{g-1} \left\lceil \frac{\lambda}{a_i} \right\rceil p_i\right) \\
& =  \sum_{\lambda = 1}^{(g+1)/2} h^0(C, K - (\lambda-1) p_1) + \sum_{\lambda = (g+3)/2}^{g+1} h^0(C,  (g - \lambda+1) p_1) \\
& =  \sum_{\lambda = 1}^{(g+1)/2} (g - \lambda + 1) + \frac{1}{2}(g+1) = \frac{1}{8}(g+1)(3g+5), 
\end{align*}
where we used the generic assumption on $p_1$ in the last step. It follows that 
$$ s(Y,q) =  12 - \frac{4(5g+6)(g-1)}{(3g+5)(g+1)} \sim \frac{16}{3} + O(1/g).  $$ 

The case when $g$ is even can be analyzed similarly; we leave the details to the interested reader.  
\end{example}

It is a long standing question to determine the {\em lower bound} for the slopes of effective divisors in $\BM_g$. Harris--Morrison conjectured that the Brill--Noether divisors have the lowest slope given by $6 + 12/(g+1)$; see~\cite[Conjecture (6.63)]{HM}. Although starting from $ g = 10$ a series of counterexamples have been discovered, their slopes are still close to the Brill--Noether slope; see~\cite{T98, FP05, F09}. In particular, no effective divisors with slope $\leq 6$ have been found. Conversely, one can construct {\em moving curves} in $\BM_g$ to bound the slopes of effective divisors by using their intersection numbers, where a moving curve passes through a general element in $\BM_g$. Nevertheless, all known lower bounds constructed this way tend to zero as $g$ approaches infinity; see~\cite{F08, C11, P12}.  We refer to~\cite{CFM13} for a survey about this fantastic question.     
Under the same assumption as in \cite[Theorem 4.2]{AFS16}, we will illustrate an amusing approach to bound the slopes of effective divisors in $\BM_g$ by using the slopes of $(Y, q)\in S(\mu)$ defined above. The following statement is an elaboration of Theorem~\ref{thm:slope-main}. 

\begin{theorem}
\label{thm:slope}
Let $\mu = (m_1, \ldots, m_n)$ be a partition of $2g-2$ with $n \geq g-1$ such that $\mu \neq (2^{g-1})^{\even}$. Let $(C, p_1, \ldots, p_n)\in \bbP\calH(\mu)$ be a general element 
and let $(Y, q) \in S(\mu)$ be the corresponding singularity class. 
Suppose the discriminant locus of singular curves in ${\rm Def}^{-}(Y, q)$ is a $\bbQ$-Cartier divisor and the locus of worse-than-nodal deformations of $(Y, q)$ has codimension at least $2$. Then $s(Y, q)$ is a lower bound for the slopes of effective divisors in $\BM_g$.  
\end{theorem}

\begin{proof}
By assumption, we can find a complete curve $B\to {\rm Def}^{-}(Y, q)$ whose image passes through the general element $(C, p_1, \ldots, p_n)$ in $\bbP\calH(\mu)$ and does not meet the higher codimensional locus of worse-than-nodal deformations. By using~\cite[Theorem 4.2]{AFS16} and~\eqref{eq:slope}, we know that $s(B) = s(Y, q)$.  Additionally, since $n\geq g-1$, we have $\dim \bbP\calH(\mu)\geq \dim \calM_g$. In this case, $\bbP\calH(\mu)\to \calM_g$ 
is known to be dominant (except for $\mu = (2^{g-1})^{\even}$ where the image has codimension $1$); see~\cite[Proposition 4.1]{C10} and \cite[Theorem 5.7]{G18}. Therefore, $B\to \BM_g$ is a moving curve, which implies the desired claim.  
\end{proof}

We can also utilize Theorem~\ref{thm:slope} backwards. For example, if the slope bound is known to be lower than $6 + 12/(g+1)$ for certain $g$, then the assumption on the discriminant locus in Theorem~\ref{thm:slope} cannot hold for $\mu  = (1^{2g-2})$ as computed in Example~\ref{ex:BN}, which would thus reveal that the geometry of ${\rm Def}^{-}(Y, q)$ in this case can be worse than expected.  

In summary, it would be very useful to figure out when the assumption about the discriminant locus of ${\rm Def}^{-}(Y, q)$ can actually hold, e.g., for the cases in Example~\ref{ex:spin-odd} and Example~\ref{ex:weierstrass} where the corresponding strata dominate $\calM_g$.

\section{Loci of Weierstrass points with fixed semigroups}
\label{sec:Weierstrass}

In this section, we will prove the two statements of Theorem~\ref{thm:Weierstrass} through Theorem~\ref{thm:(2g-2)-tauto} and Theorem~\ref{thm:(2g-2)-affine}, respectively.  

We first review the setup and some known facts. Given a semigroup $H$ with gap sequence $G = \{b_1, \ldots, b_g\}$, recall that $\calM_{g,1}^{H}\subset \calM_{g,1}$ is the locally closed locus of smooth pointed curves $(C, p)$, where the semigroup of $p$ is $H$. Eisenbud--Harris observed that $\calM_{g,1}^{H}$ does not contain any complete curves; see \cite[c) 6)]{EH87W}. The following relation of cycle classes on $\calM_{g,1}^{H}$ is a crucial step in their argument. Since this relation was only sketched in~\cite{EH87W}, for completeness and our applications, here we provide a detailed proof.  

\begin{lemma}
\label{le:class}
Let $f\colon \calC \to \calM_{g,1}^{H}$ be the universal curve. Then, the total Chern class of the Hodge bundle on $\calM_{g,1}^{H}$ satisfies that 
$$ c(f_{*}\omega_{f}) = \prod_{i=1}^g (1 + b_i \psi), $$
where $\omega_f$ is the relative dualizing bundle and 
 $\psi$ is the cotangent line class associated to the marked point. In particular, the first Chern class of the Hodge bundle on $\calM_{g,1}^{H}$ satisfies that 
$$ \lambda_1 = \left(\sum_{i=1}^g b_i\right) \psi. $$
\end{lemma}

\begin{proof}
Let $P\subset \calC$ be the corresponding section of the marked point $p$. We have the short exact sequence:  
$$ 0 \to \omega_{f} (- b_i P) \to \omega_{f} (- (b_i-1) P) \to  \omega_{f} (- (b_i-1) P) \otimes \calO_P \to 0. $$
For every fiber curve $C$, we have 
$$h^0(C, K - b_i p) + 1 = h^0(C, K - (b_i-1) p)  = h^0(C, K - b_{i-1}p).$$
Therefore, applying $f_{*}$ to the above exact sequence, we conclude that 
\begin{align*}
f_{*} (\omega_{f} (- b_{i-1} P)) / f_{* } (\omega_{f} (- b_i P)) & \cong  f_{*} (\omega_{f} (- (b_i-1) P)) / f_{* } (\omega_{f} (- b_i P)) \\ 
& \cong  \calL^{b_i} 
\end{align*}
where $\calL$ is the bundle of the cotangent line associated to $p$ and $c_1(\calL) = \psi$. Using the following filtration of the Hodge bundle: 
$$ 0 \subset \cdots \subset f_{* } (\omega_{f} (- b_i P)) \subset f_{*} (\omega_{f} (- b_{i-1} P)) \subset \cdots \subset f_{*}\omega_{f}, $$
the desired claim thus follows.  
\end{proof}

\begin{remark}
Consider the ordinary gap sequence $G = \{1, \ldots, g\}$, where the corresponding locus in $\calM_{g,1}$ is the complement of the divisor $W$ of Weierstrass points. By \cite[Proposition 3.2]{CY25}, the Hodge bundle over $\calM_{g,1}\setminus W$ actually splits as a direct sum of line bundles $\oplus_{i=1}^g \calL^{i}$. 
\end{remark}

Now we focus on the case when $G$ is {\em symmetric}, i.e., $2g-1\in G$ and $\calM_{g,1}^H \subset \bbP\calH (2g-2)$. Recall that the {\em tautological ring} of $\calM_{g,1}$ is generated by $\psi$, the Hodge classes $\lambda_i$, and the Miller--Morita--Mumford classes $\kappa_j$. Similarly, we use the restrictions of these cycle classes to generate the tautological ring of $\calM_{g,1}^H$. 

\begin{theorem}
\label{thm:(2g-2)-tauto}
Let $G$ be a gap sequence of order $g$ such that $2g-1\in G$. Then the tautological ring of $\calM_{g,1}^H$ is trivial in any positive degree. 
\end{theorem}

\begin{proof}
The following divisor class relation holds on $\bbP\calH (2g-2)$; see \cite[Proposition 2.1]{ChenTauto}: 
\begin{align*}
12 \lambda_1 = \kappa_1 = (4g^2 - 4g) \psi. 
\end{align*}
Combining it with Lemma~\ref{le:class}, we conclude that 
$$ \left(3\sum_{i=1}^g b_i - g^2 + g\right) \psi = 0 $$
holds on $\calM_{g,1}^H$. Since $\sum_{i=1}^g b_i \geq g(g+1)/2$, the coefficient of $\psi$ in the above equality is nonzero, and hence $\psi = 0$ on $\calM_{g,1}^H$. The desired claim thus follows from the fact that the tautological ring of any strata of holomorphic differentials is generated by a single $\psi$-class; see \cite[Theorem 1.1 and Proposition 2.1]{ChenTauto}. 
\end{proof}

\begin{remark}
Tautological classes in $\calM_{g,1}^H$ can be {\em nontrivial} for general semigroups $H$. For example, consider the ordinary gap sequence $G = \{1, \ldots, g\}$. In this case, $\calM_{g,1}^H$ is the complement of the divisor $W$ of Weierstrass points 
 in $\calM_{g,1}$. Note that for $g\geq 3$, the Picard group of $\calM_{g,1}$ has rank $2$, generated by $\lambda_1$ and $\psi$. Moreover, $W$ is irreducible and has divisor class $W = \frac{1}{2}g(g+1)\psi - \lambda_1$ in $\calM_{g,1}$; see \cite[Theorem]{EH87M} and 
 \cite[(2.0.1)]{C89}, respectively. It follows that the tautological divisor class 
$\lambda_1 = \frac{1}{2}g(g+1)\psi$ is nontrivial on $\calM_{g,1}^{H} = \calM_{g,1}\setminus W$.   
\end{remark}

In general, $\calM_{g,1}^H$ might be reducible. We denote by $\calM_{g,1}^{H^{\circ}}$ an irreducible component of $\calM_{g,1}^H$. We also denote by 
$\BM_{g,1}^{H^\circ}$ its closure in the Deligne--Mumford compactification $\BM_{g,1}$. 
We say that $\calM_{g,1}^{H^\circ}$ is {\em minimal}, if every smooth pointed curve in the closure of $\calM_{g,1}^H$ has the same semigroup as $H$, i.e., no further degeneration to more special gap sequences. For example, for $H = \{0, 2, 4, \ldots, 2g-2, 2g, 2g+1, 2g+2, \ldots \}$ in genus $g$, the locus $\calM_{g,1}^H$ is minimal, parameterizing Weierstrass points in hyperelliptic curves. 

\begin{theorem}
\label{thm:(2g-2)-affine}
Let $G$ be a gap sequence of order $g$ such that $2g-1\in G$. Then an irreducible component $\calM_{g,1}^{H^{\circ}}$ of $\calM_{g,1}^H$ 
 is an affine variety if and only if $\BM_{g,1}^{H^{\circ}}\setminus \calM_{g,1}^{H^{\circ}}$ is pure codimension $1$. In particular, if $2g-1\in G$ and if $\calM_{g,1}^{H^{\circ}}$ is minimal, then $\calM_{g,1}^H$ is affine. 
\end{theorem}

\begin{proof}
First, if the complement of $\calM_{g,1}^{H^{\circ}}$ in $\BM_{g,1}^H$ is not pure codimension $1$, then it follows from the general result \cite[Proposition 1]{G69} that $\calM_{g,1}^H$ is not affine. 

Conversely, suppose the complement of $\calM_{g,1}^{H^{\circ}}$ in $\BM_{g,1}^H$ is pure codimension $1$. Note that $\kappa_1 + \psi$ is an ample divisor class on $\BM_{g,1}$; see~\cite[Theorem 2.2]{C93}. Additionally, by Theorem~\ref{thm:(2g-2)-tauto}, $\kappa_1 + \psi$ restricted to $\calM_{g,1}^H$ is trivial.  Therefore, the affinity of $\calM_{g,1}^{H^{\circ}}$ follows since an ample divisor is supported on the divisorial boundary complement; see~\cite[Proposition 3]{G69}. 

Moreover, note that $\BM_{g,1}$ has a normal crossing boundary $\Delta$ of pure codimension $1$ parameterizing nodal curves. Therefore, the intersection of $\Delta$ with $\BM_{g,1}^{H^{\circ}}$ remains to be pure codimension $1$. If $\calM_{g,1}^{H^\circ}$ is minimal, then $ \BM_{g,1}^{H^{\circ}} \setminus \calM_{g,1}^{H^{\circ}} = \Delta \cap \BM_{g,1}^{H^{\circ}}$ is pure codimension $1$, thus implying the last claim.  
\end{proof}

We caution the reader that $\BM_{g,1}^H$ might contain loci of smooth pointed curves $\calM_{g,1}^{H'}$, where $H'$ is a more special semigroup than $H$. Therefore, affinity of $\calM_{g,1}^H$ reduces to verifying the dimensions of such $\calM_{g,1}^{H'}$ and their containment relations. 

\begin{example}[The nonvarying minimal strata in low genus]
We have seen that any connected component of $\bbP\calH(2g-2)$ is nonvarying for $g\leq 5$, i.e., the corresponding semigroup $H$ is the same for all pointed curves in such a component. Then Theorem~\ref{thm:(2g-2)-affine} implies that it is an affine variety. This recovers the corresponding part in~\cite[Theorem 1.1]{C24}. 
\end{example}

\begin{example}[The varying stratum $\bbP\calH(10)^{\odd}$]
\label{ex:(10)-odd}
There are two possible semigroups for $(C, p)\in \bbP\calH(10)^{\odd}$ and both of their loci are irreducible; see~\cite[Table 1]{S23}. We follow the notation therein and also rename them as follows: 
\begin{align*}
H_1 \coloneqq N(6)_{22} & = \langle 6, 7, 8, 9, 10\rangle  \quad  {\rm and}  \quad \dim \calM_{6,1}^{H_1} = 11, \\
H_2 \coloneqq N(6)_{10} & = \langle 4, 5 \rangle \quad  {\rm and} \quad \dim \calM_{6,1}^{H_2} = 10. 
\end{align*}
Since $\calM_{6,1}^{H_2}$ is pure codimension $1$ in $\bbP\calH(10)^{\odd}$ and it is minimal, Theorem~\ref{thm:(2g-2)-affine} implies that both $\calM_{6,1}^{H_1}$ and $\calM_{6,1}^{H_2}$ are affine varieties. 
\end{example}

\begin{example}[The varying stratum $\bbP\calH(10)^{\even}$]
\label{ex:(10)-even}
There are three possible semigroups for $(C, p)\in \bbP\calH(10)^{\even}$; see~\cite[Table 1]{S23}: 
\begin{align*}
H_3 \coloneqq N(6)_{14} & = \langle 5, 7, 8, 9 \rangle  \quad  {\rm and}  \quad \dim \calM_{6,1}^{H_3} = 11, \\
H_4 \coloneqq N(6)_{9} & = \langle 4, 6, 9 \rangle \quad  {\rm and} \quad \dim \calM_{6,1}^{H_4} = 10, \\
 H_5 \coloneqq N(6)_{4} & = \langle 3, 7 \rangle  \quad {\rm and} \quad \dim \calM_{6,1}^{H_5} = 10. 
  \end{align*}
  Both $\calM_{6,1}^{H_4}$ and $\calM_{6,1}^{H_5}$ are pure codimension $1$ in $\bbP\calH(10)^{\even}$ and are minimal. It follows that $\calM_{6,1}^{H_3}$, $\calM_{6,1}^{H_4}$, and $\calM_{6,1}^{H_5}$ are all affine varieties. 
\end{example}

Finally, we remark that Arbarello--Mondello found that most of the strata in the {\em Weierstrass flags} on $\calM_g$ and $\calM_{g,1}$ are {\em not} affine by showing that the complements in their closures are not purely divisorial; see~\cite{AM12}. However, Weierstrass flags parameterize {\em $k$-gonal} curves, where the semigroups in each stratum of the flags can still vary. Therefore, it remains an interesting question to find a semigroup $H$ such that the boundary of $\calM_{g,1}^H$ is not pure codimension $1$.

\bibliographystyle{alpha}
\bibliography{biblio}

\end{document}